\documentclass[a4paper, 1sided,11pt,reqno]{amsart}
\usepackage{inputenc}
\usepackage{amsmath,amsthm,amsfonts,amssymb, mathtools}
\usepackage{mathabx}
\usepackage{fouridx}
\usepackage{bbm, dsfont}
\usepackage{enumerate}
\usepackage{mathrsfs, bera, stmaryrd}
\usepackage{mathrsfs, ccfonts}
\usepackage{amsbsy}
\usepackage{ dsfont}
\usepackage[pdftex,dvipsnames]{xcolor}
\usepackage[colorlinks,citecolor=blue,urlcolor=blue]{hyperref}
\usepackage[colorinlistoftodos,prependcaption,textsize=small]{todonotes}
\usepackage{xargs}  
\newcommandx{\unsure}[2][1=]{\todo[linecolor=red,backgroundcolor=red!25,bordercolor=red,#1]{#2}}
\newcommandx{\change}[2][1=]{\todo[linecolor=blue,backgroundcolor=blue!25,bordercolor=blue,#1]{#2}}
\newcommandx{\info}[2][1=]{\todo[linecolor=OliveGreen,backgroundcolor=OliveGreen!25,bordercolor=OliveGreen,#1]{#2}}
\newcommandx{\improvement}[2][1=]{\todo[linecolor=Plum,backgroundcolor=Plum!25,bordercolor=Plum,#1]{#2}}
\newcommandx{\thiswillnotshow}[2][1=]{\todo[disable,#1]{#2}}
\usepackage[english]{babel}
\usepackage{amsmath}
\theoremstyle{remark}
\newcommand{\re}[1]{(\ref{#1})}

\usepackage[left=2.5cm,right=2.5cm,top=3cm,bottom=3cm]{geometry}
\newcommand{\Sum }    {\displaystyle \sum}

\newcommand{\valabs}[1]{\left\vert#1\right\vert}
\newcommand{\abs}[1]{\left\Vert#1\right\Vert}

\newcommand{\bH}{\mathbf{H}}
\newcommand{\bV}{\mathbf{V}}
\newcommand{\ve}{\mathbf{V}}

\newcommand{\MO}{\mathcal{O}}

\selectfont

\theoremstyle{plain}
\newtheorem{theorem}{\textbf{Theorem}}[section]
\newtheorem{lemma}[theorem]{\textbf{Lemma}}
\newtheorem{claim}[theorem]{\textbf{Claim}}

\newtheorem{proposition}[theorem]{\textbf{Proposition}}

\theoremstyle{definition}
\newtheorem{assumption}[theorem]{\textbf{Assumption}}

\newtheorem{remark}[theorem]{\textbf{Remark}}

\newtheorem{definition}[theorem]{\textbf{Definition}}

\numberwithin{equation}{section}
\numberwithin{figure}{section}
\newenvironment{prev}[1][]{\noindent\textbf{#1.} }{\ \rule{0.6em}{0.6em}}

\newcommand{\del}[1]{}
\newcommand{\lve}{\lVert}
\newcommand{\rve}{\rVert}
\newcommand{\h}{\mathrm{H}}
\newcommand{\elm}{\mathrm{L}}
\newcommand{\Pl}{\mathrm{P}_L}
\newcommand{\sobm}{\mathrm{W}}
\newcommand{\sobb}{\mathbb{W}}
\newcommand{\Hb}{\mathbb{H}}
\newcommand{\el}{\mathbb{L}}

\newcommand{\bh}{\mathbf{H}}

\newcommand{\be}{\mathbb{E}}
\newcommand{\bu}{\mathbf{u}}

\newcommand{\bv}{\mathbf{v}}
\newcommand{\bw}{\mathbf{w}}
\newcommand{\bo}{\mathcal{O}}
\newcommand{\bk}{\mathcal{R}}
\newcommand{\br}{\mathcal{K}}

\newcommand{\bec}{\mathcal{E}}

\newcommand{\eps}{\varepsilon}

\title{ On the Stochastic chemotaxis-Navier-Stokes model in 2D/3D}

\begin{document}
	\selectlanguage{english}
	\title[Semi-discrete scheme of 3D stochastic  chemotaxis-fluid model driven by transport noise]{Time discretization of a semi-discrete scheme for  3D Chemotaxis-Navier-stokes system driven by transport noise}
\author{Erika Hausenblas$^{a}$, Boris Jidjou Moghomye$^{a}$  and Paul Andr\'e Razafimandimby$^{b}$}
\dedicatory{\vspace{-10pt}\normalsize{$^{a}$  Department of Mathematics and Information Technology, Montanuniversitaet Leoben, Leoben Franz Josef Strasse 18, 8700 Leoben, Austria \\$^{b}$ School of Mathematical Sciences, Dublin City University, Collins Avenue,  Dublin $9$, Ireland 
%E-mail addresses: erika.hausenblas@unileoben.ac.at (E. Hausenblas), \\ boris.jidjou-moghomye@unileoben.ac.at (B. Jidjou Moghomye), \\ paul.razafimandimby@dcu.ie (P. A. Razafimandimby)\\
%Corresponding author: B. Jidjou Moghomye
}}
%	\author{Erika Hausenblas, Boris Jidjou Moghomye and Paul Andr\'e Razafimandimby}
%	\address{$^1$ Department of Mathematics and Information Technology, Montanuniversitaet Leoben, Leoben Franz Josef Strasse 18, 8700 Leoben, Austria}
%	\email{erika.hausenblas@unileoben.ac.at}
%	\address{$^1$ Department of Mathematics and Information Technology, Montanuniversitaet Leoben, Leoben Franz Josef Strasse 18, 8700 Leoben, Austria}
%	\email{boris.jidjou-moghomye@unileoben.ac.at}
%	\address{$^2$ School of Mathematical Sciences, Dublin City University, Collins Avenue,  Dublin $9$, Ireland }
%	\email{paul.razafimandimby@dcu.ie}
	%\dedicatory{\vspace{-10pt}\absalsize{$^{a}$  \\$^{b}$  }}

	
	\begin{abstract}
 This work is devoted to the convergence of a time-discrete numerical scheme of a semi-discretization model arising from biology, consisting of  a chemotaxis equation coupled with a Galerkin approximation of Navier-Stokes system  driven by transport noise  in a three-dimensional  bounded and convex domain. We propose a semi-implicit Euler numerical scheme approximating the infinite dimensional model, for which we study the well-posedness and derive some uniform  estimates for the discrete variables. %We present a convergence analysis proving the existence of probabilistic solutions of the semi-discretization model. Our result might  be interesting not only for the numerical simulation  but also for the construction of martingale solutions of the stochastic Chemotaxis-Navier-Stokes system.  
	\end{abstract}
	
	\date{\today}
	\maketitle

	\section{Introduction}

%%%%%%%%%%%%%%

This paper focuses on the time discretization  of a semi-discretization of  the following system of stochastic partial differential equations describing the evolution of biological organisms which react on a chemo-attractant in a fluid environment (see for instance \cite{Zhai,Zhang1}):
%%%%%%%%%%%%%%%%%%%%%%%%%%%%%%%%%%%%%%%%%%%%%%%%%%	
%%%%%%%%%%%%%%%%%%%%%%%%%%%%%%%%%%%%%%%%%%%%%%%%%%%%%%%%%%%%%
\begin{equation}\label{1.1}
	\begin{cases}
		d \bu +\left[(\bu\cdot \nabla) \bu +\nabla P-\vartheta\Delta \bu
		\right]dt=n\nabla\Phi dt+  \alpha\mathbf{f}(\bu) \circ d W_t,\ \text{ in }(0,T)\times \bo,\\
		dc + \bu\cdot \nabla cdt=\left[\mu\Delta c -nc\right]dt+ \gamma\Sum_{k=1}^3\mathbf{g}_k\cdot\nabla c \circ d \beta^k_t,\ \text{ in }(0,T)\times \bo,\\
		d n +\bu\cdot \nabla n dt= \left[\delta\Delta n-\nabla \cdot (n \nabla c)\right]dt, \ \text{ in }(0,T)\times \bo,\\
		\nabla \cdot \bu=0,\ \text{ in }(0,T)\times \bo,\\
		\frac{\partial n}{\partial \nu}=\frac{\partial c}{\partial \nu}=0, \ \bu=0,\qquad\text{on}\qquad (0,T)\times\partial \mathcal{O},\\
		n(0)=n_0,\quad c(0)=c_0,\quad \bu(0)=\bu_0,\qquad\text{in}\qquad(0,T)\times\mathcal{O},
	\end{cases}
\end{equation}
where $T>0$ is the final observation time, $\bo\subset\mathbb{R}^3$, is a bounded domain and  $\nu$ is the  outward normal on the boundary  $\partial\bo$. The process $W_t$ is given by $W_t=\sum_{k=1}^3e_kW_k(t)$, $t\in [0,T]$, while the  processes $\{\beta^k,W^k\}_{k=1,2,3}$  are   given family of i.i.d. standard real-valued  Brownian motions over a fixed  complete filtered probability space $(\Omega, \mathcal{F}, \mathbb{F}=(\mathcal{F}_t)_{t\in[0,T]}, \mathbb{P})$ with the filtration $\mathbb{F}$ satisfying the usual conditions. Here,  $\{e_k: k=1,2,3\}$ is the canonical basis of $\mathbb{R}^3$. The noises coefficients $\mathbf{g}_k:\mathbb{R}^3\to \mathbb{R}^3$, $k=1,2,3$ and $\mathbf{f}$  are given functions, while the non-negative constants $\alpha$ and $\gamma$ represent  the noises intensities, $\vartheta$ the viscosity of the fluid, $\mu$ the  diffusion of the chemo-attrandant and $\delta$ the  diffusion of the biological organism.   The symbol $\circ$ means that the stochastic differential is understood in the Stratonovich sense. The unknowns are  $\bu$, $P$ representing
the velocity and the pressure of the fluid, respectively, $c$ standing the chemical concentration and $n$ denoting the densities of the biological organism.

The stochastic system \re{1.1} has been studied in \cite{Zhang1} where the authors have incorporated L\'evy noise acting only on the fluid velocity equation, assuming the absence of noise in the equation governing the chemical concentration $c$.  In our forthcoming paper, we will show the tightness property and the convergence of probability laws of solutions of system \re{3.2}. In this way, we have shown the convergence of the approximating system towards a solution of \re{1.1}. In fact, Let $(\mathbf{e}_i)_{i\in \mathbb{N}}\subset [C^\infty(\bo)]^3$ be an orthonormal basis consisting of the eigenfunctions  of the Stokes operator $A \bv = -\Pl \Delta \bv, \, \bv \in D(A)$, with $D(A)= \Hb^2(\bo) \cap \mathbf{V}$.   The  finite dimensional space is constructed by
$$ \mathbf{H}_m :=\text{span}\left\{ \mathbf{e}_i: \ i=1,...,m \right\},\ \forall m\in \mathbb{N},$$ and the considered model is obtained by projecting the fluid velocity equation of system \re{1.1} into $\mathbf{H}_m$, for any fixed $m\in\mathbb{N}$. Let us consider $\mathcal{V}:= \{ \bv\in C^\infty_c(\bo;\mathbb{R}^3): \; \nabla\cdot\bv=0\},$ $\bH :=\text{ closure of }  \mathcal{V} \text{ in }  \mathbb{L}^2(\bo),$ and $\bV:= \text{ closure of }  \mathcal{V} \text{ in }   \mathbb{H}^1_0(\bo)$. We denote by $\Pl:\mathbb{L}^2(\bo)\to \bH$ the Leray-Helmholtz  orthogonal projection and  the  Neumann Laplacian $A_1$ acting on $\mathbb{R}$-valued function is defined by $A_1 \phi=-\Delta\phi, \ \phi\in D(A_1)$, where $D(A_1)=\{ \phi\in \h^2 : \frac{\partial \phi}{\partial \nu}=0\text{ on }\partial \bo \}$.  We recall that the duality pairing between a given space $X$ and its dual space $X^\prime$ is denoted by $\langle \cdot\, ,\, \cdot\rangle$, and introduce the bi-linear maps  $B:\ve \times \ve \to \ve'$, $B_2:\elm^2(\bo)\times \h^2(\bo)\to \elm^{2}(\bo)$, and  $B_3:\sobm^{1,4}(\bo)\times D(A_1) \to \elm^2(\bo)$ as follows: 
	\begin{equation*}
		\begin{split}
		&\langle
		B(\bu,\bv),\bw\rangle =\sum_{i,j=1}^3 \int_\mathcal{O}\bu_i(x)\frac{\partial
			\bv_j(x)}{\partial x_i}\bw_j(x)dx,\,\,  \bw\in \ve,\text{ and } \bu, \bv\in \ve,\\
		& ( B_2(\varphi,\phi),\psi) =\int_\bo \varphi(x) \phi(x)\psi(x) dx,	\    \varphi\in \elm^2(\bo), \ \phi\in \h^{2}(\bo),\  \psi \in \h^{1}(\bo),\\
		& (B_3(\varphi,\phi),\psi)=-\int_\bo \nabla\cdot(\varphi(x)\nabla \phi(x))\psi(x) dx,	\    \varphi\in \sobm^{1,4}(\bo), \  \phi\in D(A_1),\  \psi \in \elm^2(\bo),
		\end{split}
\end{equation*}
where $(\cdot,\cdot)$ is denoting the inner product of $L^2(\bo)$.
For any fixed $m\in \mathbb{N}$,  we will deal on the time discretization of the following system of stochastic differential equation on the  complete filtered probability space $(\Omega, \mathcal{F}, \mathbb{F}=(\mathcal{F}_t)_{t\in[0,T]}, \mathbb{P})$,
	\begin{align}\label{3.2}
	\begin{split}
		& \bu_m(t)+\eta A\bu_mdt+B^m(\bu_m,\bu_m)dt-B_0^m(\theta_m(n_m),\varPhi)d t=\alpha  \pi_m (\Pl(\mathbf{f}(\bu_m))) d W(t), \\
		&	c_m(t)+ \eps A_1c_mdt+B_1(\bu_m,  c_m)d t=- B_2(\theta^{\eps_m}_m(n_m),c_m)d t+\gamma F^1_m(c_m)g( c _m )d \beta(t), \\
		&	 n_m(t)+ \delta A_1n_mdt+B_1(\bu_m,  n_m)d t= B_3(\theta_m(n_m),c_m) dt, 
	\end{split}
\end{align}
in $\bh_m$, $\elm^2(\bo)$ and $(\h^1(\bo))'$ respectively, with  	$\eta=\vartheta+\frac{\alpha^2}{2}$,  $\eps=\mu+\frac{\gamma^2}{2}$,
$	\eps_m:=\frac{8}{m}$, $\delta_m:=\eps_m+1$, $\theta^{\eps_m}_m(n_m):=\theta_m(n_m)+\eps_m>0$, where  $\theta_m:\mathbb{R}\to \mathbb{R}$ is  a $C^2$-truncation   of the real valued identity function defined by 
\begin{equation}\label{3.1}
	\theta_m(x)=	\begin{cases}
		0,\  \text{ if } x\leq0,\\
		3m^4x^5-8m^3x^4+6m^2x^3,\  \text{ if } 0<x\leq\frac{1}{m},\\
		x,\ \text{ if }  \frac{1}{m}< x\leq m,\\
		C^2-\text{extension},\  \text{ if }   m<x<m+1,\\
		m+1, \  \text{ if }  x\geq m+2,
	\end{cases}
\end{equation}
and 
such that $0\leq \theta'_m\leq 1 \text{ in } (m,m+1).$    $F_m:]0,+\infty[\to ]0,+\infty[$  is a function defined  by 
\begin{equation*}
  	F_m(x):=\min \left(1, \frac{m}{x}\right), \  x\in ]0,+\infty[,
\end{equation*}
and $F_m^1(c):= F_m(\abs{c}_{H^1}+1),\;  m\geq 1,\, c\in \mathrm{H}^1(\bo).$ 
$B^m$ and $B_0^m$ are operators defined respectively by 
\begin{align*}
	B^m(\bv,\bv)&=\pi_m B(\bv,\bv), \, \bv \in \bV, \\ 
	B_0^m(\varphi,\phi)&=\pi_m(\Pl(\varphi \nabla \phi)),\, \varphi\in \elm^2(\bo) \text{ and } \phi\in \sobm^{1,\infty}(\bo),
\end{align*}
 and  $\pi_m : \bH  \to \bH_m $ the orthogonal projection. 
The problem \eqref{3.2}    subjected to a regularised initial data: 
$$ \bu_m(0)=\bu_0^m:= \pi_m\bu_0, \ c_m(0)= c_0^m \ \text{ and } \  n_m(0)=n^m_0,$$
 is an infinite dimensional semi-discretization of \eqref{1.1} converging towards a weak martingale solution of \re{1.1} when the parameter $m$ goes to $\infty$.

The main goal of this work is to prove the existence  of probabilistic solution of system \re{3.2} for any arbitrary fixed $m\in \mathbb{N}$ by implementing the method used from \cite{Banas1,Banas,De,Glatt,Razafimandimby} for some nonlinear stochastic partial differential equations. Here the advantage is that our method of proof will be very useful for the numerical simulation of problem \re{1.1} whenever we will have a non-negativity conservation scheme for the chemotaxis system. While we are able to prove the existence of a probabilistic solution to \eqref{3.2}, the uniqueness of the   solution remains an open problem.

The layout of this paper is as follows: In Section 2, we introduce some basic notations, preliminary results, and state the main result corresponding to the problem \re{3.2}. In Section3,  we introduce a time discretization and analyze a semi-discret numerical scheme. Specifically, we prove the well-posedness of the numerical scheme and some uniform estimates for any discrete solutions, which are required in the convergence analysis, and finally in Section 4, we prove the main result by carrying out the convergence analysis of the numerical scheme.

	\section{Basic notations, assumptions and main result}
	
 For any $p\in [1,\infty)$ and $k\in \mathbb{N}$, we denote by $\elm^p, \sobm^{k,p}$ (resp. $\mathbb{L}^p, \mathbb{W}^{k,p}$) the well-known Lebesgue and Sobolev
		spaces, respectively, of functions $u: \MO \to \mathbb{R}$ (resp. $\bv:\MO \to \mathbb{R}^3$).  
%		 The spaces of functions $\bv:\mathbb{R}^3\to \mathbb{R}^3$  such that each component of $\bv$  belongs to $\mathrm{L}^p(\MO)$ (resp. 
%		to $\mathrm{W}^{k,p}(\MO)$) are denoted by $\e\elm^p(\MO)$  (resp. by 
%		$\mathbb{W}^{k,p}(\MO)$).
%		Hereafter, we simply write $\elm^p, \sobm^{k,p}, \mathbb{L}^p, \mathbb{W}^{k,p}$ instead of  $\elm^p(\MO), \sobm^{k,p}(\MO)$, $\mathbb{L}^p(\MO), \mathbb{W}^{k,p}(\MO)$ and
%		Throughout this paper,  all $\mathbb{R}^3$-valued functions will be written in bold-faced letters. 
  For $p\in [1,\infty)$ and $k\in \mathbb{N}\cup \{0\}$, the norms of the Banach spaces $\sobm^{k,p}$ and $\mathbb{W}^{k,p}$ are all denoted by $\abs{\cdot}_{k,p}$ which are the usual Sobolev norms for real-valued and $\mathbb{R}^3$-valued functions. The spaces $\elm^2, \h^k:= \sobm^{k,2}, \, \el^2,\,  \mathbb{H}^k:= \mathbb{W}^{k,2}$ are all separable Hilbert spaces.       The usual scalar products on $\elm^2$ and $\mathbb{L}^2$  are all denoted  by the same symbol $(
		\cdot \,,\, \cdot)$   and its associated norm is, of course,  denoted by $\abs{\cdot}_{0,2}$. \newcommand{\abse}[1]{\lve #1 \rve}
		By $\h^1_0$ we mean the well-known  separable  Hilbert space of functions in $\h^1$
		that vanish on the boundary on $\MO$.  For a Banach space $X$,  by  $X_{weak}$ we will denote the space $X$ equipped with the weak topology. 
	%We remark that $B_1(\bv, \phi)\in \elm^{2}(\bo)$ whenever  $\bv\in \bV$ and $\phi\in \h^2$. 
	%\begin{equation*}
	%( B_1(\bv,\phi),\psi)=\int_\bo\bv(x)\cdot\nabla \phi(x)\psi(x) dx,	\ \bv\in V, \  \phi\in \h^2,\  \psi \in  \elm^2
	%\end{equation*}
Throughout this paper, given two separable Hilbert spaces $X$ and $Y$  we denote by $\mathcal{L}_{HS}(X; Y)$ the space of of Hilbert–Schmidt operators from $X$ into $Y$.
	
		We impose the following assumptions on $\varPhi$, $\mathbf{g}_k$, $k=1,2,3,$  $\mathbf{f}$,  and on the  initial conditions.
	
	\begin{assumption} \label{assum1}
%		We assume that the domain $\bo$ is a bounded, opened and convex subset  of $\mathbb{R}^3$ with smooth boundary $\partial \bo$.
		 We assume that:
		\begin{itemize}
				\item [(i)] $\mathbf{f}:\ve\longrightarrow \mathcal{L}_{HS}(\mathbb{R}^3; \mathbb{L}^2)$ is defined by for any $\bv\in \bV$ and $z\in \mathbb{R}^3$, 
				\begin{equation*}
				\mathbf{f}(\bv)(z)=(\nabla \bu) z.
				\end{equation*}
			\item [(ii)]  For $k\in \{1,2,3\}$, $\mathbf{g}_k:=(g_k^1,g_k^2,g_k^3)\in  \sobb^{1,\infty}$, $\mathbf{g}_k=0$ on $\partial \bo$ and $\abs{\mathbf{g}_k}_{1,\infty}\leq 1$,
			\item [(iii)]  $\mathbf{g}_k$ is a divergence free vector fields, that is $\nabla\cdot \mathbf{g}_k=0$, for $k=1,2,3$,
			\item  [(iv)]  the matrix-valued function $\mathbb{G} : \bo \times  \bo \to \mathbb{R}^3 \otimes  \mathbb{R}^3$ defined by
			\begin{equation*}
				g^{i,j}(x,y)=\sum_{k=1}^{3}g_k^i(x)g_k^j(y), \qquad \forall i, j=1,2,3\ \text{and} \ \forall x, y \in\bo,
			\end{equation*}
			satisfies  $\mathbb{G}(x, x) = Id_{\mathbb{R}^3}$ for any $x\in  \bo$.
		\end{itemize}
		 In addition to the above conditions, we also assume that for any $m\in \mathbb{N}$,
		\begin{align}
			 \Phi\in \sobm^{1,\infty},\ \bu_0\in \bh,    \text{  }n^m_0\in  H^2 \text{ and } \ c^m_0\in H^2.\label{3.4}
		\end{align}	
%		as well as
%		\begin{equation}
%		
%		\end{equation}
	\end{assumption}

 %The definition of the space $\mathcal{L}_{HS}(\mathbb{R}^3;\mathbb{L}^2)$  is given in Appendix \ref{appen}.
  An  observation on Assumption \ref{assum1} is that by setting for $k=1,2,3$, 
		\begin{equation*}
			\mathbf{g}_k(x)=\begin{cases}
				1_{\bar\bo}(x) e_k\quad \text{  if } x\in  \bo,\vspace{0.2cm}\\
				0 \quad  \text{  if } x\in \partial\bo,
		\end{cases}	\end{equation*}
		where  $1_{\bar \bo}$ is the indicator function of the closure $\bar \bo$ of the domain $\bo$, the family of vector fields $\{\mathbf{g}_k\}_{k=1,2,3}$ satisfies (ii), (iii) and (iv).

It is important to note that the family $(\theta_m)_{m\in \mathbb{N}}$ defined by \re{3.1} satisfies the following properties.
	\begin{itemize}
		\item For any $m\geq 1$, $\theta_m\in C^2(\mathbb{R})$.
		\item  There exists a constant $\br_\theta>0$ such that for all $m\in \mathbb{N}$, 
		\begin{equation}\label{3.3}
			\begin{split}
				&\valabs{\theta_m(x)}\leq 17(m+1), \ \valabs{\theta'_m(x)}\leq \br_\theta,\ \valabs{\theta_m^{''}(x)}\leq \br_\theta m,\ \forall x\in \mathbb{R}, \\
				& -\frac{8}{m}\leq \theta_m(x), \ \forall x\in \mathbb{R}, \text{ and }\theta_m(x)\leq 9x, \ \forall x\geq 0.
			\end{split}
		\end{equation}
	\end{itemize} 
As a consequence of the above, for all $m\geq 1$, $\theta_m$ and $\theta'_m$ are Lipschitz and the Lipschitz  constant of $\theta_m$ does not depend on $m$. 

%Next, for any $m\in \mathbb{N}$, let $F_m:]0,+\infty[\to ]0,+\infty[$ be  a function defined  by 
%\begin{equation*}
	%F_m(x):=\min \left(1, \frac{m}{x}\right), \  x\in ]0,+\infty[.
%\end{equation*}
We recall, see   \cite{Caraballo}, that 
%\red{ By definition of $F_m$, we derive that for any integer $m\geq 1$ and real number $x>0$, the following hold}
\begin{align}
	&\valabs{F_m(x)-F_m(y)}\leq \frac{\valabs{x-y}}{y}, \ \forall x,y\in ]0,+\infty[,\ \forall m\geq 1, \label{s1.5} \\
	& 0\leq F_m(x)\leq 1 \text{ and }  0\leq F_m(x)\leq \frac{m}{x},\; \forall m\geq1, \, x>0.\label{s1.4}
\end{align}

Next, we prove the global existence of martingale solution to problem \re{3.2}  in the following sense. 
\begin{definition}\label{defi3.2}
	Let $m$ be an arbitrary non-negative integer but fixed. A martingale  solution of the problem \re{3.2} is a system $$(\bar{\Omega}, \bar{\mathcal{F}},\bar{\mathbb{F}}_m, \bar{\mathbb{P}},(\bu_m,c_m,n_m),(\bar{W}^m,\bar{\beta}^m)),$$
	where
	\begin{enumerate}
	\item $(\bar{\Omega},\bar{\mathcal{F}},\bar{\mathbb{F}}_m, \bar{\mathbb{P}})$ is a   complete filtered probability space, with the filtration $\bar{\mathbb{F}}_m=(\bar{\mathcal{F}}^m_t)_{t\in [0,T]}$  satisfying the usual conditions; 
	\item $(\bar{W}^m,\bar{\beta}^m)$ is a  Wiener processes on $\mathbb{R}^3\times\mathbb{R}^3$ over $(\bar{\Omega},\bar{\mathcal{F}},\bar{\mathbb{F}}_m, \bar{\mathbb{P}})$;
	\item$(\bu_m,c_m,n_m):[0,T]\times\bar{\Omega}\to \bh_m\times D(A_1)\times \h^1$ is  $\bar{\mathbb{F}}_m$-progressively measurable such that 
	$\bar{\mathbb{P}}_m$-a.s.,
	\begin{equation}\label{3.4*}
		\begin{split}
&	\hspace{2cm}\bu_m\in C([0,T];\bh_m), \\
	& c_m\in C([0,T];(\h^1)')\cap \elm^\infty(0,T;\h^1)\cap \elm^2(0,T, D(A_1)),\\
	 &n_m\in C([0,T];(\h^1)')\cap  \elm^\infty(0,T;\elm^2)\cap \elm^2(0,T;  \h^1),
		\end{split}
	\end{equation}	
	\item for all $t\in [0,T]$ and $\bar{\mathbb{P}}$-a.s. $\bu_m$, $c_m$ and $n_m$ satisfy the integral versions of the stochastic evolution equations  of system \re{3.2}  in $\bh_m$, $\elm^2$ and $(\h^1)^\prime$, respectively with $W$ and $\beta$ replaced respectively by $\bar{W}^m$ and $\bar{\beta}^m$.  
	\end{enumerate}

\end{definition}

%\subsection{Existence and uniqueness of a strong solution to the approximated system}

%In fact, the uniqueness is an open problem even for the deterministic 3D Chemotaxis without the Navier-Stokes equations.  
The main result of this paper is  labeled in the following proposition.
\begin{proposition}\label{lem3.2}
Assume that 
	$	\gamma^2<\frac{1}{484}\eps,$
	and that Assumption \ref{assum1} hold.  Then, 	for any $m\in \mathbb{N}$  the problem \re{3.2} has a  martingale solution $(\bar{\Omega}, \bar{\mathcal{F}},\bar{\mathbb{F}}_m, \bar{\mathbb{P}},(\bu_m,c_m,n_m),(\bar{W}^m,\bar{\beta}^m))$.
\end{proposition}
The rest of this paper is devoted on the proof of Proposition \ref{lem3.2} that is based on a time-discretization. In the  next section, we present and analyse a semi-discrete scheme based on a time discretisation of system \re{3.2}. Later on, we prove some useful a priori estimates of the semi-discret solutions.

	\section{Time discretisation and measurability of semi-discrete solutions.}
	
	For any integer $N\in \mathbb{N}$, let $h=\frac{T}{N}$  and $ \{0 = t_0 < t_1 <\cdot \cdot \cdot < t_N= T \},$
	be a partition of $[0,T ]$ where   $t_\ell = \ell h$, $\ell \in \{0,...N\}$. Let $(\Omega, \mathcal{F}, \mathbb{F},\mathbb{P})$ be a filtered probability space with the filtration satisfying the usual conditions. Let  $(W_1,W_2,W_3)$ be two independent $\mathbb{R}^3$-valued $\mathbb{F}$-Wiener processes. For $\ell\in\{0,1,2,...,N-1\}$ let 
	\begin{equation*}
		\Delta^\ell W_k=W_k(t_{\ell+1})-W_k(t_\ell)\text{ and }\Delta^\ell \beta_k=\beta_k(t_{\ell+1})-\beta_k(t_\ell), \ k\in \{1,2,3\}.
	\end{equation*}
	We set $\bu^0=\bu_0^m$, $c^0=c_0^m$, $n^0=n_0^m$ and $\Delta^\ell \beta=(\Delta^\ell \beta_1, \Delta^\ell \beta_2, \Delta^\ell \beta_3)$. For any  $N\in \mathbb{N}$,  we consider the following approximate scheme which consist on   filtered  probability space $(\Omega, \mathcal{F}_{t_\ell},  \mathbb{P})$ to find $\bu^\ell:\Omega\to \bh_m$, $c^\ell:\Omega\to D(A_1)$ and $n^\ell:\Omega\to D(A_1)$, $\ell=1,...,N$, such that the following equations hold $\mathbb{P}$-a.s.,
	\begin{equation}\label{s2.3}
		\begin{split}
			&	\bu^\ell-\bu^{\ell-1}+[\eta A\bu^\ell+B^m(\bu^{\ell-1},\bu^\ell)-B_0^m(\theta_m(n^\ell),\phi)]h=\alpha \pi_m (\Pl(\mathbf{f}(\bu^{\ell-1})))  \Delta^\ell W,\\
			&	c^\ell-c^{\ell-1}+ [\eps A_1c^\ell +B_1(\bu^\ell,  c^\ell)] h=- B_2(\theta_m^{\eps_m}(n^\ell),c^\ell)h+\gamma F^1_m(c^{\ell-1})g( c^{\ell-1})  \Delta^\ell\beta,\\
			&n^\ell-n^{\ell-1}+ [\delta A_1n^\ell+B_1(\bu^\ell,  n^\ell)]h= B_3(\theta_m(n^\ell),c^\ell)h, 
		\end{split}		
	\end{equation}
	in $\bh_m$, $\elm^2$ and $\elm^2$ respectively.
	\begin{proposition}\label{propo3.7}
		Let $N\in \mathbb{N}$. Given $(\bu^0,c^0,n^0)\in \bh_m\times \h^2\times \h^2$,  there exists $\{\bu^\ell\}_{\ell=1}^N\subset \bh_m$,  $\{c^\ell\}_{\ell=1}^N\subset D(A_1)$ and $\{n^\ell\}_{\ell=1}^N\subset D(A_1)$ satisfying \re{s2.3}. Moreover, for all $\ell\in \{1,2,...,N\}$, $\bu^\ell$, $c^\ell$ and $n^\ell$ are $\mathcal{F}_{t_\ell}$-measurable.
	\end{proposition}
	As a preparation of the proof of this proposition, we state and prove several auxiliary results. For this purpose, let us fix $\omega\in \Omega$, $N\in \mathbb{N}$ and $\ell\in\{1,...,N\}$.  For all $(\bar{c},\bar{n})\in \sobm^{1,4}\times \sobm^{1,4}$, we consider the following system:
	\begin{equation}\label{s2.4}
		\begin{split}
			&	\bu^\ell+[\eta A\bu^\ell+B^m(\bu^{\ell-1}(\omega),\bu^\ell)-B_0^m(\theta_m(\bar n),\phi)]h=\mathbf{f}_\ell(\omega),\text{ in } \bh_m,\\
			&	(I+h\eps A_1)c^\ell=\Psi_\ell(\omega,\bu^\ell,\bar c),\text{ in } \elm^2,\\
			&(I+h\delta A_1)n^\ell  =\mathcal{Z}(\omega,\bu^\ell,c^\ell,\bar n),\text{ in } \elm^2,
		\end{split}		
	\end{equation}
	where 
	\begin{equation*}
		\begin{split}
			&\mathbf{f}_\ell(\omega)=\bu^{\ell-1}(\omega)+\alpha\pi_m (\Pl(\mathbf{f}(\bu^{\ell-1}(\omega))))  \Delta^\ell W_k(\omega),\\
			&\Psi_\ell(\omega,\bu^\ell,\bar c)=c^{\ell-1}(\omega)-B_1(\bu^\ell,  \bar c) h- B_2(\theta^{\eps_m}_m(\bar n),\bar c)h+\gamma F^1_m(c^{\ell-1})g( c^{\ell-1}(\omega))  \Delta^\ell\beta(\omega),\text{ and }\\
			&\mathcal{Z}(\omega,\bu^\ell,c^\ell,\bar n)=n^{\ell-1}(\omega)-B_1(\bu^\ell,  \bar n)h +B_3(\theta_m(\bar n),c^\ell)h.
		\end{split}
	\end{equation*}
	\begin{lemma}\label{Lemmas2.3}
		Let  $\omega\in \Omega$, $N\in \mathbb{N}$ and $j\in\{1,...,N\}$ be fixed. For any $(\bar{c},\bar{n})\in \sobm^{1,4}\times \sobm^{1,4}$, the problem \re{s2.4} has a unique solution $(\bu^\ell ,c^\ell,n^\ell)\in \bh_m\times D(A_1)\times D(A_1)$ such that
		\begin{equation}\label{s2.5}
			\begin{split}
				&\abs{\bu^\ell}^2_{0,2}+2\eta h \abs{\nabla\bu^\ell}^2_{0,2}\leq 2\abs{\bu^{\ell-1}}^2_{0,2}+ 2h(m+1)\abs{\nabla\Phi}_{0,\infty}+\abs{\mathbf{f}_\ell(\omega)}^2_{0,2}=:\mathcal{T}^1_m(\omega,\ell),\\
				&\abs{c^\ell}^2_{0,2}+2h\eps\abs{\nabla c^\ell}^2_{0,2}+\abs{A_1c^\ell}^2_{0,2}\leq \mathcal{T}^2_m(\omega,\ell,\abs{ \bar c}^2_{1,4}),\\
				&\abs{n^\ell}^2_{0,2}+h\delta\abs{\nabla n^\ell}^2_{0,2}\leq \mathcal{T}^3_m(\omega,\ell,\abs{ \bar c}^2_{1,4},\abs{ \bar n}^2_{1,4}),
			\end{split}
		\end{equation}
		where 
		\begin{equation*}
			\mathcal{T}^2_m(\omega,\ell,\abs{ \bar c}^2_{1,4}):= \bk\abs{c^{\ell-1}}^2_{0,2}+\bk(m)\mathcal{T}^1_m(\omega,\ell)\abs{ \bar c}^2_{1,4}+\bk \abs{\gamma F^1_m(c^{\ell-1})g(c^{\ell-1}(\omega) ) \Delta^\ell\beta(\omega)}^2_{0,2},
		\end{equation*}
		and 
		\begin{equation*}
			\begin{split}
				\mathcal{T}^3_m(\omega,\ell,\abs{ \bar c}^2_{1,4},\abs{ \bar n}^2_{1,4})&=\bk\abs{n^{\ell-1}(\omega)}_{0,2}^2+\bk(m) \mathcal{T}^1_m(\omega,\ell) \abs{\bar n}^2_{1,4} +\bk(m)\mathcal{T}^2_m(\omega,\ell,\abs{ \bar c}^2_{1,4}))\\
				&\qquad+ \bk\abs{\bar n}^4_{1,4}+\bk(\mathcal{T}^2_m(\omega,\ell,\abs{ \bar c}^2_{1,4})))^2.
			\end{split}
		\end{equation*}
	\end{lemma}
	 \begin{proof}[\textbf{Proof of Lemma \ref{Lemmas2.3}}]
		Let us fix $\omega\in \Omega$, $N\in \mathbb{N}$ and  $j\in\{1,...,N\}$. Let $(\bar{c},\bar{n})\in \sobm^{1,4}\times \sobm^{1,4}$ arbitrary but fixed. %We will first prove the existence of the unique $\bu^\ell$ satisfying \re{s2.4}$_1$ and \re{s2.5}$_2$, secondly with $\bu^\ell$ in hand, we prove the existence and uniqueness of $c^\ell$ satisfying \re{s2.4}$_2$ and \re{s2.5}$_2$. Later on, with both $\bu^\ell$ and $c^\ell$ in hand, we will show the unique solvability of \re{s2.4}$_3$ satisfying \re{s2.5}$_3$. 
		For any $\bu$ and $\bv\in \bh_m$, let us define the following bilinear and linear operators $a_\ell(\cdot,\cdot)$ and $F_\ell(\cdot)$, respectively by
		\begin{equation*}
			\begin{split}
				&a_\ell(\bu,\bv):=(\bu,\bv)+(\eta hA\bu +hB^m(\bu^{\ell-1}(\omega),\bu ),\bv), \text{ and }\\
				&F_\ell(\bv)=(hB_0^m(\theta_m(\bar n),\phi)+\mathbf{f}_\ell(\omega),\bv).
			\end{split}
		\end{equation*}
		Due to the fact that $(B^m(\bu^{\ell-1}(\omega),\bv),\bv)=0$, for any $\bv\in \bh_m$, we can   prove that the bilinear form $a_\ell$ is   coercive on $\bh_m$, and both $a_\ell$ and $F_\ell$ are continuous on this space. By the Lax–Milgram lemma (see, e.g. \cite[Lemma 2.2.1.1, P.85]{Grisvard}), there exists a unique solution $\bu^\ell\in \bh_m$ such that $a_\ell(\bu^\ell,\bv)= F_\ell(\bv)$, for all $\bv\in \bh_m$, that is, $\bu^\ell$ satisfies also \re{s2.4}$_1$. Now, we easily establish  \re{s2.5}$_1$ from multiplying \re{s2.4}$_1$ by $\bu^\ell$  using the Young inequality.
		
		Let  $\bu^\ell$ as obtained above. By the equivalence of all norm on $\bh_m$, and the embedding $\elm^4\hookrightarrow \elm^2$ we obtain from   \re{s2.5}$_1$ that
		\begin{equation*}
			\begin{split}
				\abs{\Psi_\ell(\omega,\bu^\ell,\bar c)}^2_{0,2}
				&\leq  \bk\abs{c^{\ell-1}}^2_{0,2}+\bk\abs{\bu^\ell}^2_{0,\infty}\abs{\nabla \bar c}^2_{0,2}\\
				&\qquad+\bk [(m+1)^2+\frac{81}{m^2}]\abs{\bar c}^2_{0,2}+\bk \abs{\gamma F^1_m(c^{\ell-1}) g(c^{\ell-1}(\omega) ) \Delta^\ell\beta(\omega)}^2_{0,2}\\
				&\leq \bk\abs{c^{\ell-1}}^2_{0,2}+\bk(m)\mathcal{T}^1_m(\omega,\ell)\abs{ \bar c}^2_{1,4}+\bk \abs{\gamma F^1_m(c^{\ell-1})g(c^{\ell-1}(\omega) ) \Delta^\ell\beta(\omega)}^2_{0,2}<\infty.
			\end{split}
		\end{equation*}
		Thus, $\Psi_\ell(\omega,\bu^\ell,\bar c)\in \elm^2$. By  \cite[Theorem 2.2.2.5, P.91]{Grisvard}
		there exists a unique $c^\ell\in D(A_1)$ satisfying	\re{s2.4}$_2$.  As in the proof of \re{s2.5}$_1$, we obtain
		\begin{equation*}
			\abs{c^\ell}^2_{0,2}+h\eps\abs{\nabla c^\ell}^2_{0,2}\leq \frac{1}{2}\abs{c^\ell}^2_{0,2}+\abs{\Psi_\ell(\omega,\bu^\ell,\bar c)}^2_{0,2},
		\end{equation*}
		which  implies that 
		\begin{equation*}
			\abs{c^\ell}^2_{0,2}+2h\eps\abs{\nabla c^\ell}^2_{0,2}\leq \bk\abs{c^{\ell-1}}^2_{0,2}+\bk(m)\mathcal{T}^1_m(\omega,\ell)\abs{ \bar c}^2_{1,2}+\bk \abs{\gamma F^1_m(c^{\ell-1})g(c^{\ell-1}(\omega) ) \Delta^\ell\beta(\omega)}^2_{0,2}.	
		\end{equation*}
		Observe also that from \re{s2.4}$_2$, we can prove  that 
		\begin{equation*}
			\abs{A_1c^\ell}^2_{0,2}\leq \frac{\bk}{h^2\eps^2}\abs{c^\ell}^2_{0,2}+\frac{\bk}{h^2\eps^2}\abs{\Psi_\ell(\omega,\bu^\ell,\bar c)}^2_{0,2}.
		\end{equation*}
		Now,  \re{s2.5}$_2$ follows from the two last inequality above.
		
		With $\bu^\ell$ and $c^\ell$ in hand, we now solve \re{s2.4}$_3$. For this purpose, we observe that by the definition of $B_1(\bu^\ell,\bar{n})$ and the fact that  $B_3(\theta_m(\bar n), c^\ell)=\theta_m(\bar n)\Delta c^\ell+\theta'_m(\bar n)\nabla\bar n\cdot\nabla c^\ell$, we see that
		\begin{equation*}
			\begin{split}
				\abs{\mathcal{Z}(\omega,\bu^\ell,c^\ell,\bar n)}^2_{0,2}&\leq \bk\abs{n^{\ell-1}(\omega)}_{0,2}^2+\bk \abs{B_1(\bu^\ell,  \bar n)}^2_{0,2} +\bk\abs{B_3(\theta_m(\bar n),c^\ell)}^2_{0,2}\\
				&\leq \bk\abs{n^{\ell-1}(\omega)}_{0,2}^2+\bk \abs{\bu^\ell}_{0,\infty}^2  \abs{\nabla\bar n}^2_{0,2} +\bk_m\abs{\Delta c^\ell}^2_{0,2}+ 2\bk_\theta\abs{\nabla\bar n}^2_{0,4}\abs{\nabla c^\ell}^2_{0,4}\\
				&\leq \bk\abs{n^{\ell-1}(\omega)}_{0,2}^2+\bk(m) \mathcal{T}^1_m(\omega,\ell) \abs{\bar n}^2_{1,4}\\
				&\qquad +\bk_m\abs{A_1c^\ell}^2_{0,2}+ \bk\abs{\bar n}^4_{1,4}+\bk\abs{A_1 c^\ell}^4_{0,2}<\infty,
			\end{split}
		\end{equation*}
		where we used the Young inequality and the embeddings $\elm^4\hookrightarrow \elm^2$ and $\h^1\hookrightarrow \elm^4$. 	This implies that $\mathcal{Z}(\omega,\bu^\ell,c^\ell,\bar n)\in \elm^2$ and by   Theorem 2.2.2.5 in \cite[P.91]{Grisvard} there exists  a unique $n^\ell\in D(A_1)$ such that \re{s2.4}$_3$ holds. Now,  from \re{s2.4}$_3$ we derive  that 
		\begin{equation*}
			\begin{split}
				\abs{n^\ell}^2_{0,2}+h\delta\abs{\nabla n^\ell}^2_{0,2}&\leq \frac{1}{2}\abs{n^\ell}^2_{0,2}+\abs{\mathcal{Z}(\omega,\bu^\ell,c^\ell,\bar n)}^2_{0,2}\\
				&\leq \frac{1}{2}\abs{n^\ell}^2_{0,2}+\bk\abs{n^{\ell-1}(\omega)}_{0,2}^2+\bk(m) \mathcal{T}^1_m(\omega,\ell) \abs{\bar n}^2_{1,4} \\
				&\qquad+\bk(m)\abs{A_1c^\ell}^2_{0,2}+ \bk\abs{\bar n }^4_{1,4}+\bk\abs{A_1 c^\ell}^4_{0,2}\\
				&\leq \frac{1}{2}\abs{n^\ell}^2_{0,2}+\bk\abs{n^{\ell-1}(\omega)}_{0,2}^2+\bk(m) \mathcal{T}^1_m(\omega,\ell) \abs{\bar n}^2_{1,2} \\
				&\qquad+\bk(m)\mathcal{T}^2_m(\omega,\abs{ \bar c}^2_{1,2})+ \bk\abs{\bar n }^4_{1,4}+\bk(\mathcal{T}^2_m(\omega,\ell,\abs{ \bar c}^2_{1,4})))^2,
			\end{split}
		\end{equation*}
		from which \re{s2.5}$_3$ follows. The Lemma \ref{Lemmas2.3} is then proved.
	\end{proof}
	
	Now, we define the maps $\Upsilon_\ell: \sobm^{1,4}\longrightarrow\bh_m$,  and  
	\begin{equation}\label{s2.6*}
		S_\ell: \sobm^{1,4}\times \sobm^{1,4}\longrightarrow D(A_1)\times D(A_1)\hookrightarrow \sobm^{1,4}\times \sobm^{1,4},
	\end{equation}
	as follows: For any $(\bar c, \bar n)\in \sobm^{1,4}\times \sobm^{1,4}$,  $\Upsilon_\ell( \bar n):=\bu^\ell$ where $\bu^\ell$ is the unique solution of \re{s2.4}$_1$ corresponding to $\bar n$ and  $S_\ell(\bar c, \bar n):= (c^\ell,n^\ell)$, such that $(\Upsilon_\ell(\bar n),c^\ell,n^\ell)$ solve \re{s2.3}. By the existence and uniqueness result obtained in Lemma \ref{Lemmas2.3} and the fact that $\h^2\hookrightarrow \sobm^{1,4}$, we infer that the maps $\Upsilon_\ell$ and $S_\ell$ are well-defined. In the next Lemma, we show a useful property of  the maps $\Upsilon_\ell$ and $S_\ell$.
	\begin{lemma}\label{Lemmas2.4}
		Let $(\bar c_1,\bar n_1)$, $(\bar c_2,\bar n_2)$ be two elements of $\sobm^{1,4}\times \sobm^{1,4}$. Then, 
		\begin{equation*}
			\abs{\Upsilon_\ell( \bar n_1)-\Upsilon_\ell( \bar n_2)}^2_{0,2}\leq \bk \abs{\nabla\Phi}_{0,\infty}^2\abs{\bar n_1-\bar n_2}^2_{1,4},
		\end{equation*}
		and 
		\begin{equation*}
			\begin{split}
				\abs{S_\ell(\bar c_1,\bar n_1)-S_\ell(\bar c_2,\bar n_2)}^2_{(\sobm^{1,4})^2}&\leq \bk\abs{ \bar c_1-\bar c_2}^2_{1,4} +\bk\abs{ \bar n_1-\bar n_2}^2_{1,4}(\mathcal{T}^1_m(\omega,\ell)+\abs{ \bar c_2}^2_{1,4}+\abs{ \bar n_2}^2_{1,4})\\
				&\qquad+\bk\abs{\bar n_1-\bar n _2}^2_{1,4}\abs{ \bar c_2}^2_{1,4}(1+\abs{\bar n_1}^2_{1,4})\\
				&\qquad+\bk\abs{ \bar c_1-\bar c_2}^2_{1,4}(1+\abs{\bar n_1}^2_{1,4})\\
				&\qquad+\bk \abs{\bar n_1-\bar n_2}^2_{1,4}(1+ \abs{\bar n_2}^2_{1,4})\mathcal{T}^2_m(\omega,\ell,\abs{ \bar c_2}^2_{1,4}),
			\end{split}
		\end{equation*}
		where $\mathcal{T}^2_m(\cdot,\cdot,\cdot)$ is defined in Lemma \ref{Lemmas2.3}. In particular, the map $S_\ell$ is continuous.
	\end{lemma}

	\begin{prev}[\textbf{Proof of Lemma \ref{Lemmas2.4}} ]
		Let $(\bar c_i,\bar n_i)\in \sobm^{1,4}\times \sobm^{1,4}$,  $ (c^\ell_i,n^\ell_i)=S_\ell(\bar c_i,\bar n_i)$, $\bu^\ell_i=\Upsilon_\ell(\bar n_i)$, $i=1,2$.  We put $\bw^\ell=\bu^\ell_1-\bu^\ell_2$, $\varphi^\ell=c^\ell_1-c^\ell_2$, $\psi^\ell=n^\ell_1-n^\ell_2$, $\bar n=\bar n_1-\bar n_2$ and $\bar c=\bar c_1-\bar c_2$. Then, $(\bw^\ell, \varphi^\ell, \psi^\ell)$ satisfies the following system,
		\begin{equation}\label{s2.6}
			\begin{split}
				&	\bw^\ell+[\eta A\bw^\ell+B^m(\bu^{\ell-1}(\omega),\bw^\ell)-B_0^m(\theta_m(\bar n_1)-\theta_m(\bar n_2),\varPhi)]h=0,\text{ in } \bh_m,\\
				&\varphi^\ell+h\eps A_1\varphi^\ell=-B_1(\bu_1^\ell,  \bar c) h-B_1(\bw^\ell,  \bar c_2) h- B_2(\theta_m^{\eps_m}(\bar n_1),\bar c)h\\
				&\hspace{5cm}- B_2(\theta^{\eps_m}_m(\bar n_1)-\theta^{\eps_m}_m(\bar n_2),\bar c_2)h,\text{ in } \elm^2,\\
				&\psi^\ell+h\eps A_1\psi^\ell=-B_1(\bu_1^\ell,  \bar n) h-B_1(\bw^\ell,  \bar n_2) h- B_3(\theta_m(\bar n_1),\varphi^\ell)h\\
				&\hspace{5cm}- B_3(\theta_m(\bar n_1)-\theta_m(\bar n_2), c^\ell_2)h,\text{ in } \elm^2.
			\end{split}		
		\end{equation}
		From  the  Lipschitz property of $\theta_m$ and    \re{s2.6}$_1$  we derive that 
		\begin{equation*}
			\abs{\bw^\ell}_{0,2}^2+\eta \abs{\nabla \bw^\ell}^2_{0,2}\leq \frac{1}{2}\abs{\bw^\ell}_{0,2}^2+\bk \abs{\nabla\Phi}_{0,\infty}^2\abs{\bar n}^2_{0,2},
		\end{equation*}
		which with the embedding $\sobm^{1,4}\hookrightarrow \elm^4\hookrightarrow \elm^2$ imply that 
		\begin{equation}\label{s2.7}
			\abs{\bw^\ell}_{0,2}^2+2\eta \abs{\nabla \bw^\ell}^2_{0,2}\leq \bk \abs{\nabla\Phi}_{0,\infty}^2\abs{\bar n}^2_{1,4}.
		\end{equation}
		The first inequality of Lemma \ref{Lemmas2.4} is deduced from \re{s2.7}. In what follows, we prove the second inequality. 
		Since $\varphi^\ell\in D(A_1)$, we obtain  by taking  the $\elm^2$-inner product of \re{s2.6}$_2$ and $\varphi^\ell+A_1\varphi^\ell$  that 
		\begin{equation}\label{s2.8}
			\begin{split}
				\abs{\varphi^\ell}^2_{1,2}+h\eps\abs{\nabla \varphi^\ell}^2_{0,2}+h\eps\abs{A_1\varphi^\ell}^2_{0,2}&\leq \frac{1}{2}\abs{\varphi^\ell}^2_{0,2}+\frac{h\eps}{2}\abs{A_1\varphi^\ell}^2_{0,2}+ \bk \abs{B_1(\bu_1^\ell,  \bar c)}^2_{0,2}\\
				&\qquad+\bk \abs{B_1(\bw^\ell,  \bar c_2) }^2_{0,2}+\bk \abs{B_2(\theta^{\eps_m}_m(\bar n_1),\bar c)}^2_{0,2}\\
				&\qquad+\bk \abs{B_2(\theta^{\eps_m}_m(\bar n_1)-\theta^{\eps_m}_m(\bar n_2),\bar c_2)}^2_{0,2}.
			\end{split}
		\end{equation}
		Now, by using the equivalence of norm on $\bh_m$ and the fact that $\elm^4\hookrightarrow \elm^2$, we get
		\begin{equation*}
			\abs{B_1(\bu_1^\ell,  \bar c)}^2_{0,2}\leq \abs{\bu_1^\ell}^2_{0,\infty}\abs{\nabla \bar c}^2_{0,2}\leq \bk(m)\mathcal{T}^1_m(\omega,\ell)\abs{ \bar c}^2_{1,4}.
		\end{equation*}
		Similarly and using  \re{s2.7}, we have
		\begin{equation*}
			\abs{B_1(\bw^\ell,  \bar c_2) }^2_{0,2}\leq \abs{\bw^\ell}^2_{0,\infty}\abs{\nabla \bar c_2}^2_{0,2}\leq \bk(m)\abs{\nabla\Phi}_{0,\infty}^2\abs{\bar n}^2_{1,4}\abs{ \bar c_2}^2_{1,4}.
		\end{equation*}
		By using \re{3.3} and the fact that $\sobm^{1,4}\hookrightarrow \elm^4\hookrightarrow \elm^2$, we infer that
		\begin{equation*}
			\abs{B_2(\theta^{\eps_m}_m(\bar n_1),\bar c)}^2_{0,2}\leq [289(m+1)^2+\frac{81}{m^2}]\abs{\bar c}^2_{0,2}\leq\bk [(m+1)^2+\frac{81}{m^2}]\abs{\bar c}^2_{1,4}.
		\end{equation*}
		By using $\theta^{\eps_m}_m(\bar n_1)-\theta^{\eps_m}_m(\bar n_2)=\theta_m(\bar n_1)-\theta_m(\bar n_2)$, and  the Lipschitz property of $\theta_m$  and the embedding $\sobm^{1,4}\hookrightarrow \elm^\infty$, we get 
		\begin{equation*}
			\abs{B_2(\theta^{\eps_m}_m(\bar n_1)-\theta^{\eps_m}_m(\bar n_2),\bar c_2)}^2_{0,2}\leq \bk\abs{\bar c_2}^2_{0,\infty}\abs{\bar n}^2_{0,2}\leq \bk\abs{\bar c_2}^2_{1,4}\abs{\bar n}^2_{1,4}.
		\end{equation*}
		By combining the above inequalities, we infer from \re{s2.8} that 
		\begin{equation}\label{s2.9}
			\begin{split}
				\abs{\varphi^\ell}^2_{1,2}+2h\eps\abs{\nabla \varphi^\ell}^2_{0,2}+h\eps\abs{A_1\varphi^\ell}^2_{0,2}&\leq \bk\abs{\bar n}^2_{1,4}\abs{ \bar c_2}^2_{1,4}+\bk\abs{ \bar c}^2_{1,4}.
			\end{split}
		\end{equation}
		Since $\psi^\ell\in D(A_1)$, we derive from \re{s2.6}$_3$  that 
		\begin{equation}\label{s2.10}
			\begin{split}
				\abs{\psi^\ell}^2_{1,2}+h\delta\abs{\nabla \psi^\ell}^2_{0,2}+h\eps\abs{A_1\psi^\ell}^2_{0,2}&\leq \frac{1}{2}\abs{\psi^\ell}^2_{0,2}+\frac{h\delta}{2}\abs{A_1\psi^\ell}^2_{0,2}+ \bk \abs{B_1(\bu_1^\ell,  \bar n)}^2_{0,2}\\
				&\qquad+\bk \abs{B_1(\bw^\ell,  \bar n_2) }^2_{0,2}+\bk \abs{B_3(\theta_m(\bar n_1),\varphi^\ell)}^2_{0,2}\\
				&\qquad+\bk \abs{B_3(\theta_m(\bar n_1)-\theta_m(\bar n_2), c^\ell_2)}^2_{0,2}.
			\end{split}
		\end{equation}
		By using \re{s2.5}$_1$ and \re{s2.7}, we can easily prove that 
		\begin{equation}
			\abs{B_1(\bu_1^\ell,  \bar n)}^2_{0,2} +\bk \abs{B_1(\bw^\ell,  \bar n_2) }^2_{0,2}\leq \bk(m)\mathcal{T}^1_m(\omega,\ell)\abs{ \bar n}^2_{1,4}+\bk(m)\abs{\nabla\phi}_{0,\infty}^2\abs{\bar n}^2_{1,4}\abs{ \bar n_2}^2_{1,4}.\label{s2.11}
		\end{equation}
		We recall that  $B_3(\theta_m(\bar n_1),\varphi^\ell)$  can be rewritten as follows:
		\begin{equation*}
			B_3(\theta_m(\bar n_1),\varphi^\ell)=-\theta_m((\bar n_1))A_1 \varphi^\ell+\theta'_m((\bar n_1))\nabla(\bar n_1)\cdot\nabla \varphi^\ell.
		\end{equation*}
		Hence by using \re{3.3}, the H\"older inequality, the embedding $\h^1\hookrightarrow \elm^2$ and inequality \re{s2.9}, we arrive at
		\begin{equation}\label{s2.12}
			\begin{split}
				\abs{B_3(\theta_m(\bar n_1),\varphi^\ell)}^2_{0,2}&\leq 578(m+1)^2\abs{A_1\varphi^\ell}^2_{0,2}+2\br_\theta\abs{\nabla\bar n_1}^2_{0,4}\abs{\nabla\varphi^\ell}^2_{0,4}\\
				%&\leq 2(m+1)^2\abs{\bj_m(A_1\varphi^\ell)}^2_{0,2}+2\bk\abs{\bar n_1}^2_{1,4}\abs{ \varphi^\ell}^2_{2,2}\\
				&\leq 578(m+1)^2\abs{A_1\varphi^\ell}^2_{0,2}+2\bk\abs{\bar n_1}^2_{1,4}\abs{ \varphi^\ell}^2_{2,2}\\
				&\leq \bk\abs{\bar n}^2_{1,4}\abs{ \bar c_2}^2_{1,4}(1+\abs{\bar n_1}^2_{1,4})+\bk\abs{ \bar c}^2_{1,4}(1+\abs{\bar n_1}^2_{1,4}).
			\end{split}
		\end{equation}
		Now, we remark that
		\begin{equation*}
			\begin{split}
				\abs{(B_3(\theta_m(\bar n_1)-\theta_m(\bar n_2)),c_2^\ell)}^2_{0,2}&\leq 2\int_{\bo}\valabs{\theta_m(\bar n_1(x))-\theta_m(\bar n_2(x))\Delta c_2^\ell(x)}^2dx\\
				&\qquad+2\int_{\bo}\valabs{\nabla(\theta_m(\bar n_1(x))-\theta_m(\bar n_2(x)))\cdot\nabla c_2^\ell(x)}^2dx\\
				&=I_1+I_2.
			\end{split}
		\end{equation*}
		Similarly to the proof of \re{s2.12}, we have
		\begin{equation*}
			I_1\leq \bk\abs{\bar n}^2_{0,\infty}\abs{A_1c_2^\ell}^2_{0,2}\leq \bk\abs{\bar n}^2_{0,\infty}\abs{A_1c_2^\ell}^2_{0,2}\leq\bk \abs{\bar n}^2_{1,4}\mathcal{T}^2_m(\omega,\ell,\abs{ \bar c_2}^2_{1,4}).
		\end{equation*}
		Since $\nabla(\theta_m(\bar n_1)-\theta_m(\bar n_2))=\theta'_m(\bar n_1)\nabla\bar n+(\theta'_m(\bar n_1)-\theta'_m(\bar n_2))\nabla\bar n_2,$
		we see that
		\begin{equation*}
			\begin{split}
				I_2&\leq 2\int_{\bo}\valabs{\theta'_m(\bar n_1(x))\nabla\bar n(x)\cdot\nabla c_2^\ell(x)}^2dx\\
				&\qquad+\int_{\bo}\valabs{\theta'_m(\bar n_1(x))-\theta'_m(\bar n_2(x))}^2\valabs{\nabla\bar n_2(x)\cdot\nabla c_2^\ell(x)}^2dx.
			\end{split}
		\end{equation*}
		By the boundedness of  $\theta'_m$,  the H\"older inequality, the embedding $\h^1\hookrightarrow \elm^4$ and \re{s2.5}$_2$ we find that
		\begin{equation*}
\begin{split}
2\int_{\bo}\valabs{\theta'_m(\bar n_1(x))\nabla\bar n(x)\cdot\nabla c^\ell_2(x)}^2dx&\leq 2\br_\theta\abs{\nabla\bar n}^2_{0,4}\abs{\nabla c_2^\ell}^2_{0,4}\\
&\leq \bk\abs{\bar n}^2_{1,4}\abs{c_2^\ell}^2_{2,2}\\
&\leq \bk \abs{\bar n}^2_{1,4}\mathcal{T}^2_m(\omega,\ell,\abs{ \bar c_2}^2_{1,4}).
\end{split}
		\end{equation*}
		Thanks to the Lipschitz property of $\theta'_m$, the embedding $\sobm^{1,4}\hookrightarrow \elm^\infty$ and the H\"older inequality, we get
		\begin{equation*}
			\begin{split}
				\int_{\bo}\valabs{\theta'_m(\bar n_1(x))-\theta'_m(\bar n_2(x))}^2\valabs{\nabla\bar n_2(x)\cdot\nabla c^\ell_2(x)}^2dx&\leq \bk(m) \abs{\bar n}_{0,\infty}^2\int_{\bo}\valabs{\nabla\bar n_2(x)\cdot\nabla c^\ell_2(x)}^2dx\\
				&\leq \bk(m)\abs{\bar n}^2_{1,4}\abs{\bar n_2}^2_{1,4}\abs{c_2^\ell}^2_{2,2}\\
				&\leq \bk(m) \abs{\bar n}^2_{1,4}\abs{\bar n_2}^2_{1,4}\mathcal{T}^2_m(\omega,\ell,\abs{ \bar c_2}^2_{1,4}).
			\end{split}
		\end{equation*}
		By combining the above inequalities, we see that 
		\begin{equation*}
			\abs{(B_3(\theta_m(\bar n_1)-\theta_m(\bar n_2)),c^\ell_2)}^2_{0,2}\leq \bk \abs{\bar n}^2_{1,4}(1+ \abs{\bar n_2}^2_{1,4})\mathcal{T}^2_m(\omega,\ell,\abs{ \bar c_2}^2_{1,4}).
		\end{equation*}
		Plugging \re{s2.11}, \re{s2.12} and the last inequality in  \re{s2.10} we obtain
		\begin{equation*}
			\begin{split}
				\frac{1}{2}\abs{\psi^\ell}^2_{1,2}+\frac{h\delta}{2}\abs{\nabla \psi^\ell}^2_{0,2}+h\eps\abs{A_1\psi^\ell}^2_{0,2}&\leq \bk\abs{ \bar n}^2_{1,4}(\mathcal{T}^1_m(\omega,\ell)+\abs{\nabla\Phi}_{0,\infty}^2\abs{ \bar n_2}^2_{1,4})\\
				&\quad+\bk\abs{\bar n}^2_{1,4}\abs{ \bar c_2}^2_{1,4}(1+\abs{\bar n_1}^2_{1,4})+\bk\abs{ \bar c}^2_{1,4}(1+\abs{\bar n_1}^2_{1,4})\\
				&\quad+\bk \abs{\bar n}^2_{1,4}(1+ \abs{\bar n_2}^2_{1,4})\mathcal{T}^2_m(\omega,\ell,\abs{ \bar c_2}^2_{1,4}),
			\end{split}
		\end{equation*}
		from which along with  \re{s2.9} and  \re{s2.7}, we derive  the second inequality of Lemma \ref{Lemmas2.4}. This ends the proof of Lemma \ref{Lemmas2.4}.
	\end{prev}
	
	We now give the proof of Proposition 	\ref{propo3.7} which will be  split into two parts.

	\begin{prev}[\textbf{Proof of Proposition \ref{propo3.7}}] \hfill\\
	\noindent 	\textbf{Part 1: Proof of the existence.}	 To prove the existence result stated in the proposition, we will use the Leray-Schauder fixed point theorem (see e.g.  \cite[Theorem 5.4, P. 124]{Granas}). For this purpose, we observe that from Lemma \ref{Lemmas2.4} the map $S_\ell$ is continuous. Also, from Lemma \ref{Lemmas2.3} we see that   for any arbitrary $\bk>0$, $S_\ell (B_\bk)$ where  $$B_\bk=\{( c, n)\in \sobm^{1,4}\times \sobm^{1,4}: \abs{(n,c)}_{(\sobm^{1,4})^2}\leq \bk \},$$ is a bounded set of $\h^2\times \h^2$, hence compact in $\sobm^{1,4}\times \sobm^{1,4}$.  
		
		We also have the following claim.
		\begin{claim}\label{claim2.10}
			Let $\lambda\in (0,1)$. Then,  
			\begin{equation*}
				\bec(S_\ell):=\left\{ (n_\lambda,c_\lambda)\in \sobm^{1,4}\times \sobm^{1,4}:  (n_\lambda,c_\lambda)=\lambda S_\ell(n_\lambda,c_\lambda) \text{ for some } 0<\lambda<1 \right\}.
			\end{equation*}
			is uniformly bounded, i.e., there exists $\bk_\ell>0$ such that 
			\begin{equation}\label{s2.13}
				\abs{(n_\lambda,c_\lambda)}_{(\sobm^{1,4})^2}\leq \bk_\ell, \ \forall (n_\lambda,c_\lambda)\in \bec(S_\ell).
			\end{equation}
		\end{claim}
		With these two observations and Claim above, it follows from the Leray-Schauder fixed point theorem (see e.g.,  \cite[Theorem 5.4, P. 124]{Granas}), that there exists $(\bu^\ell,c^\ell,n^\ell)$ satisfying  \re{s2.3}.
		
		To complete the proof of the existence part, let us proof the claim.

		\begin{prev}[\textbf{Proof of Claim \ref{claim2.10}}]
			Let $(n_\lambda,c_\lambda)$ in $\bec(S_\ell)$, $\lambda\in (0,1)$. By definition of  $S_\ell$,  the triple  $(\bu_\lambda, c_\lambda,n_\lambda)$, with $\bu_\lambda=\Upsilon_\ell(n_\lambda)$, satisfies the following system:
			\begin{equation}\label{s2.14}
				\begin{split}
					&	\bu_\lambda-\lambda\bu^{\ell-1}+[\eta A\bu_\lambda+B^m(\bu^{\ell-1},\bu_\lambda)-\lambda B_0^m(\theta_m(n_\lambda),\varPhi)]h=\lambda\alpha \pi_m(
					\Pl(\mathbf{f}(\bu^{\ell-1})) ) \Delta^\ell W,\\
					&	c_\lambda-\lambda c^{\ell-1}+ [\eps A_1c_\lambda +B_1(\bu_\lambda,  c_\lambda)] h=- \lambda B_2(\theta^{\eps_m}_m(n_\lambda),c_\lambda)h+\lambda\gamma F^1_m(c^{\ell-1})g(c^{\ell-1})  \Delta^\ell\beta,\\
					&n_\lambda-\lambda n^{\ell-1}+ [\delta A_1n_\lambda+B_1(\bu_\lambda,  n_\lambda)]h= B_3(\theta_m(n_\lambda),c_\lambda)h, 
				\end{split}		
			\end{equation}
			in $\bh_m$, $\elm^2$ and $\elm^2$ respectively.
			
			Now, by multiplying  \re{s2.14}$_1$ in $\elm^2$ by  $\bu_\lambda$ and  by using \re{3.3},
			\begin{equation}\label{s2.15}
				\abs{\bu_\lambda}^2_{0,2}+2\eta h \abs{\nabla\bu_\lambda}^2_{0,2}\leq \mathcal{T}^1_m(\omega,\ell),
			\end{equation}
			where $\mathcal{T}^1_m(\omega,\ell)$ is defined in Lemma \ref{Lemmas2.3}. Similarly, but using   \re{s1.4}$_1$,  the H\"older and Young inequalities, we see that
			\begin{equation*}
				\begin{split}
					\abs{c_\lambda}^2_{0,2}+\eps h\abs{\nabla c_\lambda}^2_{0,2}&= \lambda(c^{\ell-1},c_\lambda)-\lambda h(B_2(\theta^{\eps_m}_m(n_\lambda),c_\lambda),c_\lambda)+\lambda\gamma F^1_m(c^{\ell-1})(g( c^{\ell-1})  \Delta^\ell\beta,c_\lambda)\\
					&\leq \frac{1}{2}\abs{c_\lambda}^2_{0,2}-\lambda h(B_2(\theta_m^{\eps_m}(n_\lambda),c_\lambda),c_\lambda)+\bk \abs{c^{\ell-1}}^2_{0,2}\\
					&\qquad+\bk \abs{F^1_m(c^{\ell-1})g( c^{\ell-1})  \Delta^\ell\beta}^2_{0,2}.
				\end{split}
			\end{equation*} 
			Since  $-\frac{8}{m}\leq \theta_m(n_\lambda)$, and $\theta^{\eps_m}_m(n_\lambda):=\theta_m(n_\lambda)+\eps_m>0$, we see that
			\begin{equation*}
				-h\lambda(B_2(\theta^{\eps_m}_m(n_\lambda),c_\lambda),c_\lambda)=-h\lambda\int_\bo\left(\theta_m(n_\lambda(x))+\eps_m\right)c_\lambda^2(x)dx\leq 0.
			\end{equation*}
			From these two last inequalities, we deduce that 
			\begin{equation}\label{s2.16}
				\abs{c_\lambda}^2_{0,2}+2\eps h\abs{\nabla c_\lambda}^2_{0,2}\leq \bk \abs{c^{\ell-1}(\omega)}^2_{0,2}+\bk \abs{F^1_m(c^{\ell-1})g( c^{\ell-1})  \Delta^\ell\beta}^2_{0,2}:=\mathcal{T}_3(\omega,\ell).
			\end{equation}
			By Lemma \ref{Lemmas2.3}, $S_\ell(n_\lambda,c_\lambda)\in D(A_1)\times D(A_1)$, hence  $A_1c_\lambda\in \elm^2$. Then, we can multiply \re{s2.14}$_2$ in $\elm^2$ by $A_1c_\lambda$, we see that
			\begin{equation*}
				\begin{split}
					\abs{\nabla c_\lambda}^2_{0,2}+\eps h\abs{A_1 c_\lambda}^2_{0,2}
					&\leq \frac{\eps h}{2}\abs{A_1c_\lambda}^2_{0,2}+\bk\abs{B_1(\bu_\lambda,c_\lambda)}_{0,2}^2+\bk\abs{B_2(\theta^{\eps_m}_m(n_\lambda),c_\lambda)}_{0,2}^2\\
					&\qquad+\bk \abs{c^{\ell-1}}^2_{0,2}+\bk \abs{F^1_m(c^{\ell-1})g( c^{\ell-1} ) \Delta^\ell\beta}^2_{0,2}.
				\end{split}
			\end{equation*} 
			By using the equivalence of norms on $\bh_m$, the inequalities \re{s2.15} and \re{s2.16}, we get
			\begin{equation*}
				\abs{B_1(\bu_\lambda,c_\lambda)}_{0,2}^2\leq \abs{\bu_\lambda}^2_{0,\infty}\abs{\nabla c_\lambda}^2_{0,2}\leq \bk(m)\abs{\bu_\lambda}^2_{0,2}\abs{\nabla c_\lambda}^2_{0,2}\leq \bk(m)\mathcal{T}^1_m(\omega,\ell) \mathcal{T}_3(\omega,\ell).
			\end{equation*}
			By using \re{3.3} and \re{s2.16}, we get
			\begin{equation*}
				\abs{B_2(\theta^{\eps_m}_m(n_\lambda),c_\lambda)}_{0,2}^2\leq 289(m+1)^2\abs{c_\lambda}^2_{0,2}\leq 289(m+1)^2\mathcal{T}_3(\omega,\ell).
			\end{equation*} 
			Since $0\leq F^1_m(c^{\ell-1})\leq 1$, 	we   infer from the  three   inequalities that 
			\begin{equation}\label{s2.17*}
				\abs{\nabla c_\lambda}^2_{0,2}+2\eps h\abs{A_1 c_\lambda}^2_{0,2}
				\leq \bk_m(\mathcal{T}^1_m(\omega,\ell)+1) \mathcal{T}_3(\omega,\ell)+\bk \abs{c^{\ell-1}}^2_{0,2}+\bk \abs{g (c^{\ell-1})  \Delta^\ell\beta}^2_{0,2},
			\end{equation} 
			from which and the embedding $\h^2\hookrightarrow \sobm^{1,4}$ and \re{s2.16} we derive  that 
			\begin{equation}\label{s2.17}
				\begin{split}
					\abs{ c_\lambda}^2_{1,4}\leq \bk \abs{c_\lambda}_{2,2}^2
					&\leq \bk_m\left(1+\mathcal{T}^1_m(\omega,\ell)\right) \mathcal{T}_3(\omega,\ell)\\
					&\qquad+\bk \abs{c^{\ell-1}}^2_{0,2}+\bk \abs{g (c^{\ell-1})  \Delta^\ell\beta}^2_{0,2}:=\mathcal{T}_4(\omega,\ell).
				\end{split}
			\end{equation} 
			By  multiplying  \re{s2.14}$_3$ with $n_\lambda\in \elm^2$  integrating by parts, using  \re{3.3},  \re{s2.16}  and the H\"older and the Young inequalities we have
			\begin{equation*}
				\begin{split}
					\abs{n_\lambda}^2_{0,2}+\delta h\abs{\nabla n_\lambda}^2_{0,2}&= \lambda(n^{\ell-1},c_\lambda)+ h(B_3(\theta_m(n_\lambda),c_\lambda),n_\lambda)\\
					&\leq \frac{1}{2}\abs{n_\lambda}^2_{0,2}+h\int_\bo \theta_m (n_\lambda(x))\nabla c_\lambda(x)\cdot\nabla n_\lambda(x)dx+\bk \abs{n^{\ell-1}}^2_{0,2}\\
					&\leq \frac{1}{2}\abs{n_\lambda}^2_{0,2}+h(m+1)\abs{\nabla c_\lambda}_{0,2} \abs{\nabla n_\lambda}_{0,2}  +\bk \abs{n^{\ell-1}}^2_{0,2}\\
					&\leq \frac{1}{2}\abs{n_\lambda}^2_{0,2}+\frac{\delta h}{2}\abs{\nabla n_\lambda}^2_{0,2}+\bk(m)\abs{\nabla c_\lambda}_{0,2}^2  +\bk \abs{n^{\ell-1}}^2_{0,2}\\
					&\leq \frac{1}{2}\abs{n_\lambda}^2_{0,2}+\frac{\delta h}{2}\abs{\nabla n_\lambda}^2_{0,2}+\frac{\bk(m)}{h\eps}\mathcal{T}_3(\omega,\ell)  +\bk \abs{n^{\ell-1}}^2_{0,2}.
				\end{split}
			\end{equation*} 
			The last line of the above chain of inequalities implies
			\begin{equation}\label{s2.18}
				\abs{n_\lambda}^2_{0,2}+\delta h\abs{\nabla n_\lambda}^2_{0,2}\leq \frac{\bk(m)}{h\eps}\mathcal{T}_3(\omega,\ell)  +\bk \abs{n^{\ell-1}}^2_{0,2}:=\mathcal{T}_5(\omega,\ell).
			\end{equation}
			By multiplying \re{s2.14}$_3$  with $A_1n_\lambda\in \elm^2$ and using  the equivalence of norms on $\bh_m$,   \re{s2.15} and \re{s2.17}, we derive that
			\begin{equation*}
				\begin{split}
					\abs{\nabla n_\lambda}^2_{0,2}+\delta h\abs{A_1n_\lambda}^2_{0,2}
					&\leq \frac{\delta h}{4}\abs{A_1n_\lambda}^2_{0,2}+\bk\abs{B_1(\bu_\lambda,n_\lambda)}_{0,2}^2+\bk\abs{B_3(\theta_m(n_\lambda),c_\lambda)}_{0,2}^2\\
					&\leq \frac{\delta h}{4}\abs{A_1n_\lambda}^2_{0,2}+\bk\abs{\bu_\lambda}_{0,\infty}^2\abs{\nabla n_\lambda}^2_{0,2}+\bk\abs{B_3(\theta_m(n_\lambda),c_\lambda)}_{0,2}^2\\
					&\leq \frac{\delta h}{4}\abs{A_1n_\lambda}^2_{0,2}+\bk(m)\mathcal{T}^1_m(\omega,\ell)\mathcal{T}_5(\omega,\ell)+\bk\abs{B_3(\theta_m(n_\lambda),c_\lambda)}_{0,2}^2.
				\end{split}
			\end{equation*} 
			Since $B_3(\theta_m(n_\lambda),c_\lambda)=-\theta_m(n_\lambda)A_1c_\lambda+\theta'_m(n_\lambda)\nabla n_\lambda\cdot\nabla c_\lambda$, we use the embedding $\h^1\hookrightarrow \elm^4$,  \re{3.3} \re{s2.17*} and \re{s2.17} to derive that
			\begin{equation*}
				\begin{split}
					\abs{B_3(\theta_m(n_\lambda),c_\lambda)}_{0,2}^2&\leq 578(m+1)^2\abs{A_1c_\lambda}^2_{0,2}+2\br_\theta\abs{\nabla n_\lambda}^2_{0,4}\abs{\nabla c_\lambda}^2_{0,4}\\
					&\leq 578(m+1)^2\abs{A_1c_\lambda}^2_{0,2}+2\br_\theta\abs{\nabla n_\lambda}^2_{0,4}\abs{ c_\lambda}^2_{2,2}\\
					&\leq \bk(m+1)^2(\bk(m)\mathcal{T}^1_m(\omega,\ell)+(m+1)^2) \mathcal{T}_3(\omega,\ell)\\
					&\qquad+2\bk (m+1)^2\abs{c^{\ell-1}}^2_{0,2}+2\bk (m+1)^2\abs{F^1_m(c^{\ell-1})g( c^{\ell-1})  \Delta^\ell\beta}^2_{0,2}\\
					&\qquad+2\br_\theta\abs{\nabla n_\lambda}^2_{0,4}\mathcal{T}_4(\omega,\ell).
				\end{split}
			\end{equation*}
			Now, we control the term $\abs{\nabla n_\lambda}^2_{0,4}$. To start, we note that since $n_\lambda\in D(A_1)$, we can obtain from \cite[Proposition 7.2, P. 404]{Taylor}  that $\abs{ n_\lambda}_{2,2}\leq \br  \abs{A_1 n_\lambda}_{0,2}+\br  \abs{ n_\lambda}_{1,2}$. By using the three dimensional Gagliardo-Nirenberg inequality, the Young inequality and \re{s2.18} we get
			\begin{equation*}
				\begin{split}
					\abs{\nabla n_\lambda}^2_{0,4}&\leq \bk \abs{\nabla n_\lambda}^\frac{1}{2}_{0,2}\abs{\nabla n_\lambda}^\frac{3}{2}_{1,2}\\
					%	&\leq \bk \abs{\nabla n_\lambda}^\frac{1}{2}_{0,2}\abs{ n_\lambda}^\frac{3}{2}_{2,2}\\
					&\leq \bk \abs{\nabla n_\lambda}^\frac{1}{2}_{0,2}\abs{A_1 n_\lambda}^\frac{3}{2}_{0,2}+\br   \abs{\nabla n_\lambda}^\frac{1}{2}_{0,2}\abs{ n_\lambda}^\frac{3}{2}_{1,2}\\
					&\leq  \frac{\delta h}{8\br_\theta \mathcal{T}_4(\omega,\ell)}\abs{A_1 n_\lambda}^2_{0,2}+\bk\abs{\nabla n_\lambda}^2_{0,2}+\br   \abs{\nabla n_\lambda}^\frac{1}{2}_{0,2}\abs{ n_\lambda}^\frac{3}{2}_{1,2}\\
					&\leq  \frac{\delta h}{8\br_\theta \mathcal{T}_4(\omega,\ell)}\abs{A_1 n_\lambda}^2_{0,2}+\bk \mathcal{T}_5(\omega,\ell)+\bk (\mathcal{T}_5(\omega,\ell))^2.
				\end{split}
			\end{equation*}
			So, in summary we have
			\begin{equation*}
				\begin{split}
					\abs{B_3(\theta_m(n_\lambda),c_\lambda)}_{0,2}^2
					&\leq 2\bk(m+1)^2(\bk(m)\mathcal{T}^1_m(\omega,\ell)+\bk(m+1)^2) \mathcal{T}_3(\omega,\ell)\\
					&\qquad+2\bk (m+1)^2\abs{c^{\ell-1}}^2_{0,2}+2\bk (m+1)^2\abs{F^1_m(c^{\ell-1})g( c^{\ell-1})  \Delta^\ell\beta}^2_{0,2}\\
					&\qquad+\frac{\delta h}{4}\abs{A_1 n_\lambda}^2_{0,2}+\bk \mathcal{T}_5(\omega,\ell)+\bk (\mathcal{T}_5(\omega,\ell))^2\\
					&:=\frac{\delta h}{4}\abs{A_1 n_\lambda}^2_{0,2}+\bk \mathcal{T}_6(\omega,\ell).
				\end{split}
			\end{equation*}
			By replacing this inequality in the first inequality after \re{s2.18}, we obtain
			\begin{equation*}
				\begin{split}
					\abs{\nabla n_\lambda}^2_{0,2}+\delta h\abs{A_1n_\lambda}^2_{0,2}
					&\leq \frac{\delta h}{2}\abs{A_1n_\lambda}^2_{0,2}+\bk(m)\mathcal{T}^1_m(\omega,\ell)\mathcal{T}_5(\omega,\ell)+\bk \mathcal{T}_6(\omega,\ell), 
				\end{split}
			\end{equation*} 
			which implies that 
			\begin{equation*}
				\abs{\nabla n_\lambda}^2_{0,2}+2\delta h\abs{A_1n_\lambda}^2_{0,2}
				\leq \bk(m)\mathcal{T}^1_m(\omega,\ell)\mathcal{T}_5(\omega,\ell)+\bk \mathcal{T}_6(\omega,\ell). 
			\end{equation*} 
			Since $\h^2\hookrightarrow \sobm^{1,4}$,  we deduce from this inequality and \re{s2.18} that 
			\begin{equation*}
				\abs{  n_\lambda}^2_{1,4}
				\leq \bk  \mathcal{T}_5(\omega,\ell)+\bk(m)\mathcal{T}^1_m(\omega,\ell)\mathcal{T}_5(\omega,\ell)+\bk \mathcal{T}_6(\omega,\ell),
			\end{equation*} 
			which along with  \re{s2.17} imply that 
			\begin{equation*}
				\abs{  c_\lambda}^2_{1,4}+\abs{  n_\lambda}^2_{1,4}
				\leq \bk  \mathcal{T}_4(\omega,\ell)+ \bk  \mathcal{T}_5(\omega,\ell)+\bk(m)\mathcal{T}^1_m(\omega,\ell)\mathcal{T}_5(\omega,\ell)+\bk \mathcal{T}_6(\omega,\ell),
			\end{equation*} 
			and the boundedness \re{s2.13} follows with $$\bk_\ell:=\sqrt{\bk  \mathcal{T}_4(\omega,\ell)+ \bk  \mathcal{T}_5(\omega,\ell)+\bk(m)\mathcal{T}^1_m(\omega,\ell)\mathcal{T}_5(\omega,\ell)+\bk \mathcal{T}_6(\omega,\ell)}.$$
			This ends the proof of Claim \ref{claim2.10}.
		\end{prev}
		
	\noindent	\textbf{Part 2: Proof of the measurability.}
		In this part,  we prove that for any $\ell\in \{0,1,2,...,N\}$,  $(\bu^\ell, c^\ell,n^\ell)$ is a random variable on the filtered  probability space $(\Omega, \mathcal{F}_{t_\ell},  \mathbb{P})$. More precisely, we prove that for any $\ell\in \{0,1,2,...,N\}$,
		\begin{equation*}
			(\bu^\ell, c^\ell,n^\ell):(\Omega, \mathcal{F}_{t_\ell})\longrightarrow \left(\bh_m\times \sobm^{1,4}\times \sobm^{1,4} ,\mathcal{B}\left(\bh_m\times \sobm^{1,4}\times \sobm^{1,4}\right)\right),
		\end{equation*} 
		is measurable. For this aim, it is sufficient to show that  one can find a Borel measurable map 
		\begin{equation*}
			\Gamma: \bh_m\times \sobm^{1,4}\times \sobm^{1,4}\times \mathbb{R}^3\times \mathbb{R}^3\longrightarrow\bh_m\times D(A_1)\times D(A_1),
		\end{equation*}
		such that for any  $\ell\in \{0,1,2,...,N\}$, $(\bu^\ell, c^\ell,n^\ell)=\varGamma(\bu^{\ell-1}, c^{\ell-1},n^{\ell-1},\Delta^\ell W,\Delta^\ell\beta)$, where $\Delta^\ell Y=(\Delta^\ell y_1,\Delta^\ell y_2,\Delta^\ell y_3)$ with $Y:=(y_1,y_2,y_3)\in \{w,\beta\}$. In fact,
		if such claim is true, then by exploiting the $\mathcal{F}_{t_\ell}$-measurability of $\Delta^\ell W$ and $\Delta^\ell \beta$, one can argue by induction and show that if $(\bu^0,c^0,n^0)$ is  $\mathcal{F}_{0}$-measurable then $\varGamma(\bu^{\ell-1}, c^{\ell-1},n^{\ell-1},\Delta^\ell W,\Delta^\ell\beta)$ is $\mathcal{F}_{t_\ell}$-measurable; hence $(\bu^\ell, c^\ell,n^\ell)$ is $\mathcal{F}_{t_\ell}$-measurable due to the fact that the composition of two measurable maps is also a measurable map. Thus, it remains to prove the existence of $\Gamma$.  For this purpose we will construct a multi-valued map which is taking value on nonempty closed subsets and that its graph is closed and apply \cite[Theorem 3.1]{Bensoussan} (see also  \cite[Theorem A.3, Appendix A]{Glatt})  to infer the existence of a such   measurable map $\Gamma$. To simplify the notation, let us set $X=\bh_m\times \sobm^{1,4}\times \sobm^{1,4}\times \mathbb{R}^3\times \mathbb{R}^3$ and define the multi-valued map 
		$$\Lambda: X\to \bh_m\times D(A_1)\times D(A_1)$$
		as follows: For any $(\bu', c',n', \eta'_W,\eta'_\beta)\in X$,  $\Lambda (\bu', c',n', \eta_W,\eta_\beta)$  is a solution of system \re{s2.3} with $\bu^{\ell-1}=\bu'$, $c^{\ell-1}=c'$, $n^{\ell-1}=n'$, $\Delta^\ell W =\eta'_w$ and $\Delta^\ell \beta=\eta'_\beta$. We note that, $\Lambda$ is a multi-valued map due to the fact that the solution of system \re{s2.3} is not unique. From the construction of  $\Lambda (\bu', c',n', \eta_W,\eta_\beta)$ we know that there exists two mappings $T_*:\sobm^{1,4}\to H_m$, $S_*:\sobm^{1,4}\times \sobm^{1,4}\to D(A_1)\times D(A_1)\hookrightarrow \sobm^{1,4}\times \sobm^{1,4}$  (see \re{s2.6*}) and  at least one element  $(\bar c,\bar n)\in \sobm^{1,4}\times \sobm^{1,4}$ such that  $\Lambda (\bu', c',n', \eta_W,\eta_\beta)=(\Upsilon_*(\bar n),S_*(\bar c,\bar n))$. This prove that $\Lambda$ is well defined as multi-valued map and also that for  any $(\bu', c',n', \eta'_W,\eta'_\beta)\in X$, $\Lambda (\bu', c',n', \eta_W,\eta_\beta)$ is  not empty. Let us fix arbitrary   $(\bu', c',n', \eta'_W,\eta'_\beta)\in X$. Our aim now is to prove that $\Lambda (\bu', c',n', \eta_W,\eta_\beta)$  is a closed subset of $X$. By construction of $\Lambda$, we note that 
		\begin{equation*}
			\Lambda (\bu', c',n', \eta'_W,\eta'_\beta)=\left\{ (\bar \bu,\bar c, \bar n)\in \bh_m\times \sobm^{1,4}\times \sobm^{1,4}:  (\bar \bu,\bar c, \bar n)= (\Upsilon_*(\bar n),S_*(\bar c,\bar n))\right\}.
		\end{equation*}
		Since  $\Upsilon_*:\sobm^{1,4}\to H_m$ and $S_*:\sobm^{1,4}\times \sobm^{1,4}\to \sobm^{1,4}\times \sobm^{1,4}$  are continuous, we derive that $\Lambda (\bu', c',n', \eta_W,\eta_\beta)$ is a closed subset of $X$ (since it  is the pre-image of the closed set ${0_X}$ by the continuous map $(I_{\bh_m},I_{\sobm^{1,4}},I_{\sobm^{1,4}})-(T,S)$, with $\Upsilon(\bar \bu,\bar c, \bar n)=\Upsilon_*(\bar n)$ and $S(\bar \bu,\bar c, \bar n)=S_*(\bar c, \bar n)$). It remains to prove that the graph of $\Lambda$ is closed in the sense given in \cite[Theorem A.3, Appendix A]{Glatt}( see also \cite[Theorem 3.1]{Bensoussan}). Let $x=(\bar \bu, \bar c,\bar n, \bar \eta_W,\bar \eta_\beta)\in X$ and $y=(\bu,  c,n)\in \bh_m\times D(A_1)\times D(A_1)$.  Let $\{x_j=(\bar  \bu^j, \bar c^j,\bar n^j,\bar  \eta^j_W,\bar \eta^j_\beta)\}_{j\geq 1}$ be a sequence of $X$ and  $\{y_j=( \bu^j,  c^j,n^j)\}_{j\geq 1}$ be a sequence of $\bh_m\times D(A_1)\times D(A_1)$ such that when $j$ goes to $\infty$,
		\begin{equation}\label{s2.19}
			x^j\longrightarrow x, \text{ in } X, \text{ and } y^j\longrightarrow y, \text{ in }\bh_m\times D(A_1)\times D(A_1).
		\end{equation}
		Let us assume that for any integer $j\geq 1$, $y_j\in \Lambda x_j$ and prove that  $y\in \Lambda x$. By definition of $\Lambda$, $y_j\in \Lambda x_j$  means that $y^j$ and $x^j$ satisfy the following system for any integer $j\geq 1$,
		\begin{equation*}
			\begin{split}
				&	\bu^j-\bar  \bu^j+[\eta A\bu^j+B^m(\bar  \bu^j,\bu^j)-B_0^m(\theta_m(n^j),\phi)]h=\alpha \pi_m (\Pl(\mathbf{f}(\bar  \bu^j) ))\Delta^\ell (\bar  \eta^j_W),\\
				&	c^j-\bar c^j+ [\eps A_1c^j +B_1(\bu^j,  c^j)] h=- B_2(\theta^{\eps_m}_m(n^j),c^j)h+\gamma F^1_m(\bar c^j)g( \bar c^j) \Delta^\ell\bar \eta^j_\beta,\\
				&n^j-\bar n^j+ [\delta A_1n^j+B_1(\bu^j,  n^j)]h= B_3(\theta_m(n^j), c^j)h, 
			\end{split}		
		\end{equation*}
		in $\bh_m$, $\elm^2$ and $\elm^2$ respectively. By using the equivalence of norms on the space $\bh_m$, the Lipschitz property of $\theta_m$ and $\theta'_m$, the property \re{3.3}, the inequality \re{s1.5},  the convergence \re{s2.19}, and the fact that each convergence sequence is bounded,  we can pass to the limit as $j$ goes to $\infty$ in this last system to derive that $y\in \Lambda x$. Then, by applying \cite[Theorem A.3, Appendix A]{Glatt}, see also \cite[Theorem 3.1]{Bensoussan},  we derive the existence of a measurable map
		\begin{equation*}
			\Gamma: (X,\mathcal{B}(X))\longrightarrow  (\bh_m\times D(A_1)\times D(A_1),\mathcal{B}(\bh_m\times D(A_1)\times D(A_1))),
		\end{equation*} 
		such that for any $x=( \bu,  c, n, \eta_W,\eta_\beta)\in X$, $ \Gamma(\bu,  c, n, \eta_W,\eta_\beta)\in \Lambda(\bu,  c, n, \eta_W,\eta_\beta)$. Let us remark that due to the continuity of the embedding $D(A_1)\hookrightarrow \sobm^{1,4}$, $\Gamma$ is also measurable from $(X,\mathcal{B}(X))$ to $(\bh_m\times \sobm^{1,4}\times \sobm^{1,4},\mathcal{B}(\bh_m\times \sobm^{1,4}\times \sobm^{1,4}))$ and the measurability of solution $\{(\bu^\ell,c^\ell,n^\ell)\}_{\ell=0}^N$ follows by an induction method as explained at the beginning.
	\end{prev}
	
Next, we are going to prove  two lemmas concerning uniform estimate of semi-discrete solutions. Before the statement of the first lemma, let us remark that for any Hilbert space $H$, we have 
	\begin{equation}\label{s2.20}
		(x-y,2x)_H=\abs{x}^2_H-\abs{y}^2_H+\abs{x-y}^2_H,\ \forall x,y\in H,
	\end{equation}
	also, for any $M\in \mathbb{N}$, 
	\begin{equation}\label{3.33}
		\left(\sum_{i=1}^Ma_i\right)^2\leq M\sum_{i=1}^Ma_i^2,\  \forall a_i\geq 0.
	\end{equation}
	Now, we prove the following result.
	\begin{lemma}\label{Lemmas2.5}
		For any $m\geq 1$, there exists a constant $\bk_m> 0$ such that for any fixed $N\in \mathbb{N}$,
		\begin{equation}\label{s2.21}
			\begin{split}
				&\max_{1\leq \ell\leq N}\abs{\bu^\ell}^2_{0,2}+\sum_{i=1}^N\abs{\bu^i-\bu^{i-1}}^2_{0,2}+h\eta \sum_{i=1}^N \abs{\nabla\bu^i}^2_{0,2}\leq \abs{\bu_{0}^m}^2_{0,2}+\bk_m\abs{\nabla \phi}^2_{0,\infty},\ \forall\omega\in \Omega,\\
				&\be\max_{1\leq \ell\leq N}  \abs{c^\ell}^2_{0,2}+\frac{1}{2} \be\sum_{i=1}^N\abs{c^i-c^{i-1}}^2_{0,2}+h\eps  \be \sum_{i=1}^N\abs{\nabla c^i}^2_{0,2}\leq \abs{c^m_0}^2_{0,2} + 18\gamma^2 m^2T,\\
				&\be\max_{1\leq \ell\leq N}  \abs{c^\ell}^4_{0,2}+\eps^2\be \left(  h\sum_{i=1}^N\abs{\nabla c^i}^2_{0,2}\right)^2\leq 2\abs{c^m_0}^2_{0,2} + 1944\gamma^4 m^4T^2,\\
				&\be \max_{1\leq \ell\leq N}\abs{\nabla c^\ell}^2_{0,2}+\frac{1}{2}\be \sum_{i=1}^N\abs{\nabla c^i-\nabla c^{i-1}}^2_{0,2}+h \be\sum_{i=1}^N\abs{A_1 c^i}^2_{0,2}\leq \abs{c^{m}_0}^2_{2,2}+\bk_m,\\
				&\be \max_{1\leq \ell\leq N}\abs{n^\ell}^2_{0,2}+\be \sum_{i=1}^N\abs{n^i-n^{i-1}}^2_{0,2}+\delta \be h\sum_{i=1}^N\abs{\nabla n^i}^2_{0,2}\leq \abs{n^m_0}^2_{0,2} +\bk_m,\\
				&\be \max_{1\leq \ell\leq N}\abs{n^\ell}^4_{0,2}+\be \left(\sum_{i=1}^N\abs{n^i-n^{i-1}}^2_{0,2}\right)^2+\delta^2 \be \left(h\sum_{i=1}^N\abs{\nabla n^i}^2_{0,2}\right)^2\leq 2\abs{n^m_0}^4_{0,2} +\bk_m.
			\end{split}
		\end{equation}
	\end{lemma}
	
	\begin{prev}[\textbf{Proof of Lemma \ref{Lemmas2.5}}]
		Let us fix arbitrary $N\in \mathbb{N}$. In order to prove \re{s2.21}$_1$, we fix $\ell \in \{1,2,...,N\}$. By taking the $\bh_m$-inner product with \re{s2.3}$_1$ and $2\bu^i$ for any $i\in \{1,...,\ell\}$ and using \re{s2.20}, we derive that 
		\begin{equation*}
\begin{split}
&\abs{\bu^i}^2_{0,2}-\abs{\bu^{i-1}}^2_{0,2}+\abs{\bu^i-\bu^{i-1}}^2_{0,2}+2h\eta  \abs{\nabla\bu^i}^2_{0,2}\\
&=\left(B_0^m(\theta_m(n^i),\phi)h+\alpha \pi_m (\Pl(\mathbf{f}(\bu^{i-1}))) \Delta^i W,2\bu^i\right).
\end{split}
		\end{equation*}
		By using \re{3.3}, the H\"older inequality, the Poincar\'e inequality and the Young inequality, we get 
		\begin{equation*}
			\begin{split}
				h(B_0^m(\theta_m(n^i),\phi),2\bu^i)&\leq 34h \valabs{\bo}^\frac{1}{2}(m+1)\abs{\nabla \phi}_{0,\infty}\abs{\bu^i}_{0,2}\\
				&\leq 34h \valabs{\bo}^\frac{1}{2}(m+1)\bk\abs{\nabla \phi}_{0,\infty}\abs{\nabla\bu^i}_{0,2}\\
				&\leq h\eta\abs{\nabla\bu^i}_{0,2}^2+ \frac{h}{\eta}\bk\valabs{\bo}(m+1)^2\abs{\nabla \phi}^2_{0,\infty}.
			\end{split}
		\end{equation*}
		By noting that $\bu^i=(u^i_1,u^i_2,u^i_3)$,  an integration-by-part and the fact that $\sum_{j=1}^3\partial_j\bu_j^i=0$ and $\bu^i|_{\partial\bo}=0$  yield
		\begin{equation}\label{s1.22}
			\begin{split}
				\alpha\left(\pi_m (\Pl(\mathbf{f}(\bu^{i-1}))) \Delta^i W,2\bu^i\right)&:=	\alpha\left(\sum_{k=1}^3  \pi_m (\nabla\bu^{i-1} e_k \Delta^i W_k),2\bu^i\right)\\
				&=2\alpha\sum_{j,k=1}^3 \Delta^i W_k\int_\bo\partial_j(u_k^{i-1}(x))u_j^i(x)dx=0\\
				&=-2\alpha\sum_{j,k=1}^3 \Delta^i W_k\int_\bo u_k^{i-1}(x)\partial_j(u_j^i(x))dx=0.
				%			&=-2\alpha\sum_{k=1}^3\Delta^i W_k\int_\bo u_k^{i-1}(x)\nabla\cdot \bu^i(x)dx=0.
			\end{split}
		\end{equation}
		From these calculations and the fact that $h\leq  T$, we derive that 
		\begin{equation*}
			\abs{\bu^i}^2_{0,2}-\abs{\bu^{i-1}}^2_{0,2}+\abs{\bu^i-\bu^{i-1}}^2_{0,2}+h\eta  \abs{\nabla\bu^i}^2_{0,2}\leq \frac{T}{\eta}\bk\valabs{\bo}(m+1)^2\abs{\nabla \phi}^2_{0,\infty}.
		\end{equation*}
		By summing this inequality from $i = 1$ to $i = \ell$, we arrive at\\
		\begin{equation*}
			\abs{\bu^\ell}^2_{0,2}+\sum_{i=1}^\ell\abs{\bu^i-\bu^{i-1}}^2_{0,2}+h\eta \sum_{i=1}^\ell \abs{\nabla\bu^i}^2_{0,2}\leq \abs{\bu_{0}^m}^2_{0,2}+\frac{T}{\eta}\bk\valabs{\bo}(m+1)^2\abs{\nabla \phi}^2_{0,\infty},
		\end{equation*}
		from which we obtain  \re{s2.21}$_1$. For the proof of \re{s2.21}$_2$, we fix $\ell \in \{1,2,...,N\}$ and recall that since $c^\ell\in D(A_1)$, it satisfies the Neumann boundary condition. We then take the $\elm^2$-inner product with \re{s2.3}$_2$ and $2c^i$ for any $i\in \{0,1,...,\ell\}$ and use \re{s2.20} and  an integration-by-part to obtain
		\begin{equation*}
			\begin{split}
				&\abs{c^i}^2_{0,2}-\abs{c^{i-1}}^2_{0,2}+\abs{c^i-c^{i-1}}^2_{0,2}+2h\eps \abs{\nabla c^i}^2_{0,2}+2h\abs{c^i\sqrt{\theta_m(n^i)+\eps_m}}^2_{0,2}\\
				%	&=2\gamma\left(\sum_{k=1}^3 \mathbf{g}_k\cdot \nabla c^{i-1} \Delta^i \beta_k,c^i\right)\\
				&=2\gamma F^1_m(c^{i-1})\left(g(c^{i-1}) \Delta^i \beta,c^i-c^{i-1}\right).
			\end{split}
		\end{equation*}
		In the last line, we have used  the following equation, which is true because $\nabla\cdot\mathbf{g}_k=0$  and $\mathbf{g}_k|_{\partial\bo}=0$,  $k=1,2,3,$,
		\begin{equation*}
			\left(g( c^{i-1}) \Delta^i \beta,c^{i-1}\right)=\frac{1}{2}\sum_{k=1}^3 \Delta^i \beta_k\left(\mathbf{g}_k\cdot \nabla( c^{i-1})^2,1\right)=\frac{1}{2}\sum_{k=1}^3 \Delta^i \beta_k\left(( c^{i-1})^2, \nabla\cdot\mathbf{g}_k\right)=0.
		\end{equation*}
		By using  $F^1_m(c^{i-1}):=F_m(\abs{c^{i-1}}_{1,2}+1)\leq \frac{m}{\abs{c^{i-1}}_{1,2}+1}$, see \re{s1.4},  $\abs{\mathbf{g}_k}_{\sobm^{1,\infty}}\leq 1$ and  the Young inequality, we obtain
		\begin{equation*}
			\begin{split}
				2\gamma F^1_m(c^{i-1})\left(g( c^{i-1}) \Delta^i \beta,c^i-c^{i-1}\right)&\leq \frac{2\gamma m}{\abs{c^{i-1}}_{1,2}+1}\sum_{k=1}^3 \abs{\mathbf{g}_k}_{\sobm^{1,\infty}} \abs{\nabla c^{i-1}}_{0,2} \valabs{\Delta^i \beta_k}\abs{c^i-c^{i-1}}_{0,2}\\
				&\leq \frac{2\gamma m}{\abs{c^{i-1}}_{1,2}+1}  \abs{\nabla c^{i-1}}_{0,2}\abs{c^i-c^{i-1}}_{0,2} \sum_{k=1}^3\valabs{\Delta^i \beta_k}\\
				&\leq \frac{1}{2}\abs{c^i-c^{i-1}}_{0,2}^2+6\gamma^2m^2\sum_{k=1}^3  \valabs{\Delta^i \beta_k}^2.
			\end{split}
		\end{equation*}
		Here, we also used the fact that $\frac{ \abs{\nabla c^{i-1}}_{0,2} }{\abs{c^{i-1}}_{1,2}+1}\leq 1$, and  \re{3.33}. By recalling that $  \valabs{\Delta^i \beta}^2=\sum_{k=1}^3\valabs{\Delta^i \beta_k}^2$,  we then derive  that 
		\begin{equation*}
			\abs{c^i}^2_{0,2}-\abs{c^{i-1}}^2_{0,2}+\frac{1}{2}\abs{c^i-c^{i-1}}^2_{0,2}+2h\eps \abs{\nabla c^i}^2_{0,2}+2h\abs{c^i\sqrt{\theta_m(n^i)+\eps_m}}^2_{0,2}\leq 6\gamma^2  m^2 \valabs{\Delta^i \beta }^2.
		\end{equation*}
		By summing this inequality from $i =1$ to $i = \ell$ and  taking the max over $\ell=1$ to $\ell=N$, we arrive at 
		\begin{equation}\label{s2.22}
			\begin{split}
				&  \max_{1\leq \ell\leq N}\abs{c^\ell}^2_{0,2}+\frac{1}{2}  \sum_{i=1}^N\abs{c^i-c^{i-1}}^2_{0,2}+2h\eps   \sum_{i=1}^N\abs{\nabla c^i}^2_{0,2}+h \ \sum_{i=1}^N\abs{c^i\sqrt{\theta_m(n^i)+\eps_m}}^2_{0,2}\\
				&\leq \abs{c^m_0}^2_{0,2} +6\gamma^2   m^2\sum_{i=1}^N  \valabs{\Delta^i \beta }^2.
			\end{split}
		\end{equation}
		By using  the inequality in \cite[Corollary 1.1]{Ichikawa},  we derive that 
		\begin{equation*}
			\be \sum_{i=1}^N  \valabs{\Delta^i \beta }^2=\sum_{i=1}^N  \be \valabs{ \beta(t_i) -\beta(t_{i-1})}^2=3Nh=3T.
		\end{equation*}
		By taking the mathematical expectation on \re{s2.22} and using this last equality we obtain   \re{s2.21}$_2$. To prove \re{s2.21}$_3$, we use  \re{3.33} and  the inequality in \cite[Corollary 1.1]{Ichikawa} and see that
		\begin{equation*}
			\be\left( \sum_{i=1}^N  \valabs{\Delta^i \beta }^2 \right)^2\leq N\be \sum_{i=1}^N  \valabs{\Delta^i \beta }^4\leq  27N^2h^2=27T^2.
		\end{equation*}
		Squaring both sides the  inequality \re{s2.22}, then taking the mathematical expectation, using \re{3.33}  and the inequality above yield \re{s2.21}$_3$. 
		
		Next, by taking the $\elm^2$-inner product of    \re{s2.3}$_2$ with  $2A_1c^i$ and integrating by parts yield 
		\begin{equation*}
			\begin{split}
				&\abs{\nabla c^i}^2_{0,2}-\abs{\nabla c^{i-1}}^2_{0,2}+\abs{\nabla c^i-\nabla c^{i-1}}^2_{0,2}+2h\eps \abs{A_1 c^i}^2_{0,2}\\
				&=-2h\left(B_1(\bu^\ell,  c^\ell)+ B_2(\theta^2_m(n^\ell),c^\ell),A_1c^i\right)+2\gamma F^1_m(c^{i-1})\left(g( c^{i-1}) \Delta^i \beta,A_1c^i-A_1c^{i-1}\right)\\
				&\qquad+2\gamma F^1_m(c^{i-1})\left(g( c^{i-1}) \Delta^i \beta,A_1c^{i-1}\right)\\
				&=I_1+I_2+I_3.
			\end{split}
		\end{equation*}
		By using the equivalence of norms on $\bh_m$ and \re{3.3}, the Young inequality, we obtain
		\begin{equation*}
			\begin{split}
				I_1&\leq \frac{h\eps}{2}\abs{A_1c^i}^2_{0,2}+\frac{4h}{\eps}\abs{B_1(\bu^\ell,  c^\ell)}^2_{0,2}+\frac{4h}{\eps}\abs{B_2(\theta^2_m(n^\ell),c^\ell)}^2_{0,2}\\
				&\leq \frac{h\eps}{2}\abs{A_1c^i}^2_{0,2}+\frac{4h}{\eps}\abs{\bu^i}_{0,\infty}^2\abs{\nabla c^i}^2_{0,2}+\frac{8h}{\eps}\left[578(m+1)^2+\frac{81}{m^2}\right]\abs{c^i}^2_{0,2}\\
				&\leq \frac{h\eps}{2}\abs{A_1c^i}^2_{0,2}+\frac{4h\bk(m)}{\eps}\abs{\bu^i}_{0,2}^2\abs{\nabla c^i}^2_{0,2}+\frac{8h}{\eps}\left[578(m+1)^2+\frac{81}{m^2}\right]\abs{c^i}^2_{0,2}.
			\end{split}
		\end{equation*}
		Since $\bo$ is a convex, we have
		\begin{equation}\label{3,37}
			\abs{\nabla^2 c^{i-1}}^2_{0,2}\leq \abs{\Delta c^{i-1}}^2_{0,2}.
		\end{equation}
		Since  $\abs{\mathbf{g}_k}_{\sobm^{1,\infty}}\leq 1$ and $\frac{ \abs{\nabla c^{i-1}}_{0,2}}{\abs{c^{i-1}}_{1,2}+1} <1$,  an integration-by-part, the Young inequality and  \re{s1.4} yield
		\begin{equation*}
			\begin{split}
				I_2&\leq \frac{1}{2}\abs{\nabla c^i-\nabla c^{i-1}}^2_{0,2}+6\gamma^2 (F^1_m(c^{i-1}))^2\sum_{k=1}^3 \abs{\nabla\left(\mathbf{g}_k\cdot \nabla c^{i-1}\right)}^2_{0,2} \valabs{\Delta^i \beta_k}^2\\
				&\leq \frac{1}{2}\abs{\nabla c^i-\nabla c^{i-1}}^2_{0,2}+\frac{12\gamma^2 m^2}{(\abs{c^{i-1}}_{1,2}+1)^2}\sum_{k=1}^3\abs{\nabla\mathbf{g}_k}^2_{0,\infty} \abs{ \nabla c^{i-1}}^2_{0,2} \valabs{\Delta^i \beta_k}^2\\
				&\qquad+12\gamma^2\sum_{k=1}^3\abs{\mathbf{g}_k}^2_{0,\infty} \abs{ \nabla^2 c^{i-1}}^2_{0,2} \valabs{\Delta^i \beta_k}^2\\
				&\leq \frac{1}{2}\abs{\nabla c^i-\nabla c^{i-1}}^2_{0,2}+12\gamma^2  m^2\valabs{\Delta^i \beta}^2 +12\gamma^2 \abs{ A_1 c^{i-1}}^2_{0,2} \valabs{\Delta^i \beta}^2.
			\end{split}
		\end{equation*}
		Hence, we have
		\begin{equation*}
			\begin{split}
				&\abs{\nabla c^i}^2_{0,2}-\abs{\nabla c^{i-1}}^2_{0,2}+\frac{1}{2}\abs{\nabla c^i-\nabla c^{i-1}}^2_{0,2}+\frac{3h\eps}{2} \abs{A_1 c^i}^2_{0,2}\\
				&\leq\frac{4h\bk(m)}{\eps}\abs{\bu^i}_{0,2}^2\abs{\nabla c^i}^2_{0,2}+\frac{8h}{\eps}\left[578(m+1)^2+\frac{81}{m^2}\right]\abs{c^i}^2_{0,2}+12m^2\gamma^2  \valabs{\Delta^i \beta}^2 \\
				&\qquad+12\gamma^2  \abs{ A_1 c^{i-1}}^2_{0,2} \valabs{\Delta^i \beta}^2+2\gamma F^1_m(c^{i-1})\left(\sum_{k=1}^3\Delta^i \beta_k\nabla \left(\mathbf{g}_k\cdot \nabla c^{i-1} \right),\nabla c^{i-1}\right).
			\end{split}
		\end{equation*}
		By summing up this inequality from $1$ to $\ell$ and taking the maximum over $\ell$,  and by using \re{s2.21}$_1$, we get
		\begin{equation}\label{s2.23}
			\begin{split}
				&\max_{1\leq \ell\leq N}\abs{\nabla c^\ell}^2_{0,2}+\frac{1}{2}\sum_{i=1}^N\abs{\nabla c^i-\nabla c^{i-1}}^2_{0,2}+\frac{3h\eps}{2} \sum_{i=1}^N\abs{A_1 c^i}^2_{0,2}\\
				&\leq \abs{\nabla c^{m}_0}^2_{0,2}+\bk(m) h\sum_{i=1}^N\abs{\nabla c^i}^2_{0,2}+\bk h\sum_{i=1}^N\abs{c^i}^2_{0,2}+12\gamma^2m^2  \sum_{i=1}^N\valabs{\Delta^i \beta}^2 \\
				&\qquad+12\gamma^2  \sum_{i=1}^N\abs{ A_1 c^{i-1}}^2_{0,2} \valabs{\Delta^i \beta}^2\\
				&\qquad+2\gamma\max_{1\leq \ell\leq N} \sum_{i=1}^\ell \left(F^1_m(c^{i-1})\sum_{k=1}^3\Delta^i \beta_k\nabla (\mathbf{g}_k\cdot \nabla c^{i-1} ),\nabla c^{i-1}\right).
			\end{split}
		\end{equation}
		Since the last term in \re{s2.23} is the stochastic integral,	we use the Burkholder–Davis–Gundy inequality and obtain
		\begin{equation*}
			\begin{split}
				&2\gamma\be \max_{1\leq \ell\leq N} \sum_{i=1}^\ell\left(F^1_m(c^{i-1})\sum_{k=1}^3\Delta^i \beta_k\nabla \left(\mathbf{g}_k\cdot \nabla c^{i-1} \right),\nabla c^{i-1}\right)\\
				&\leq  2\gamma\bk\be \left(h\sum_{i=1}^N(F^1_m(c^{i-1}))^2\sum_{k=1}^3\abs{\nabla \left(\mathbf{g}_k\cdot \nabla c^{i-1} \right)}^2_{0,2}\abs{\nabla c^{i-1}}^2_{0,2}\right)^\frac{1}{2}\\
				&\leq  2\gamma\bk \be \left(h\sum_{i=1}^N\sum_{k=1}^3\abs{\nabla \left(\mathbf{g}_k\cdot \nabla c^{i-1} \right)}^2_{0,2}\right)^\frac{1}{2},
			\end{split}
		\end{equation*}
		where $(F^1_m(c^{i-1}))^2\abs{\nabla c^{i-1}}^2_{0,2}<1$ was used in the last line.  By using  $\abs{\mathbf{g}_k}_{\sobm^{1,\infty}}\leq 1$, \re{3.37}, the H\"older  and  the Young inequalities we see that
		\begin{equation*}
			\begin{split}
				&2\gamma\be \max_{1\leq \ell\leq N} \sum_{i=1}^\ell\left(F^1_m(c^{i-1})\sum_{k=1}^3\Delta^i \beta_k\nabla \left(\mathbf{g}_k\cdot \nabla c^{i-1} \right),\nabla c^{i-1}\right)\\
				&\leq  2\gamma\bk \be \left(3h\sum_{i=1}^N\abs{\nabla c^{i-1}}^2_{0,2}\right)^\frac{1}{2}+2\gamma\bk \be \left(3h\sum_{i=1}^N\abs{\nabla^2 c^{i-1}}^2_{0,2}\right)^\frac{1}{2}\\
				&\leq  2\gamma\bk  \left(3h\be\sum_{i=1}^N\abs{\nabla c^{i-1}}^2_{0,2}\right)^\frac{1}{2}+2\gamma\bk  \left(3h\be\sum_{i=1}^N\abs{\nabla^2 c^{i-1}}^2_{0,2}\right)^\frac{1}{2}\\
				&\leq  12\gamma^2  h\be\sum_{i=1}^N\abs{\nabla c^{i-1}}^2_{0,2}+12\gamma^2  h\be\sum_{i=1}^N\abs{A_1 c^{i-1}}^2_{0,2}+2\bk^2\\
				&\leq  12\gamma^2  h\be\sum_{i=1}^N\abs{\nabla c^{i}}^2_{0,2}+12\gamma^2  h\be\sum_{i=1}^N\abs{A_1 c^{i}}^2_{0,2}+2\bk^2+12\gamma^2  T\abs{c^m_0}^2_{0,2}+12\gamma^2  T\abs{A_1c^m_0}^2_{0,2}.
			\end{split}
		\end{equation*}
		From the $\mathcal{F}_{t_{\ell-1}}$-measurability of $c^{\ell-1}$, the tower property of the conditional mathematical expectation, the independence of the increments of the Wiener process $\beta$ and the inequality in \cite[Corollary 1.1]{Ichikawa},  we derive that 
		\begin{equation*}
			\begin{split}
				12\gamma^2 \be  \sum_{i=1}^N\abs{ A_1 c^{i-1}}^2_{0,2} \valabs{\Delta^i \beta}^2&=12\gamma^2  \sum_{i=1}^N \be\left[\be\left( \abs{A_1c^{i-1}}^2_{0,2} \valabs{\Delta^i \beta }^2\right)| \mathcal{F}_{t_{\ell-1}}\right]\\
				&=12\gamma^2  \sum_{i=1}^N \be  \abs{A_1 c^{i-1}}^2_{0,2} \be  \valabs{ \beta(t_i)-\beta(t_{i-1}) }^2 =12h\gamma^2  \sum_{i=1}^N\be  \abs{A_1c^{i-1}}^2_{0,2}\\
				&\leq 12h\gamma^2 \be \sum_{i=1}^N  \abs{A_1c^{i}}^2_{0,2}+12\gamma^2  T\abs{A_1c^m_0}^2_{0,2}.
			\end{split}
		\end{equation*}
		In the last line, we used the fact that $h\leq T$. Now,  taking  the expectation  in \re{s2.23} and using the above inequalities imply
		\begin{equation*}
			\begin{split}
				&\be \max_{1\leq \ell\leq N}\abs{\nabla c^\ell}^2_{0,2}+\frac{1}{2}\be \sum_{i=1}^N\abs{\nabla c^i-\nabla c^{i-1}}^2_{0,2}+\left(\frac{3\eps}{2}-24\gamma^2\right)h \be\sum_{i=1}^N\abs{A_1 c^i}^2_{0,2}\\
				&\leq \abs{\nabla c^{m}_0}^2_{0,2}+12\gamma^2  T\abs{c^m_0}^2_{0,2}+24\gamma^2  T\abs{A_1c^m_0}^2_{0,2}+\bk(m) h\be \sum_{i=1}^N\abs{\nabla c^i}^2_{0,2} \\
				&\qquad+\bk h\be \sum_{i=1}^N\abs{c^i}^2_{0,2}+12\gamma^2m^2  \sum_{i=1}^N\be\valabs{\Delta^i \beta}^2\\
				&\leq \abs{\nabla c^{m}_0}^2_{0,2}+12\gamma^2  T\abs{c^m_0}^2_{0,2}+24\gamma^2  T\abs{A_1c^m_0}^2_{0,2}+\bk(m) h\be \sum_{i=1}^N\abs{\nabla c^i}^2_{0,2} \\
				&\qquad+\bk Nh\be \max_{1\leq \ell\leq N}\abs{c^i}^2_{0,2}+12\gamma^2m^2  Nh.
			\end{split}
		\end{equation*}
		Since $Nh=T$ and $\gamma^2\leq \frac{3\eps}{1452}$, we use this inequality and \re{s2.21}$_{2}$ and obtain \re{s2.21}$_4$. 
		
		Next, we prove  \re{s2.21}$_5$.  We fix $\ell \in \{1,2,...,N\}$ and   take the $\elm^2$-inner product with \re{s2.3}$_3$ and $2n^i$ for any $i\in \{1,...,\ell\}$ and use \re{s2.20}, \re{3.3} and  an integration-by-part to obtain
		\begin{equation*}
			\begin{split}
				\abs{n^i}^2_{0,2}-\abs{n^{i-1}}^2_{0,2}+\abs{n^i-n^{i-1}}^2_{0,2}+2h\delta \abs{\nabla n^i}^2_{0,2}&=2h\int_\bo\theta_m(n^i(x))\nabla c^i(x)\cdot\nabla n^i(x)dx\\
				&\leq 34h(m+1)\abs{\nabla c^i}_{0,2}\abs{\nabla n^i}_{0,2}\\
				&\leq h\delta \abs{\nabla n^i}_{0,2}^2+\frac{289(m+1)^2}{\delta}h\abs{\nabla c^i}^2_{0,2}.
			\end{split}
		\end{equation*}
		By summing this inequality from $i=1$ to $i=\ell$ and taking the maximum over $\ell$ we get
		\begin{equation*}
			\max_{1\leq \ell\leq N}\abs{n^\ell}^2_{0,2}+\sum_{i=1}^N\abs{n^i-n^{i-1}}^2_{0,2}+h\delta \sum_{i=1}^N\abs{\nabla n^i}^2_{0,2}\leq \abs{n^m_0}^2_{0,2} +\frac{289(m+1)^2}{\delta}h\sum_{i=1}^N\abs{\nabla c^i}^2_{0,2}.
		\end{equation*}
		We take the mathematical expectation on  this inequality and use  \re{s2.21}$_{2}$ to obtain \re{s2.21}$_5$. For the proof of   \re{s2.21}$_6,$  we square the above  inequality, take the mathematical expectation and use the fact that $(a+b)^2\leq 2a^2+2b^2$ and \re{s2.21}$_4$ to get \re{s2.21}$_6$  and and the proof of Lemma \ref{Lemmas2.5}.
	\end{prev}
	\begin{lemma}\label{Lemmas2.6}
		For any $m\geq 1$, there exists a constant $\bk_m> 0$ such that for any  $j\in \{1,...,N\}$,
		\begin{equation*}
			\begin{split}
				&\be h\sum_{\ell=0}^{N-j}\abs{\bu^{\ell+j}-\bu^\ell}^4_{0,2}\leq \bk_m t^2_j,\
				\be h\sum_{\ell=0}^{N-j}\abs{c^{\ell+j}-c^\ell}^4_{(\h^1)'}\leq \bk_m t^2_j,\text{ and }\\
				&\hspace{3cm}\be h\sum_{\ell=0}^{N-j}\abs{n^{\ell+j}-n^\ell}^4_{(\h^1)'}\leq \bk_m t^2_j.
			\end{split}
		\end{equation*}
	\end{lemma}
	
	\begin{prev}[\textbf{Proof of Lemma \ref{Lemmas2.6}}]
		Let us fix an arbitrary $j\in \{1,...,N\}$ and choose $\ell\in \{1,...,N-j\}$. Then, summing  \re{s2.3}$_1$ from $i=\ell+1$ to $i=\ell+j$ gives
		\begin{equation*}
			\frac{1}{h}(\bu^{\ell+j}-\bu^{\ell})+\sum_{i=\ell+1}^{\ell+j}[\eta A\bu^i+B^m(\bu^{i-1},\bu^i)-B_0^m(\theta_m(n^i),\phi)]=\frac{\alpha}{h}\sum_{i=\ell+1}^{\ell+j}\sum_{k=1}^3  \pi_m (\nabla\bu^{i-1} e_k \Delta^i W_k).
		\end{equation*}
		By taking  the $\el^2$-inner product of the above identity with $\bu^{\ell+j}-\bu^{\ell}$, then raising to the power 2 the resulting equation and summing from $\ell = 0$ to $\ell = N-j$, we obtain
		\begin{equation*}
			\begin{split}
				h\sum_{\ell=0}^{N-j}\abs{\bu^{\ell+j}-\bu^\ell}^4_{0,2}&\leq\bk  h^3\sum_{\ell=0}^{N-j} \valabs{\sum_{i=1}^{j} (\nabla\bu^{\ell+i}, \nabla( \bu^{\ell+j}-\bu^\ell))}^2\\
				&\qquad+\bk  h^3\sum_{\ell=0}^{N-j}\valabs{\sum_{i= 1}^{ j}(B(\bu^{\ell+i-1},\bu^{\ell+i}),\bu^{\ell+j}-\bu^\ell)}^2\\
				&\qquad+\bk  h^3\sum_{\ell=0}^{N-j}\valabs{\sum_{i=1}^{j}(B_0(\theta_m(n^{\ell+i}),\phi),\bu^{\ell+j}-\bu^\ell)}^2\\
				&=II_1+II_2+II_3.
			\end{split}
		\end{equation*}
		Here, we have used \re{3.33}  and  $\left(\sum_{k=1}^3  \nabla\bu^{i-1} e_k \Delta^i W_k,\bu^{\ell+j}-\bu^\ell\right)=0$ because $\nabla\cdot (\bu^{\ell+j}-\bu^\ell)=0$. By using  the H\"older inequality, the equivalence of norms on $\bh_m$, the Young inequality and \re{s2.21}$_1$, we see that
		\begin{equation*}
			\begin{split}
				II_1&\leq \bk_m j h^3\sum_{\ell=0}^{N-j} \sum_{i=1}^{j} \abs{\bu^{\ell+i}}_{0,2}^2 \abs{\bu^{\ell+j}-\bu^\ell}_{0,2}^2\\
				&\leq j^2 h^4 \bk_mT\sum_{\ell=0}^{N-j} \sum_{i=1}^{j} \abs{\bu^{\ell+i}}_{0,2}^4+\frac{1}{6T} j h^2\sum_{\ell=0}^{N-j} \abs{\bu^{\ell+j}-\bu^\ell}_{0,2}^4\\
				&\leq t_j^2 \bk_mT^3\max_{0\leq \ell\leq N} \abs{\bu^{\ell}}_{0,2}^4+\frac{1}{6} h\sum_{\ell=0}^{N-j} \abs{\bu^{\ell+j}-\bu^\ell}_{0,2}^4\\
				&\leq t_j^2 \bk_m+\frac{1}{6} h\sum_{\ell=0}^{N-j} \abs{\bu^{\ell+j}-\bu^\ell}_{0,2}^4.
			\end{split}
		\end{equation*}
		Here, we used the fact that $j(N-j+1) h^2=t_jt_{N-j+1}\leq T^2$ and $jh=t_j\leq T$. In a very similar way, we can prove that 
		\begin{equation*}
			\begin{split}
				II_2&\leq t_j^2 \bk_mT^3\max_{0\leq \ell\leq N} \abs{\bu^{\ell}}_{0,2}^8+\frac{1}{6} h\sum_{\ell=0}^{N-j} \abs{\bu^{\ell+j}-\bu^\ell}_{0,2}^4\\
				&\leq t_j^2 \bk_m+\frac{1}{6} h\sum_{\ell=0}^{N-j} \abs{\bu^{\ell+j}-\bu^\ell}_{0,2}^4.
			\end{split}
		\end{equation*}
		By using \re{3.3}, we get
		\begin{equation*}
			\begin{split}
				II_3&\leq \bk jh^3\sum_{\ell=0}^{N-j} \sum_{i=1}^{j} \abs{B_0(\theta_m(n^{\ell+i}),\phi)}_{0,2}^2 \abs{\bu^{\ell+j}-\bu^\ell}_{0,2}^2\\
				&\leq j^2h^4 \bk T\sum_{\ell=0}^{N-j} \sum_{i=1}^{j} \abs{B_0(\theta_m(n^{\ell+i}),\phi)}_{0,2}^4+\frac{1}{6T} jh^2\sum_{\ell=0}^{N-j} \abs{\bu^{\ell+j}-\bu^\ell}_{0,2}^4\\
				&\leq t_j^2 \bk h^2j(N-j+1)T(m+1)^4\abs{\nabla\phi}^4_{0,\infty}+\frac{1}{6} h\sum_{\ell=0}^{N-j} \abs{\bu^{\ell+j}-\bu^\ell}_{0,2}^4\\
				&\leq t_j^2 T^3\bk_m+\frac{1}{6} h\sum_{\ell=0}^{N-j} \abs{\bu^{\ell+j}-\bu^\ell}_{0,2}^4.
			\end{split}
		\end{equation*}
		From the above inequalities, we derive that $\frac{1}{2} h\sum_{\ell=0}^{N-j} \abs{\bu^{\ell+j}-\bu^\ell}_{0,2}^4\leq t_j^2 \bk_m$ and by taking the expectation on this inequality, we obtain the first inequality of Lemma \ref{Lemmas2.6}. 
		
		Now, we proceed to the proof of the second inequality. Summing \re{s2.3}$_2$ from $i=\ell+1$ to $i=\ell+j$ and using an integration by parts and \re{3.3} imply for any  $\psi\in \h^1$,
		\begin{equation*}
			\begin{split}
				\frac{1}{h}(c^{\ell+j}-c^{\ell},\psi)&\leq\sum_{i=1}^{j}\left[\eps \abs{\nabla c^{\ell+i}}_{0,2}+\abs{\bu^{\ell+i}}_{0,\infty}\abs{\nabla c^{\ell+i}}_{0,2} +(m+\eps_m+1)\abs{c^{\ell+i}}_{0,2}\right]\abs{\psi}_{1,2}\\
				&\qquad+\frac{\gamma}{h}\abs{\sum_{i=1}^{j} F^1_m(c^{\ell+i-1})g(c^{\ell+i-1})  \Delta^{\ell+i}\beta}_{0,2}\abs{\psi}_{0,2}.
			\end{split}
		\end{equation*}
		From this estimate and using   the equivalence of norms on $\bh_m$ and \re{s2.21}$_1$  we infer that 
		\begin{equation*}
			\begin{split}
				\be h\sum_{\ell=0}^{N-j}\abs{c^{\ell+j}-c^\ell}^4_{(H^1)'}&\leq \bk_m  h^5\be \sum_{\ell=0}^{N-j}\left(\sum_{i=1}^{j}\abs{\nabla c^{\ell+i}}_{0,2}\right)^4 +\bk  h^5 \be\sum_{\ell=0}^{N-j}\left(\sum_{i=1}^{j}\abs{c^{\ell+i}}_{0,2}\right)^4\\
				&\qquad+\gamma^4h\sum_{\ell=0}^{N-j}\be\abs{\sum_{i=1}^{j} F^1_m(c^{\ell+i-1})g( c^{\ell+i-1})  \Delta^{\ell+i}\beta}_{0,2}^4.
			\end{split}
		\end{equation*}
		Now, by	using  $(\sum_{i=1}^j a_i)^2\leq j\sum_{i=1}^j a_i^2$ and  \re{s2.21}$_3$ we get
		\begin{equation*}
			\begin{split}
				\bk_m  h^5\be \sum_{\ell=0}^{N-j}\left(\sum_{i=1}^{j}\abs{\nabla c^{\ell+i}}_{0,2}\right)^4 &\leq \bk_m h^5\be \sum_{\ell=0}^{N-j}\left(j\sum_{i=1}^{j}\abs{\nabla c^{\ell+i}}^2_{0,2}\right)^2\\
				&\leq \bk_m j^2 h^3 \sum_{\ell=0}^{N-j}\be\left(h\sum_{i=1}^{N}\abs{\nabla c^{i}}^2_{0,2}\right)^2\\
				&\leq \bk_m t^2_jh(N-j+1)\leq \bk_mT t^2_j,
			\end{split}
		\end{equation*}
		and 
		\begin{equation*}
			\bk  h^5 \be\sum_{\ell=0}^{N-j}\left(\sum_{i=1}^{j}\abs{c^{\ell+i}}_{0,2}\right)^4\leq \bk  h^5 \be\sum_{\ell=0}^{N-j}j^4\max_{1\leq i\leq N}\abs{c^{i}}^4_{0,2}\leq \bk_m t^2_j h^3j^2(N-j+1)\leq \bk_mT^3 t^2_j.
		\end{equation*}
		For the stochastic term, we write it as a stochastic integral of a piecewise constant integral, apply  the Burkholder–Davis–Gundy inequality, use the fact that $\abs{\mathbf{g}}_{\sobm^{1,\infty}} \leq 1$ and $F^1_m(c^{\ell+i-1})\abs{\nabla c^{\ell+i-1}}_{0,2}\leq m$ and we obtain
		\begin{equation*}
			\begin{split}
				\gamma^4h\sum_{\ell=0}^{N-j}\be\abs{\sum_{i=1}^{j} F^1_m(c^{\ell+i-1})g( c^{\ell+i-1})  \Delta^{\ell+i}\beta}_{0,2}^4&\leq \bk h\sum_{\ell=0}^{N-j}\be\left(h\sum_{i=1}^{j}(F^1_m(c^{\ell+i-1}))^2\abs{\nabla c^{\ell+i-1}}^2_{0,2} \right)^2 \\
				&\leq \bk m^4 j^2 h^3(N-j+1)\\
				&\leq \bk m^4t_j^2T.
			\end{split}
		\end{equation*}
		By combining the four last inequalities, we obtain the second inequality of Lemma \ref{Lemmas2.6}. 
		
		For the third inequality,  by summing \re{s2.3}$_3$ from $i=\ell+1$ to $i=\ell+j$ and using an integration by parts and \re{3.3}, we deduce that for any  $\psi\in \h^1$,
		\begin{equation*}
			\frac{1}{h}(n^{\ell+j}-n^{\ell},\psi)\leq\sum_{i=1}^{j}\left[\delta \abs{\nabla n^{\ell+i}}_{0,2}+\abs{\bu^{\ell+i}}_{0,\infty}\abs{\nabla n^{\ell+i}}_{0,2} +(m+1)\abs{\nabla c^{\ell+i}}_{0,2}\right]\abs{\psi}_{1,2},
		\end{equation*}
		from which the equivalence of norms on $\bh_m$,  \re{s2.21}$_1$,  \re{s2.21}$_3$ and \re{s2.21}$_6$  we  infer  that 
		\begin{equation*}
			\begin{split}
				\be h\sum_{\ell=0}^{N-j}\abs{n^{\ell+j}-n^\ell}^4_{(H^1)'}&\leq \bk_m h^5\be \sum_{\ell=0}^{N-j}\left(\sum_{i=1}^{j}\abs{\nabla n^{\ell+i}}_{0,2}\right)^4 +\bk h^5 \be\sum_{\ell=0}^{N-j}\left(\sum_{i=1}^{j}\abs{\nabla c^{\ell+i}}_{0,2}\right)^4\\
				&\leq \bk_m j^2h^3 \sum_{\ell=0}^{N-j}\be\left(h\sum_{i=1}^{N}\abs{\nabla n^{i}}^2_{0,2}\right)^2+\bk_m j^2h^3 \sum_{\ell=0}^{N-j}\be\left(h\sum_{i=1}^{N}\abs{\nabla c^{i}}^2_{0,2}\right)^2\\
				&\leq \bk_m t^2_jh(N-j+1)\leq \bk_mT t^2_j.
			\end{split}
		\end{equation*}
		This ends the proof of Lemma \ref{Lemmas2.6}.
	\end{prev}	 
	
	\section{Construction of the interpolants and proof of the main result} 
	In this section,  we will study the compactness of some interpolants of the sequences of the  random variables $\{\bu^\ell\}^N_{\ell=1}$, $\{c^\ell\}^N_{\ell=1}$ and $\{n^\ell\}^N_{\ell=1}$ constructed in the previous section.  More precisely, we define the  following piecewise continuous processes $\bu_N:[0,T]\times \Omega\to \bh_m$, $c_N:[0,T]\times \Omega\to D(A_1)$ and $n_N:[0,T]\times \Omega\to D(A_1)$ respectively by: For $t\in [0,T]$,
	\begin{equation*}
		\begin{split}
			&\bu_N(t)=\sum_{\ell=1}^N\left(\bu^{\ell-1}+\frac{\bu^\ell-\bu^{\ell-1}}{h}\left(t-t_{\ell-1}\right)\right)1_{[t_{\ell-1},t_\ell]}(t),\\
			&c_N(t)=\sum_{\ell=1}^N\left(c^{\ell-1}+\frac{c^\ell-c^{\ell-1}}{h}\left(t-t_{\ell-1}\right)\right)1_{[t_{\ell-1},t_\ell]}(t),\text{ and }\\
			&n_N(t)=\sum_{\ell=1}^N\left(n^{\ell-1}+\frac{n^\ell-n^{\ell-1}}{h}\left(t-t_{\ell-1}\right)\right)1_{[t_{\ell-1},t_\ell]}(t),
		\end{split}
	\end{equation*}
	and their piecewise constant extensions defined by
	\begin{equation*}
		\begin{split}
			&\hat{\bu}_N(t)=\sum_{\ell=1}^N\bu^{\ell}1_{(t_{\ell-1},t_\ell]}(t),\qquad\hat{c}_N(t)=\sum_{\ell=1}^Nc^{\ell}1_{(t_{\ell-1},t_\ell]}(t),\qquad\hat{n}_N(t)=\sum_{\ell=1}^Nn^{\ell}1_{(t_{\ell-1},t_\ell]}(t),\\
			&\hspace{2.5cm}\check{\bu}_N(t)=\sum_{\ell=1}^N\bu^{\ell-1}1_{[t_{\ell-1},t_\ell)}(t), \text{ and }\check{c}_N(t)=\sum_{\ell=1}^Nc^{\ell-1}1_{[t_{\ell-1},t_\ell)}(t).%,\text{ and } \check{n}_N(t)=\sum_{\ell=1}^Nn^{\ell-1}1_{[t_{\ell-1},t_\ell)}(t).
		\end{split}
	\end{equation*}
	We note that   the paths of the processes $\hat{\bu}_N$, $\hat{c}_N$, $\hat{n}_N$,  $\check{\bu}_N$ and  $\check{c}_N$   are not  continuous. The processes  $\check{\bu}_N$, $\check{c}_N$ and $\check{n}_N$ are $(\mathcal{F}_t)_{t\in [0,T]}$-adapted, but $\hat{\bu}_N$ and $\hat{c}_N$  are not.  The following result is a corollary of Lemma \ref{Lemmas2.5}.
	\begin{lemma}\label{Lemmas2.7}
		Let  $m\geq 1$ be a fixed integer. For any $p\in [1,2]$, there exists a constant $\bk_m>0$ such that for any $N\in \mathbb{N}$,
		\begin{equation}\label{s2.25}
			\begin{split}
				&\be \sup_{0\leq s\leq T}\abs{\bu_N(s)}_{0,2}^{2p}+\be \sup_{0\leq s\leq T}\abs{\hat{\bu}_N(s)}_{0,2}^{2p}+\be \sup_{0\leq s\leq T}\abs{\check{\bu}_N(s)}_{0,2}^{2p}\leq \bk_m,\\
				&\be \sup_{0\leq s\leq T}\abs{c_N(s)}_{0,2}^{2p}+\be \sup_{0\leq s\leq T}\abs{\hat{c}_N(s)}_{0,2}^{2p}+\be \sup_{0\leq s\leq T}\abs{\check{c}_N(s)}_{0,2}^{2p}\leq \bk_m,\\
				&\be \sup_{0\leq s\leq T}\abs{c_N(s)}_{1,2}^{2}+\be \sup_{0\leq s\leq T}\abs{\hat{c}_N(s)}_{1,2}^{2}+\be \sup_{0\leq s\leq T}\abs{\check{c}_N(s)}_{1,2}^{2}\leq \bk_m,\\
				&\be \sup_{0\leq s\leq T}\abs{n_N(s)}_{0,2}^{2p}+\be \sup_{0\leq s\leq T}\abs{\hat{n}_N(s)}_{0,2}^{2p}\leq \bk_m,
			\end{split}
		\end{equation}
		and 
		\begin{equation}\label{s2.26}
			\begin{split}
				&\be\int_0^T\abs{A_1c_N(s)}_{0,2}^2ds+\be\int_0^T\abs{A_1\hat{c}_N(s)}_{0,2}^2ds+\be\int_0^T\abs{A_1\check{c}_N(s)}_{0,2}^2ds\leq \bk_m,\\
				&\be\left(\int_0^T\abs{c_N(s)}_{1,2}^2ds\right)^{p}+\be\left(\int_0^T\abs{\hat{c}_N(s)}_{1,2}^2ds\right)^{p}+\be\left(\int_0^T\abs{\check{c}_N(s)}_{1,2}^2ds\right)^{p}\leq \bk_m,\\
				&\be\left(\int_0^T\abs{\bu_N(s)}_{0,2}^2ds\right)^{p}+\be\left(\int_0^T\abs{n_N(s)}_{1,2}^2ds\right)^{p}+\be\left(\int_0^T\abs{\hat{n}_N(s)}_{1,2}^2ds\right)^{p}\leq \bk_m.
			\end{split}
		\end{equation}
	\end{lemma}
	The proof of Lemma \ref{Lemmas2.7} uses Lemma \ref{Lemmas2.5} and is very similar to  \cite[Proof of Lemma 4.1]{Deugoue}. Thus we omit the proof.  We also state the following proof, because its uses Lemma \ref{Lemmas2.5} and very similar to \cite[Proposition 3.11]{Razafimandimby}.
	\begin{lemma}\label{lemmas2.8}
		We have,
		\begin{equation*}
			\begin{split}
				&\lim_{N\longrightarrow\infty}\be\int_0^T\abs{\bu_N(s)-\hat{\bu}_N(s)}_{0,2}^2ds+\lim_{N\longrightarrow\infty}\be\int_0^T\abs{\bu_N(s)-\check{\bu}_N(s)}_{0,2}^2ds=0,\\
				&\lim_{N\longrightarrow\infty}\be\int_0^T\abs{c_N(s)-\hat{c}_N(s)}_{1,2}^2ds+\lim_{N\longrightarrow\infty}\be\int_0^T\abs{c_N(s)-\check{c}_N(s)}_{1,2}^2ds=0,\\
				&\hspace{3cm}\lim_{N\longrightarrow\infty}\be\int_0^T\abs{n_N(s)-\hat{n}_N(s)}_{0,2}^2ds=0.
			\end{split}
		\end{equation*} 
	\end{lemma}
	Now, we state and prove the following crucial result.
	\begin{lemma}\label{lemmas2.9}
		For any integer $m\geq 1$, there exists a constant $\bk_m>0$ such that  for any integer $N\in \mathbb{N}$,
		\begin{equation*}
			\begin{split}
				&\be\sup_{0<\theta\leq T}\theta^{-\frac{1}{2}}\left(\int_\theta^{T-\theta}\abs{\bu_N(s+\theta)-\bu_N(s)}_{0,2}^4ds\right)^\frac{1}{4}\leq \bk_m,\\
				& \be\sup_{0<\theta\leq T}\theta^{-\frac{1}{2}}\left(\int_\theta^{T-\theta}\abs{c_N(s+\theta)-c_N(s)}_{(H^1)'}^4ds\right)^\frac{1}{4}\leq \bk_m,\\
				& \be\sup_{0<\theta\leq T}\theta^{-\frac{1}{2}}\left(\int_\theta^{T-\theta}\abs{n_N(s+\theta)-n_N(s)}_{(H^1)'}^4ds\right)^\frac{1}{4}\leq \bk_m.
			\end{split}
		\end{equation*}
	\end{lemma}
	\begin{prev}[\textbf{Proof of Lemma \ref{lemmas2.9}}]
		We start by noting that the idea of the proof is coming from \cite[Proof of Lemma 3.2]{Banas}.  We will prove only the first inequality since the proof of  two last inequalities is very similar. Let us take arbitrary $\theta\in (0,T]$. Since $t_N=Nh=T$, we have
		\begin{equation}\label{s2.27}
			\begin{split}
				I_N(\theta)&:=\int_\theta^{T-\theta}\abs{\bu_N(s+\theta)-\bu_N(s)}_{0,2}^4ds\\
				&\leq \int_0^{T-\theta}\abs{\bu_N(s+\theta)-\bu_N(s)}_{0,2}^4ds\\
				&= \left[\sum_{\ell=0}^{N-2}\left(\int_{t_{\ell}}^{t_{\ell+1}-\theta}+\int^{t_{\ell+1}}_{t_{\ell+1}-\theta}\right)+\int_{t_{N-1}}^{t_{N}-\theta}\right]\abs{\bu_N(s+\theta)-\bu_N(s)}_{0,2}^4ds.
			\end{split}
		\end{equation}
		We note that
		\begin{equation*}
			\be\sup_{0<\theta\leq T}\theta^{-\frac{1}{2}}\left(I_N(\theta\right))^\frac{1}{4}\leq \be\sup_{0<\theta\leq h}\theta^{-\frac{1}{2}}\left(I_N(\theta\right))^\frac{1}{4}+\be\sup_{h<\theta\leq T}\theta^{-\frac{1}{2}}\left(I_N(\theta\right))^\frac{1}{4}.
		\end{equation*}
		We then distinguish two cases: the case $0<\theta\leq h$ and the case $h<\theta\leq T$.
		
		\textbf{Case (a): $0<\theta\leq h$.}
		We recall that  for any $\ell\in \{0,1,...,N-1\}$,
		\begin{equation*}
			\abs{\bu_N(t)-\bu_N(s)}_{0,2}\leq \frac{\valabs{t-s}}{h}\abs{\bu^{\ell+1}-\bu^\ell}_{0,2}, \ \forall t,s\in [t_\ell,t_{\ell+1}].
		\end{equation*}
		Let  us fix $\ell\in \{0,...,N-1\}$ and take $s\in [t_{\ell+1} -\theta,t_{\ell+1}]\subset[t_\ell,t_{\ell+1}]$. Then $|t_{\ell+1} -s|\leq \theta$. Moreover, $s+\theta\in[t_{\ell+1}, t_{\ell+1} + \theta] \subset [t_{\ell+1}, t_{\ell+2}]$ and $|(s+\theta) -t_{\ell+1}| \leq \theta$. Hence,
		\begin{equation*}
			\begin{split}
				\abs{\bu_N(s+\theta)-\bu_N(s)}_{0,2}&\leq \abs{\bu_N(s+\theta)-\bu^{\ell+1}}_{0,2}+\abs{\bu^{\ell+1}-\bu_N(s)}_{0,2}\\
				&\leq\frac{\theta}{h} \abs{\bu_N^{\ell+2}-\bu^{\ell+1}}_{0,2}+\abs{\bu_N(t_{\ell+1})-\bu_N(s)}_{0,2}\\
				&\leq\frac{\theta}{h} \abs{\bu_N^{\ell+2}-\bu^{\ell+1}}_{0,2}+\frac{\valabs{t_{\ell+1}-s}}{h}\abs{\bu^{\ell+1}-\bu^\ell}_{0,2}\\
				&\leq\frac{\theta}{h} \abs{\bu_N^{\ell+2}-\bu^{\ell+1}}_{0,2}+\frac{\theta}{h}\abs{\bu^{\ell+1}-\bu^\ell}_{0,2}.
			\end{split}
		\end{equation*}
		This implies that  (since $\theta\leq h$)
		\begin{equation*}
			\begin{split}
				&\sum_{\ell=0}^{N-2}\int^{t_{\ell+1}}_{t_{\ell+1}-\theta}\abs{\bu_N(s+\theta)-\bu_N(s)}_{0,2}^4ds\\
				&\leq \bk\left(\frac{\theta}{h}\right)^4h\sum_{\ell=0}^{N-2}\abs{\bu_N^{\ell+2}-\bu^{\ell+1}}^4_{0,2}+\bk\left(\frac{\theta}{h}\right)^4h\sum_{\ell=0}^{N-2}\abs{\bu^{\ell+1}-\bu^\ell}^4_{0,2}\\
				&\leq\bk\left(\frac{\theta}{h}\right)^4h\sum_{\ell=0}^{N-1}\frac{\theta}{h}\abs{\bu^{\ell+1}-\bu^\ell}^4_{0,2}.
			\end{split}
		\end{equation*}
		Accordingly, now let $s \in  [t_\ell, t_{\ell+1} -\theta] \subset [t_\ell, t_{\ell+1}]$. Then also $s +\theta \in  [t_\ell, t_{\ell+1}]$ and hence
		\begin{equation*}
			\abs{\bu_N(s+\theta)-\bu_N(s)}_{0,2}\leq \frac{\theta}{h}\abs{\bu^{\ell+1}-\bu^\ell}_{0,2}.
		\end{equation*}
		Consequently, by using \re{s2.27}, we have $ I_N(\theta)\leq \bk\left(\frac{\theta}{h}\right)^4h\Sum_{\ell=0}^{N-1}\abs{\bu^{\ell+1}-\bu^\ell}^4_{0,2}$. Thus, 
		\begin{equation*}
			\sup_{0<\theta\leq h}\theta^{-\frac{1}{2}}\left(I_N(\theta\right))^\frac{1}{4}\leq\bk\sup_{0<\theta\leq h}\theta^{-\frac{1}{2}}\frac{\theta}{h}\left(h\sum_{\ell=0}^{N-1}\abs{\bu^{\ell+1}-\bu^\ell}^4_{0,2}\right)^\frac{1}{4}.
		\end{equation*}
		By using the H\"older inequality, the first inequality of Lemma \ref{Lemmas2.6} with $j=1$ (so that $t_j = t_1 = h$) and the fact that $\theta\leq h$,  we get 
		\begin{equation*}
			\be\sup_{0<\theta\leq h}\theta^{-\frac{1}{2}}\left(I_N(\theta\right))^\frac{1}{4}\leq\bk\sup_{0<\theta\leq h}\frac{\theta^\frac{1}{2}}{h}\left(\be h\sum_{\ell=0}^{N-1}\abs{\bu^{\ell+1}-\bu^\ell}^4_{0,2}\right)^\frac{1}{4}
			\leq \bk_m\sup_{0<\theta\leq h}\frac{\theta^\frac{1}{2}}{h^\frac{1}{2}}\leq \bk_m.
		\end{equation*} 
		\textbf{Case (b): $h<\theta\leq T$.} In this case, for any $\theta>h$, we can find $1\leq j\leq N-1$ and $\eta\in(0,1)$ such that $\theta=h(j+\eta)$. For any $\ell\in \{0,1,...,N-1\}$, let $s\in
		[t_\ell, t_{\ell+1}] \cap [0, T - \theta]$. Then, by the triangle inequality we have
		\begin{equation*}
			\begin{split}
				&\abs{\bu_N(s+\theta)-\bu_N(s)}_{0,2}\\
				&\leq \abs{\bu_N(s+\theta)-\bu_N(s+t_j)}_{0,2}+\abs{\bu_N(s+t_j)-\bu_N(t_{\ell+j})}_{0,2}\\
				&\qquad+\abs{\bu_N(t_{\ell+j})-\bu_N(t_\ell)}_{0,2}+\abs{\bu_N(t_\ell)-\bu_N(s)}_{0,2}\\
				&:= I+II+III+IV.
			\end{split}
		\end{equation*}
		We may proceed as in the case (a) to control the terms I, II and IV to  arrive at
		\begin{equation*}
			I_N(\theta)\leq 3\times 4^3h\sum_{\ell=0}^{N-1}\abs{\bu^{\ell+1}-\bu^\ell}^4_{0,2}+4^3h\sum_{\ell=0}^{N-j}\abs{\bu^{\ell+j}-\bu^\ell}^4_{0,2}.
		\end{equation*}
		By using the H\"older inequality and the first inequality of Lemma \ref{Lemmas2.6} as well as the fact that $t_1=h$,  $h<\theta$ and $\theta=t_j+h\eta>t_j$, with $\eta\in (0,1)$,  we get 
		\begin{equation*}
			\begin{split}
				\be\sup_{h<\theta\leq T}\theta^{-\frac{1}{2}}\left(I_N(\theta\right))^\frac{1}{4}&\leq\bk\sup_{h<\theta\leq T}\theta^{-\frac{1}{2}}\left(\be h\sum_{\ell=0}^{N-j}\abs{\bu^{\ell+1}-\bu^\ell}^4_{0,2}+\be h\sum_{\ell=0}^{N-j}\abs{\bu^{\ell+j}-\bu^\ell}^4_{0,2}\right)^\frac{1}{4}\\
				&\leq\bk_m\sup_{h<\theta\leq T}\theta^{-\frac{1}{2}}\left(t_1^2+t_j^2\right)^\frac{1}{4}\\
				&\leq\bk_m\sup_{h<\theta\leq T}\theta^{-\frac{1}{2}}\left(\theta^2+\theta^2\right)^\frac{1}{4}=\bk_m,
			\end{split}
		\end{equation*}
		which completes the proof of Lemma \ref{lemmas2.9}.
	\end{prev}
	\begin{remark}\label{remarks2.10} Owing to  \cite[Section 13, Corollary 24]{Sim1}, it follows from Lemma \ref{lemmas2.9} that for any $m\geq 1$, there exists a constant $\bk_m>0$ such that  for any $\zeta\in (0,\frac{1}{2})$ and any integer $N\in \mathbb{N}$,
		\begin{equation*}
			\be \abs{\bu_N}_{\sobm^{\zeta,4}(0,T;\bh_m)}+\be \abs{c_N}_{\sobm^{\zeta,4}(0,T;(\h^1)')}+\be \abs{n_N}_{\sobm^{\zeta,4}(0,T;(\h^1)')}\leq \bk_m.
		\end{equation*}
	\end{remark}
	
	Hereafter, for any integer $ N\in \mathbb{N}$ and  $t \in[0,T]$, we set 
	\begin{equation*}
		\ell_N (t):=\min\{\ell\in \{0,1...,N-1\}:\  t\in [t_\ell,t_{\ell+1}]\}.% \text{ and } t^N_\ell :=\ell_N(t)h.
	\end{equation*}
	Next, we claim that the processes $\bu_N$, $c_N$ and $n_N$ satisfy the following system of stochastic equation on the complete filtered probability space $(\Omega, \mathcal{F}, \mathbb{F}=(\mathcal{F}_t)_{t\in[0,T]}, \mathbb{P})$. The proof is very similar to \cite[Proof of Proposition 3.12]{Razafimandimby} (see also \cite[Section 4.2]{Glatt}).   For any integer $ N\in \mathbb{N}$, the processes $\bu_N$, $c_N$ and $n_N$ satisfy for any $t\in [0,T]$ and $\mathbb{P}$-a.s., 
	\begin{equation}\label{s2.28}
		\begin{split}
			& 	\bu_N(t)+\int_0^t[\eta A\hat{\bu}_N(s)+B^m(\check{\bu}_N(s),\hat{\bu}_N(s))-B_0^m(\theta_m(\hat{n}_N(s)),\phi)]d s\\
			&=\bu^m_0+\alpha\int_0^t \pi_m( \Pl(\mathbf{f}\check{	 \bu}_N(s))) d \bar W(s)+\bec_N^\bu(t),\\
			&	c_N(t)+ \int_0^t[\eps A_1\hat{c}_N(s) +B_1(\hat{\bu}_N(s),  \hat{c}_N(s))]d s\\
			&=c_0^m-\int_0^tB_2(\theta^{\eps_m}_m(\hat{n}_N(s)),\hat{c}_N(s))d s+\gamma\int_0^t F^1_m(\check{c}_N(s))g (\check{c}_N(s) ) d \beta(s)+\bec_N^c(t),\\
			&n_N(t)+ \int_0^t[\eps A_1\hat{n}_N(s) +B_1(\hat{\bu}_N(s),  \hat{n}_N(s))]d s=n_0^m+\int_0^tB_3(\theta_m(\hat{n}_N),\hat{c}_N)d s+\bec_N^n(t),
		\end{split}		
	\end{equation}
	in $\bh_m$, $\elm^2$ and $(\h^1)'$ respectively, where the error terms are defined by
	\begin{equation*}
		\begin{split}
			\bec_N^\bu(t)&:=\alpha\int_t^{t_{\ell_N(t)+1}}\pi_m( \Pl(\mathbf{f}\check{	 \bu}_N(s))) d \bar W(s)-\alpha\frac{h\wedge t}{h}\int_0^h\pi_m( \Pl(\mathbf{f}\check{	 \bu}_N(s))) d \bar W(s)\\
			&\qquad+ \int_0^t[\eta A\hat{\bu}_N(s)+B^m(\check{\bu}_N(s),\hat{\bu}_N(s))-B_0^m(\theta_m(\hat{n}_N(s)),\phi)]1_{[0,h]}(s)d s,
		\end{split}
	\end{equation*}
	\begin{equation*}
		\begin{split}
			\bec_N^c(t)&:=\gamma\int_t^{t_{\ell_N(t)+1}}F^1_m(\check{c}_N(s))g( \check{c}_N(s) ) d \beta(s)-\gamma\frac{h\wedge t}{h}\int_0^hF^1_m(\check{c}_N(s))g(\check{c}_N(s))  d \beta(s)\\
			&\qquad- \int_0^t[\eps A_1\hat{c}_N(s) +B_1(\hat{\bu}_N(s),  \hat{c}_N(s))-B_2(\theta^{\eps_m}_m(\hat{n}_N(s)),\hat{c}_N(s))]1_{[0,h]}(s)d s,
		\end{split}
	\end{equation*}
	and 
	\begin{equation*}
		\bec_N^n(t)=- \int_0^t[\eps A_1\hat{n}_N(s) +B_1(\hat{\bu}_N(s),  \hat{n}_N(s))+B_3(\theta_m(\hat{n}_N(s)),\hat{c}_N(s))]1_{[0,h]}(s)d s,
	\end{equation*}
	where $h\wedge t :=\min(h,t)$. 
	
	In what follows, we analyse the tightness of the law of processes we defined above. For this aim, we set 
	\begin{equation*}
		\begin{split}
			&\mathcal{Y}_\bu:=C([0,T];\bh_m)\cap  \elm^2([0,T];\bh_m),\quad  \hat{\mathcal{Y}}_n:= \elm^2_{weak}(0,T;\h^1), \quad \mathcal{Y}_W:=C([0,T];\mathbb{R}^3),\\			
			&\mathcal{Y}_c:=C([0,T];(\h^1)')\cap C([0,T];\h^1_{weak})\cap \elm^2(0,T;\h^1)\cap \elm^2_{weak}(0,T;D(A_1)),\\
			& \mathcal{Y}_n:=C([0,T];(\h^1)')\cap C([0,T];\elm^2_{weak})\cap \elm^2(0,T;\elm^2)\cap \elm^2_{weak}(0,T;\h^1),\\
			&\hat{\mathcal{Y}}_c:= \elm^2_{weak}(0,T;D(A_1)), \ \hat{\mathcal{Y}}_\bu:= \elm^2_{weak}(0,T;\bh_m).
		\end{split}
	\end{equation*}
	We define a sequence of $\mathbb{R}^3$-valued Wiener processes  $\{\beta^N, W^N\}_{N\in \mathbb{N}}$ defined by
	\begin{equation*}
		W^N:=(W_1,W_2,W_3) \text{ and } \beta^N:=(\beta_1,\beta_2,\beta_3), \ \forall N\in \mathbb{N}.
	\end{equation*}
	Let us denote the family of laws of $\{\bu_N\}_{N\in \mathbb{N}}$, $\{c_N\}_{N\in \mathbb{N}}$, $\{n_N\}_{N\in \mathbb{N}}$, $\{\hat{n}_N\}_{N\in \mathbb{N}}$,  $\{\hat{c}_N\}_{N\in \mathbb{N}}$, $\{\check{c}_N\}_{N\in \mathbb{N}}$,  $\{\hat{\bu}_N\}_{N\in \mathbb{N}}$, $\{\check{\bu}_N\}_{N\in \mathbb{N}}$,$\{W^N\}_{N\in \mathbb{N}}$ and $\{\beta^N\}_{N\in \mathbb{N}}$ on $\mathcal{Y}_\bu$, $\mathcal{Y}_c$, $\mathcal{Y}_n$, $\hat{\mathcal{Y}}_n$, $\hat{\mathcal{Y}}_c$, $\hat{\mathcal{Y}}_c$, $\hat{\mathcal{Y}}_\bu$, $\hat{\mathcal{Y}}_\bu$, $\mathcal{Y}_W$ and $\mathcal{Y}_W$ respectively by $\{\nu^\bu_N\}_{N\in \mathbb{N}}$, $\{\nu^c_N\}_{N\in \mathbb{N}}$, $\{\nu^n_N\}_{N\in \mathbb{N}}$, $\{\hat{\nu}^n_N\}_{N\in \mathbb{N}}$,   $\{\hat{\nu}^c_N\}_{N\in \mathbb{N}}$, $\{\check{\nu}^c_N\}_{N\in \mathbb{N}}$, $\{\hat{\nu}^\bu_N\}_{N\in \mathbb{N}}$, $\{\check{\nu}^\bu_N\}_{N\in \mathbb{N}}$,  $\{\nu^W_N\}_{N\in \mathbb{N}}$ and $\{\nu^\beta_N\}_{N\in \mathbb{N}}$ and  prove the following result.
	\begin{lemma}\label{Lemmas2.11}
		The family $\{(\nu^\bu_N,\nu^c_N,\nu^n_N,\hat{\nu}^n_N,  \hat{\nu}^c_N, \check{\nu}^c_N,\hat{\nu}^\bu_N, \check{\nu}^\bu_N, \nu^W_N,\nu^\beta_N)\}_{N\in \mathbb{N}}$ is  tight on  $$\mathcal{Y}:=\mathcal{Y}_\bu\times \mathcal{Y}_c\times \mathcal{Y}_n\times\hat{ \mathcal{Y}}_n\times\hat{ \mathcal{Y}}_c\times \hat{ \mathcal{Y}}_c\times \hat{ \mathcal{Y}}_\bu\times \hat{ \mathcal{Y}}_\bu\times \mathcal{Y}_W\times \mathcal{Y}_W.$$
	\end{lemma}
\begin{proof}[	\textbf{Proof of Lemma \ref{Lemmas2.11}}]
	  It is  sufficient to consider the tightness of each component of\\ $\{(\nu^\bu_N,\nu^c_N,\nu^n_N,\nu^W_N,\nu^\beta_N)\}_{N\in \mathbb{N}}$ since  a cartesian product of finite compact sets is compact.  To start, we note that since $\bh_m$ is a finite dimensional space, we have $\bh_m\hookrightarrow\hookrightarrow \bh_m\hookrightarrow \bh_m$. Then, by applying  \cite[Corollary 5]{Sim}, we derive that for any $\zeta\in (\frac{1}{4},\frac{1}{2})$, $\elm^\infty(0,T,\bh_m)\cap \sobm^{\zeta, 4}(0,T;\bh_m)\hookrightarrow\hookrightarrow C([0,T];\bh_m)$ and $\elm^2(0,T,\bh_m)\cap \sobm^{\zeta, 4}(0,T;\bh_m)\hookrightarrow\hookrightarrow \elm^2(0,T;\bh_m).$  With this compact embedding in hand, the tightness of  $\{\nu^\bu_N\}_{N\in \mathbb{N}}$ on $\mathcal{Y}_\bu$ follows from Remark \ref{remarks2.10} and the a priori estimate \re{s2.25}$_1$. Next,  we fix $\zeta\in (\frac{1}{4},\frac{1}{2})$ and $p=4$. Since $\zeta p\in (1,2)$, we see  that  $\sobm^{\zeta,4}(0,T;(\h^1)')\hookrightarrow C^\beta([0,T]; (\h^1)')$, for any fixed $\beta\in (0,4\zeta-1)$ (see \cite[Proof of Theorem 2.2]{Flan}). It  follows from Remark \ref{remarks2.10} that\\
	$\sup_{N\in \mathbb{N}}\be \abs{c_N}_{C^\beta([0,T]; (\h^1)')}<\infty \text{ and }\sup_{N\in \mathbb{N}} \be \abs{n_N}_{C^\beta([0,T]; (\h^1)')}<\infty.$ 
Thanks to these two  last inequalities and the estimates \re{s2.25} and \re{s2.26}, we can apply \cite[Lemma 3.3 ]{Bre7} and the Markov inequality to obtain the tightness of  $\{\nu^c_N\}_{N\in \mathbb{N}}$ on $\mathcal{Y}_c$ and  of $\{\nu^n_N\}_{N\in \mathbb{N}}$ on $\mathcal{Y}_n$.  By the Banach Alaoglu theorem, Bounded sets of  $\elm^2(0,T;\h^1)$ are compact in $L_{weak}^2(0,T;\h^1)$,  the tightness of $\{\hat{\nu}^n_N\}_{N\in \mathbb{N}}$ on $\hat{\mathcal{Y}}_n$ easily follows from  \re{s2.26}$_3$. In a very similar way, we can obtain the tightness of $\check{\nu}^n_N$, $\hat{\nu}^c_N$, $\check{\nu}^c_N$, $\hat{\nu}^\bu_N$ and $ \check{\nu}^\bu_N$. The tightness of $\{\nu^W_N\}_{N\in \mathbb{N}}$ on $\mathcal{Y}_W$ and $\{\nu^\beta_N\}_{N\in \mathbb{N}}$ on $\mathcal{Y}_W$ comes from the fact that by construction each of those families is reduced to   one element, and the Lemma \ref{Lemmas2.11} is then proved.
\end{proof}
	
	\textbf{Step 4: Passage to the limit.} In this step, we construct a martingale solution of problem \re{3.2}.  By the  Prokhorov Theorem \cite[Theorem 5.1]{Billing}, one can  find a subsequence of $$\{(\nu^\bu_N,\nu^c_N,\nu^n_N,\hat{\nu}^n_N,  \hat{\nu}^c_N, \check{\nu}^c_N,\hat{\nu}^\bu_N, \check{\nu}^\bu_N, \nu^W_N,\nu^\beta_N)\}_{N\in \mathbb{N}},$$ still labelled by $N$, which convergences to a probability measure  $\{(\nu^\bu,\nu^c,\nu^n,\hat{\nu}^n,   \hat{\nu}^c, \check{\nu}^c,\hat{\nu}^\bu, \check{\nu}^\bu, \nu^W,\nu^\beta)\}$ on  $\mathcal{Y}$. Since  $\mathcal{Y}_c$, $\mathcal{Y}_n$ and $\hat{\mathcal{Y}}_n$  are not  metric space, we cannot apply  the classical  Skorokhod representation theorem. We instead use the  Jakubowski-Skorokhod  representation  theorem, see \cite[Theorem 2]{Jakubowski}, which is appropriate to our framework to infer the existence of  a new probability space $(\bar {\Omega}_m,\bar {\mathcal{F}}_m,\bar {\mathbb{P}}_m)$ on which one can find a sequence of $\mathcal{Y}$-valued random variables $\{(\bar \bu_N, \bar c_N, \bar n_N,  \hat{\bar n}_N,    \hat{\bar c}_N,  \check{\bar c}_N,   \hat{\bar \bu}_N,  \check{\bar \bu}_N, \bar {\beta}_N, \bar W_N); \ N\in \mathbb{N}\}$ such that its family of laws on $\mathcal{Y}$ is equal to $\{(\nu^\bu_N,\nu^c_N,\nu^n_N,\hat{\nu}^n_N,  \hat{\nu}^c_N, \check{\nu}^c_N,\hat{\nu}^\bu_N, \check{\nu}^\bu_N, \nu^W_N,\nu^\beta_N)\}_{N\in \mathbb{N}}$. On $(\bar {\Omega}_m,\bar{ \mathcal{F}}_m,\bar{ \mathbb{P}}_m)$  one can also find a $\mathcal{Y}$-valued random variable 
	$(\bu_m,c_m,n_m,\hat n_m, \hat c_m, \check c_m, \hat \bu_m, \check \bu_m, \bar{ \beta}^m, \bar W^m)$  such that the following convergences hold $\bar {\mathbb{P}}_m$-a.s.
	\begin{equation}\label{s2.29}
		\begin{split}
			&\bar \bu_N\longrightarrow \bu_m \text{ in } C([0,T];\bh_m)\cap \elm^2(0,T;\bh_m),\\
			&\bar c_N\longrightarrow c_m \text{ in } C([0,T];(\h^1)')\cap C([0,T];\h^1_{weak})\cap \elm^2(0,T;\h^1)\cap \elm^2_{weak}(0,T;D(A_1)),\\
			&\bar n_N\longrightarrow n_m \text{ in } C([0,T];(\h^1)')\cap C([0,T];\elm^2_{weak})\cap \elm^2(0,T;\elm^2)\cap \elm^2_{weak}(0,T;\h^1),\\
			&\hat{\bar n}_N\longrightarrow \hat n_m \text{ in }  \elm^2_{weak}(0,T;\h^1),\quad \bar W^N\longrightarrow \bar W^m\text{ and } \bar {\beta}^N\longrightarrow \bar{\beta}^m\text{ in } C([0,T];\mathbb{R}^3),\\
			&\check{\bar c}_N\longrightarrow \check c_m \text{ in }  \elm^2_{weak}(0,T;D(A_1)),\quad\hat{\bar c}_N\longrightarrow \hat c_m \text{ in }  \elm^2_{weak}(0,T;D(A_1)),\\
			&\check{\bar \bu}_N\longrightarrow \check \bu_m \text{ in }  \elm^2_{weak}(0,T;\bh_m),\quad\hat{\bar \bu}_N\longrightarrow \hat \bu_m \text{ in }  \elm^2_{weak}(0,T;\bh_m),.
		\end{split}
	\end{equation}
	Since      $(\bar {\Omega}_m,\bar {\mathcal{F}}_m,\bar{\mathbb{P}}_m)$ is constructed from the   Jakubowski-Skorokhod  representation  theorem, it is independent of $m$. Indeed,  one may choose for any $m\geq 1$,  $(\bar {\Omega}_m,\bar {\mathcal{F}}_m,\bar{\mathbb{P}}_m)= ([0, 1], \mathcal{B}([0, 1]), \mu_0)$ where $\mu_0$ denotes the $1$-dimensional Lebesgue measure,   see for instance  \cite[Proof of Theorem 2]{Jakubowski}. Hereafter, we set
	\begin{equation*}
		(\bar {\Omega},\bar {\mathcal{F}},\bar{\mathbb{P}}):=([0, 1], \mathcal{B}([0, 1]), \mu_0)=(\bar {\Omega}_m,\bar {\mathcal{F}}_m,\bar{\mathbb{P}}_m), \ \forall m\geq 1.
	\end{equation*}
	
	Sine $\elm^2(0,T;\elm^2)$ and $\elm^2(0,T;\h^1)$ are Polish spaces, by the Kuratowski theorem (\cite[Theorem 1.1, P.5]{Vakhania}), $\elm^\infty(0,T;\elm^2)$ is a Borel set of  $\elm^2(0,T;\elm^2)$ and in the same way, $\elm^\infty(0,T;\h^1)$ is  Borel set of $\elm^2(0,T;\h^1)$. In addition of this, the equality of law of \\$\{(\bar \bu_N, \bar c_N, \bar n_N,\hat{\bar n}_N,\bar {\beta}_N, \bar W_N); \ N\in \mathbb{N}\}$ and  $\{(\nu^\bu_N,\nu^c_N,\nu^n_N,\hat \nu^n_N,\nu^W_N,\nu^\beta_N)\}_{N\in \mathbb{N}}$ on $\mathcal{Y}$ and Lemma \ref{Lemmas2.7}, imply that  for any $p\in [1,2]$, there exists a constant $\bk_m>0$ such that for any $N\in \mathbb{N}$,
	\begin{equation}\label{s2.30}
		\begin{split}
			&\bar \be\sup_{0\leq s\leq T}\abs{\bar\bu_N}_{0,2}^{2p}+\bar \be\sup_{0\leq s\leq T}\abs{\bar c_N}_{0,2}^{2p}+\bar \be\sup_{0\leq s\leq T}\abs{\bar c_N}_{1,2}^{2}+\bar \be\sup_{0\leq s\leq T}\abs{\bar n_N}_{0,2}^{2p}\leq \bk_m,\\
			&\bar \be\left(\int_0^T\abs{\bar \bu_N(s)}_{0,2}^2ds\right)^{p}+\bar \be\int_0^T\abs{A_1\bar c_N(s)}_{0,2}^2ds+\bar \be\left(\int_0^T\abs{\check{\bar c}_N(s)}_{1,2}^2ds\right)^{p}\leq \bk_m,\\
			&\bar \be\left(\int_0^T\abs{\bar c_N(s)}_{1,2}^2ds\right)^{p}+\bar \be\left(\int_0^T\abs{\bar n_N(s)}_{1,2}^2ds\right)^{p}+\bar \be\left(\int_0^T\abs{\hat{\bar n}_N(s)}_{1,2}^2ds\right)^{p}\leq \bk_m,\\
			&\bar \be\left(\int_0^T\abs{\hat{\bar \bu}_N(s)}_{0,2}^2ds\right)^{p}+\bar \be\left(\int_0^T\abs{\check{\bar \bu}_N(s)}_{0,2}^2ds\right)^{p}\leq \bk_m,\\
			%&\bar \be\sup_{0\leq s\leq T}\abs{\hat{\bar \bu}_N(s)}_{0,2}^{2p}+\bar \be\sup_{0\leq s\leq T}\abs{\check{\bar \bu}_N(s)}_{0,2}^{2p}+\bar \be\sup_{0\leq s\leq T}\abs{\hat{\bar c}_N(s)}_{0,2}^{2p}+\bar \be\sup_{0\leq s\leq T}\abs{\check{\bar c}_N(s)}_{0,2}^{2p}\leq \bk_m,\\
			%&\bar \be\sup_{0\leq s\leq T}\abs{\hat{\bar n}_N(s)}_{0,2}^{2p}+\bar \be\sup_{0\leq s\leq T}\abs{\hat{\bar c}_N(s)}_{1,2}^{2}+\bar \be\sup_{0\leq s\leq T}\abs{\check{\bar c}_N(s)}_{1,2}^{2}\leq \bk_m,\\
			&\bar \be\int_0^T\abs{A_1\hat{\bar c}_N(s)}_{0,2}^2ds+\bar \be\int_0^T\abs{A_1\check{\bar c}_N(s)}_{0,2}^2ds+\bar \be\left(\int_0^T\abs{\hat{\bar c}_N(s)}_{1,2}^2ds\right)^{p}\leq \bk_m,\\
			&\bar \be\sup_{0\leq s\leq T}\abs{\hat{\bar \bu}_N(s)}_{0,2}^{2p}+\bar \be\sup_{0\leq s\leq T}\abs{\check{\bar \bu}_N(s)}_{0,2}^{2p}+\bar \be\sup_{0\leq s\leq T}\abs{\hat{\bar c}_N(s)}_{1,2}^{2p}\leq \bk_m.
		\end{split}
	\end{equation}
	
	In fact, since the space $\elm^2(0,T;\h^1)$ is a separable Banach space, the borel set of $\elm^2(0,T;\h^1)$ coincides with the borel sets of $\elm^2_{weak}(0,T;\h^1)$  (the same remark holds for $\elm^2(0,T;D(A_1))$).
	
	By using \re{s2.30} and the convergence \re{s2.29}, we derive that  for any $p\in [1,2]$, 
	\begin{equation}\label{s2.31}
		\begin{split}
			& \elm^{2p}(\bar {\Omega},\bar{ \mathcal{F}},\bar{ \mathbb{P}}; \elm^2(0,T;\bh_m)),\qquad c_m\in  \elm^{2p}(\bar {\Omega},\bar{ \mathcal{F}},\bar{ \mathbb{P}}; \elm^2(0,T;\h^1)),\\
			&\hspace{3 cm}n_m\in \elm^{2p}(\bar {\Omega},\bar{ \mathcal{F}},\bar{ \mathbb{P}}; \elm^2(0,T;\elm^2)).
		\end{split}
	\end{equation}
	In fact,   due to the strong convergence almost surely in $\elm^2(0,T;\elm^2)$, we have for some integer $N_0>0$, $\int_0^T\abs{n_m(s)}^2_{0,2}ds\leq \int_0^T\abs{\bar n_N(s)}^2_{0,2}ds$, $\bar{\mathbb{P}}$-a.s., for any $N\geq N_0$ and  therefore,  $$\bar \be\left( \int_0^T\abs{n_m(s)}^2_{0,2}ds\right)^p\leq \sup_{N\geq N_1}\bar \be\left(\int_0^T\abs{\bar n_N(s)}^2_{0,2}ds\right)^p\leq \bk_m.$$
	Moreover, by using \re{s2.31}, \re{s2.30}, the convergence \re{s2.29} and the Vitali convergence theorem, we derive that 
	\begin{equation}\label{s2.32}
		\begin{split}
			&\lim_{N\longrightarrow\infty}\bar \be\int_0^T\abs{\bar \bu_N(s)-\bu_m(s)}_{0,2}^2ds = \lim_{N\longrightarrow\infty}\bar \be\int_0^T\abs{\bar c_N(s)-c_m(s)}_{1,2}^2ds=0,\\
			&\hspace{3cm}\lim_{N\longrightarrow\infty}\bar \be\int_0^T\abs{\bar n_N(s)-n_m(s)}_{0,2}^2ds=0.
		\end{split}
	\end{equation}
	From the equality of laws of $(\bu_N,c_N,\hat{\bu}_N,\hat{c}_N, \hat{n}_N, \check{\bu}_N,\check{c}_N)$ and $(\bar \bu_N,\bar c_N,\hat{\bar \bu}_N,\hat{\bar c}_N, \hat{\bar n}_N, \check{\bar \bu}_N,\check{\bar c}_N)$, the convergence \re{s2.32} and the Lemma \ref{lemmas2.8} we deduce that 
	\begin{equation*}
		\begin{split}
			&\lim_{N\longrightarrow\infty}\bar \be\int_0^T\abs{\hat{\bar \bu}_N(s)-\bu_m(s)}_{0,2}^2ds = \lim_{N\longrightarrow\infty}\bar \be\int_0^T\abs{\check{\bar \bu}_N(s)-\bu_m(s)}_{0,2}^2ds =0,\\
			& \lim_{N\longrightarrow\infty}\bar \be\int_0^T\abs{\hat{\bar c}_N(s)-c_m(s)}_{1,2}^2ds= \lim_{N\longrightarrow\infty}\bar \be\int_0^T\abs{\check{\bar c}_N(s)-c_m(s)}_{1,2}^2ds=0,\\
			&\hspace{3cm}\lim_{N\longrightarrow\infty}\bar \be\int_0^T\abs{\hat{\bar n}_N(s)-n_m(s)}_{0,2}^2ds=0.
		\end{split}
	\end{equation*}
	Therefore, up to a subsequence, the following convergences holds, $\bar{ \mathbb{P}}$-a.s.
	\begin{equation}\label{s2.33}
		\begin{split}
			&\lim_{N\longrightarrow\infty} \int_0^T\abs{\hat{\bar \bu}_N(s)-\bu_m(s)}_{0,2}^2ds =  \lim_{N\longrightarrow\infty}\int_0^T\abs{\check{\bar \bu}_N(s)-\bu_m(s)}_{0,2}^2ds =0,\\
			& \lim_{N\longrightarrow\infty}\int_0^T\abs{\hat{\bar c}_N(s)-c_m(s)}_{1,2}^2ds=  \lim_{N\longrightarrow\infty}\int_0^T\abs{\check{\bar c}_N(s)-c_m(s)}_{1,2}^2ds=0,\\
			&\hspace{3cm}\lim_{N\longrightarrow\infty}\int_0^T\abs{\hat{\bar n}_N(s)-n_m(s)}_{0,2}^2ds=0.
		\end{split}
	\end{equation}
	Hereafter, we work with these subsequences.

	Let $\mathcal{N}$, and $\bar{ \mathcal{N}}$ be the set of null sets of $\mathcal{F}$ and $\bar{\mathcal{F}}$ respectively. Let us  consider the following filtrations of $\mathcal{F}$ and of  $\bar{\mathcal{F}}$  defined respectively  by for $t\in [0,T]$,
	\begin{equation*}
		\mathcal{F}^N_t=\sigma\left( \sigma\left((\check{\bu}_N(s), \check{c}_N(s), W^N(s), \beta^N(s)): s\leq t \right)\cup \mathcal{N}\right),  \ N \geq 1,
	\end{equation*}
	and 
	\begin{equation*}
		\bar{	\mathcal{F}}^{N}_t=\sigma\left( \sigma\left((\check{\bar{\bu}}_N(s),\check{\bar c}_N(s), \bar W^N(s),\bar \beta^N(s)): s\leq t \right)\cup \mathcal{N}'\right),  \ N \geq 1.
	\end{equation*}
	Since  $(W^N,\beta^N)$ has the same law as $(W,\beta)$, it is not difficult to prove that $(W^N,\beta^N)$  (resp. $(\bar W^N,\bar \beta^N)$ ) are $\{\mathcal{F}^{N}_t\}_{t\in [0,T]}$ (resp. $\{\bar{\mathcal{F}}^{N}_t\}_{t\in [0,T]}$)  Wiener processes. Since $\bu_N$, $n_N$, $n_N$, $\hat{\bu}_N$, $\check{ c}_N$, $\check{n}_N$, $\check{\bu}_N$, $\check{ c}_N$, $\check{ n}_N$, $ W^N$ and  $ \beta^N(s)$ satisfy   \re{s2.28},  we can follow the exact same lines as the proof of   \cite[Proposition 6.3]{Raza1} to  derive that on the new filtered probability space  $(\bar {\Omega},\bar {\mathcal{F}}_m,\{\bar {\mathcal{F}}^{N}_t\}_{t\in [0,T]},\bar{ \mathbb{P}})$ the processes $\bar \bu_N$, $\bar n_N$, $\bar{n}_N$, $\hat{\bar{\bu}}_N$, $\check{\bar c}_N$, $\check{\bar n}_N$, $\check{\bar{\bu}}_N$, $\check{\bar c}_N$, $\check{\bar n}_N$, $\bar W^N$ and  $\bar \beta^N$ satisfy the following system for any $t\in [0,T]$ and $\bar{ \mathbb{P}}$-a.s., 
	\begin{equation}\label{s2.35}
		\begin{split}
			& 	\bar \bu_N(t)+\int_0^t[\eta A\hat{	\bar \bu}_N(s)+B^m(\check{	\bar \bu}_N(s),\hat{	\bar \bu}_N(s))-B_0^m(\theta_m(\hat{	\bar n}_N(s)),\phi)]d s\\
			&=\bu^m_0+\alpha\int_0^t  \pi_m( \Pl(\mathbf{f}\check{	\bar \bu}_N(s))) d \bar W^N(s)+	\bar \bec_N^\bu(t),\\
			&		\bar c_N(t)+ \int_0^t[\eps A_1\hat{	\bar c}_N(s) +B_1(\hat{	\bar \bu}_N(s),  \hat{	\bar c}_N(s))]d s\\
			&=c_0^m-\int_0^tB_2(\theta^{\eps_m}_m(\hat{	\bar n}_N(s)),\hat{\bar c}_N(s))d s+\gamma\int_0^t F^1_m(\check{	\bar c}_N(s))g( \check{	\bar c}_N(s))  d \bar \beta^N(s)+	\bar \bec_N^c(t),\\
			&	\bar n_N(t)+ \int_0^t[\eps A_1\hat{	\bar n}_N(s) +B_1(\hat{	\bar \bu}_N(s),  \hat{	\bar n}_N(s))]d s=n_0^m+\int_0^tB_3(\theta_m(\hat{	\bar n}_N),\hat{\bar c}_N)d s+	\bar \bec_N^n(t),
		\end{split}		
	\end{equation}
	in $\bh_m$, $\elm^2$ and $(\h^1)'$ respectively, where the error terms are defined as,
	\begin{equation*}
		\begin{split}
			\bar \bec_N^\bu(t)&:=\alpha\int_t^{t_{\ell_N(t)+1}} \pi_m( \Pl(\mathbf{f}\check{	\bar \bu}_N(s))) d \bar W^N(s)-\alpha\frac{h\wedge t}{h}\int_0^h \pi_m( \Pl(\mathbf{f}\check{	\bar \bu}_N(s))) d \bar W^N(s)\\
			&\qquad+ \int_0^t[\eta A\hat{	\bar \bu}_N(s)+B^m(\check{	\bar \bu}_N(s),\hat{	\bar \bu}_N(s))-B_0^m(\theta_m(\hat{	\bar n}_N(s)),\phi)]1_{[0,h]}(s)d s,,
		\end{split}
	\end{equation*}
	\begin{equation*}
		\begin{split}
			\bar \bec_N^c(t)&:=\gamma\int_t^{t_{\ell_N(t)+1}}F^1_m(\check{	\bar c}_N(s))g(\check{	\bar c}_N(s))  d \bar \beta^N(s)-\gamma\frac{h\wedge t}{h}\int_0^hF^1_m(\check{	\bar c}_N(s))g( \check{	\bar c}_N(s)  )d \bar \beta^N(s)\\
			&\qquad- \int_0^t[\eps A_1\hat{	\bar c}_N(s) +B_1(\hat{	\bar \bu}_N(s),  \hat{	\bar c}_N(s))-B_2(\theta^{\eps_m}_m(\hat{	\bar n}_N(s)),\hat{	\bar c}_N(s))]1_{[0,h]}(s)d s,
		\end{split}
	\end{equation*}
	and 
	\begin{equation*}
		\bar \bec_N^n(t)=- \int_0^t[\eps A_1\hat{	\bar n}_N(s) +B_1(\hat{	\bar \bu}_N(s),  \hat{	\bar n	}_N(s))+B_3(\theta_m(\hat{	\bar n}_N(s)),\hat{\bar c}_N(s))]1_{[0,h]}(s)d s,
	\end{equation*}
	By following  \cite[Proposition 4.4]{Razafimandimby},  we can prove that the limit processes  $\bar W^m$ and  $\bar \beta^m$ are $\mathbb{R}^3$-valued Brownian motions on the filtered probability space $(\bar {\Omega},\bar {\mathcal{F}},\{\bar {\mathcal{F}}^m_t\}_{t\in [0,T]},\bar{ \mathbb{P}})$  where the filtration $\{\bar {\mathcal{F}}^m_t\}_{t\in [0,T]}$ is defined by 
	\begin{equation*}
		\bar {\mathcal{F}}^m_t=\sigma\left( \sigma\left((\bu_m(s), c_m(s), n_m(s),\bar W^m(s),\bar \beta^m(s)): s\leq t \right)\cup \bar{\mathcal{N}}\right),
	\end{equation*}
	Our aim now is to pass to the limit to the system \re{s2.35}. We start by proving the following lemma. 
	\begin{lemma}\label{lemma 3.17}
		We have the following equalities $\bar{ \mathbb{P}}$-a.s.,
		\begin{equation*}
			\hat \bu_m=\bu_m, 	\ \check \bu_m=\bu_m, \	\check c_m=c_m, 	\ \hat c_m=c_m, \	\hat n_m=n_m.
		\end{equation*}
		\begin{proof}[\textbf{Proof of Lemma \ref{lemma 3.17}}]
			We will only  prove that $\hat n_m=n_m$, $\bar{ \mathbb{P}}$-a.s., since the rest will be similar. By the convergence \re{s2.33}, we derive that $\{\hat{\bar n}_N\}_{N\in \mathbb{N}}$ convergences to $n_m$ strongly on $\elm^2(0,T;\elm^2)$ and therefore converges also to $n_m$ in  $L_{weak}^2(0,T;\elm^2)$.  In the other hand, from the convergence \re{s2.29}$_3$, $\{\hat{\bar n}_N\}_{N\in \mathbb{N}}$ convergences to $\hat n_m$  in $L_{weak}^2(0,T;\h^1)$. Due to the embedding $L_{weak}^2(0,T;\h^1)\hookrightarrow L_{weak}^2(0,T;\elm^2)$, we derive that  $\{\hat{\bar n}_N\}_{N\in \mathbb{N}}$ convergences also to  $\hat n_m$ in  $L_{weak}^2(0,T;\elm^2)$ and due to the uniqueness of the limit, we derive that $\hat n_m=n_m$.
		\end{proof}
	\end{lemma}
	Next, we prove that the error terms $	\bar \bec_N^\bu$, 	$	\bar \bec_N^c$ and  $	\bar \bec_N^n$ vanishes when $N$ goes to infinity. 
	\begin{lemma}\label{Lemmas2.12}
		We have
		\begin{equation*}
			\lim_{N\longrightarrow\infty}\bar \be\int_0^T\abs{\bar \bec_N^\bu(s)}_{0,2}^2ds= \lim_{N\longrightarrow\infty}\bar \be\int_0^T\abs{\bar \bec_N^c(s)}_{0,2}^2ds=0  \lim_{N\longrightarrow\infty}\bar \be\int_0^T\abs{\bar \bec_N^n(s)}_{(\h^1)'}^2ds=0.
		\end{equation*}	
	\end{lemma}
\begin{proof}[	\textbf{Proof of Lemma \ref{Lemmas2.12}. }]
	 The first convergence is easy to obtain  because of the equivalence of norms on $\bh_m$, the inequalities of  \re{3.3} and the inequalities \re{s2.30}$_1$ and \re{s2.30}$_6$. We will only  focus on the proof of the second convergence. First, notice that thanks to the equivalence of norms on $\bh_m$,  \re{3.3},  \re{s2.30} and \re{s2.30}$_6$  it is not difficult to show that
	\begin{equation*}
		\begin{split}
			&\bar \be\int_0^T \abs{A_1\hat{	\bar c}_N(s) }^2_{0,2}ds\leq \bk_m,\ \bar \be\int_0^T \abs{B_2(\theta^{\eps_m}_m(\hat{	\bar n}_N(s)),\hat{	\bar c}_N(s))}^2_{0,2}ds\leq\bk_m\bar \be\sup_{0\leq s\leq T}\abs{\hat{	\bar c}_N(s)}^2_{0,2}\leq \bk_m,\\
			&\bar \be\int_0^T \abs{B_1(\hat{	\bar \bu}_N(s),  \hat{	\bar c}_N(s))}^2_{0,2}ds\leq\bk_m\left[\bar \be\sup_{0\leq s\leq T}\abs{\hat{	\bar \bu}_N(s)}^4_{0,2}\right]^\frac{1}{2}\left[\bar \be\left(\int_0^T\abs{\nabla\hat{	\bar c}_N(s)}^2_{0,2}\right)^2\right]^\frac{1}{2}\leq \bk_m.
		\end{split}
	\end{equation*}
	Thanks to these inequalities and the H\"older inequality, we easily infer that
	\begin{equation*}
		\begin{split}
			\bar \be\int_0^T \abs{\int_0^tA_1\hat{	\bar c}_N(s) 1_{[0,h]}ds}^2_{0,2}dt&\leq  h \bar \be\int_0^T\int_0^{t\wedge h} \abs{A_1\hat{	\bar c}_N(s) }^2_{0,2}dsdt\\
			&\leq  h \bar \be\int_0^T\int_0^{T} \abs{A_1\hat{	\bar c}_N(s) }^2_{0,2}dsdt\leq \bk_mh,
		\end{split}
	\end{equation*}
	and 
	\begin{equation*}
		\bar \be\int_0^T \abs{\int_0^tB_1(\hat{	\bar \bu}_N(s),  \hat{	\bar c}_N(s))1_{[0,h]}ds}^2_{0,2}dt+\bar \be\int_0^T \abs{\int_0^tB_2(\theta^{\eps_m}_m(\hat{	\bar n}_N(s)),\hat{	\bar c}_N(s)) 1_{[0,h]}ds}^2_{0,2}dt\leq \bk_m h.
	\end{equation*}
	By making use of the Fubini theorem, the It\^o  isometry we arrive at
	\begin{equation*}
		\begin{split}
			\bar \be\int_0^T \abs{\int_t^{t_{\ell_N(t)+1}}F^1_m(\check{	\bar c}_N)g( \check{	\bar c}_N)  d \bar \beta^N(s)}^2_{0,2}dt
			&\leq\int_0^T \bar \be  \int_t^{t_{\ell_N(t)+1}}(F^1_m(\check{	\bar c}_N(s)))^2\abs{\nabla \check{	\bar c}_N(s)}^2_{0,2}dsdt\\
			&\leq\int_0^T \bar \be  \int_{t_{\ell_N(t)}}^{t_{\ell_N(t)+1}}(F^1_m(\check{	\bar c}_N(s)))^2\abs{\nabla \check{	\bar c}_N(s)}^2_{0,2}dsdt\\
			&\leq m^2Th.
		\end{split}
	\end{equation*}
	Here, we used the fact that $\abs{\mathbf{g}_k}_{\sobm^{1,\infty}}\leq 1$,  $(F^1_m(\check{	\bar c}_N))^2\abs{\nabla \check{	\bar c}_N}^2_{0,2}\leq m^2$ (see \re{s1.4}) and  the fact that  by definition,  amongst the subdivision intervals of $[0, T ]$, $[t_{\ell_N(t)}, t_{\ell_N(t)+1}]$ is the first interval containing $t$. Analogously, we have 
	\begin{equation*}
		\bar \be\int_0^T \abs{\frac{h\wedge t}{h}\int_0^hF^1_m(\check{	\bar c}_N(s))\sum_{k=1}^3\mathbf{g}_k\cdot\nabla \check{	\bar c}_N(s)  d \bar \beta^N_k(s)}^2_{0,2}dt\leq \bk_mh.
	\end{equation*}
	By summing up the above inequality and using the fact that $h=\frac{T}{N}$, we infer that 
	\begin{equation*}
		\lim_{N\longrightarrow\infty}\bar \be\int_0^T\abs{\bar \bec_N^c(s)}_{0,2}^2ds\leq \lim_{N\longrightarrow\infty}\frac{\bk_mT}{N}=0,
	\end{equation*}
	and the second convergence on Lemma \ref{Lemmas2.12} follows.  The third convergence can be proved in a similar way  using the fact that  due to the equivalence of norms on $\bh_m$,  \re{3.3},  \re{s2.30} and \re{s2.30}$_6$ the following hold,
	\begin{equation*}
		\begin{split}
			&\bar \be\int_0^T \abs{A_1\hat{	\bar n}_N(s) }^2_{(\h^1)'}ds\leq \bar \be\int_0^T \abs{\nabla\hat{	\bar n}_N(s) }^2_{(\h^1)'}ds \leq \bk_m,\\
			& \bar \be\int_0^T \abs{B_3(\theta_m(\hat{	\bar n}_N(s)),\hat{	\bar c}_N(s))}^2_{(\h^1)'}ds\leq287(m+1)^2\bar \be\int_0^T\abs{\nabla\hat{	\bar c}_N(s)}^2_{0,2}\leq \bk_m,\\
			&\bar \be\int_0^T \abs{B_1(\hat{	\bar \bu}_N(s),  \hat{	\bar n}_N(s))}^2_{(\h^1)'}ds\leq\bk_m\left[\bar \be\sup_{0\leq s\leq T}\abs{\hat{	\bar \bu}_N(s)}^4_{0,2}\right]^\frac{1}{2}\left[\bar \be\left(\int_0^T\abs{\nabla\hat{	\bar n}_N(s)}^2_{0,2}\right)^2\right]^\frac{1}{2}\leq \bk_m.
		\end{split}
	\end{equation*}
	Thus, we omit the proof of the third convergence.
	
\end{proof}
	
	To complete this step, we need to pass to the limit in the other terms of \re{s2.35}. To this end, we will have three lemmas. The first is stated as follows.
	\begin{lemma}\label{Lemmas2.13}
		For every $\bv\in \bh_m$ and  $t\in [0,T]$, the following convergences hold $\bar{ \mathbb{P}}$-a.s.,
		\begin{equation*}
			\begin{split}
				&\lim_{N\longrightarrow\infty}(	\bar \bu_N(t),\bv)=(	\bu_m(t),\bv),\\
				&\lim_{N\longrightarrow\infty}\int_0^t(A\hat{	\bar \bu}_N(s),\bv)ds=\int_0^t(A\bu_m(s),\bv)ds,\\
				&\lim_{N\longrightarrow\infty}\int_0^t(B^m(\check{	\bar \bu}_N(s),\hat{	\bar \bu}_N(s)),\bv)ds=\int_0^t(B^m( \bu_m(s), \bu_m(s)),\bv)ds,\\
				&\lim_{N\longrightarrow\infty}\int_0^t(B_0^m(\theta_m(\hat{	\bar n}_N(s)),\phi),\bv)ds=\int_0^t(B_0^m(\theta_m( n_m(s)),\phi),\bv)ds.
			\end{split}
		\end{equation*}
		In addition, 
		\begin{equation*}
			\lim_{N\longrightarrow\infty}\int_0^t ( \pi_m( \Pl\mathbf{g}(\check{	\bar \bu}_N(s)))  ,\bv)d \bar W^N(s)=\int_0^t ( \pi_m( \Pl\mathbf{f}(\bu_m(s))) ,\bv)d \bar W(s),
		\end{equation*}
		\text{in probability in } $\elm^2(0,T;\mathbb{R})$.
	\end{lemma}
	The proof of first four convergences are now standard, see for instance \cite{Bre7}. The proof of the $5^{th}$ convergence is very similar to the proof of the convergence $Q_5$ in Lemma \ref{Lemmas2.14}, hence we omit it.
	\begin{lemma}\label{Lemmas2.14}
		For every $\psi\in \elm^2$ and $t\in [0,T]$,   the following convergences hold $\bar{ \mathbb{P}}$-a.s.,
		\begin{equation*}
			\begin{split}
				&Q_1:=\lim_{N\longrightarrow\infty}\left[(	\bar c_N(t),\psi)-(	c_m(t),\psi)\right]=0,\\
				&Q_2:=\lim_{N\longrightarrow\infty}\left[\int_0^t(A_1\hat{	\bar c}_N(s),\psi)ds-\int_0^t(A_1c_m(s),\psi)ds\right]=0,\\
				&Q_3:=\lim_{N\longrightarrow\infty}\left[\int_0^t(B_1(\hat{	\bar \bu}_N(s),  \hat{	\bar c}_N(s)),\psi)ds-\int_0^t(B_1( \bu_m(s),   c_m(s)),\psi)ds\right]=0,\\
				&Q_4:=\lim_{N\longrightarrow\infty}\left[\int_0^t(B_2(\theta^{\eps_m}_m(\hat{	\bar n}_N(s)),\hat{\bar c}_N(s)),\psi)ds-\int_0^t(B_2(\theta^{\eps_m}_m(n_m(s)), c_m(s)),\psi)ds\right]=0.
			\end{split}
		\end{equation*}
		In addition, 
		\begin{equation*}
			Q_5:=\lim_{N\longrightarrow\infty}\left[\int_0^tF^1_m(\check{	\bar c}_N(s)(g( \check{	\bar c}_N(s)) ,\psi)d \bar \beta^N(s)-\int_0^tF^1_m(c_m(s)(g(c_m(s)) ,\psi)d \bar \beta^m(s)\right]=0,
		\end{equation*}
		\text{in probability in } $\elm^2(0,T;\mathbb{R})$.
	\end{lemma}
\begin{proof}[\textbf{Proof of Lemma \ref{Lemmas2.14}}]
		 Let us fix arbitrary $\psi\in \elm^2$.   The proof of the two convergences follow from   \re{s2.29}$_2$, \re{s2.29}$_5$ and Lemma \ref{lemma 3.17}. So, we omit the details
	
	By using the H\"older inequality, we get 
	\begin{equation*}
		\begin{split}
			Q_3&\leq  \lim_{N\longrightarrow\infty}\int _0^T\valabs{\int_0^t((\hat{	\bar \bu}_N(s)-\bu_m(s))\nabla\hat{	\bar c}_N(s),\psi)ds}+\lim_{N\longrightarrow\infty}\int _0^T\valabs{\int_0^t(\bu_m(s)\nabla(\hat{	\bar c}_N(s)-c_m(s)),\psi)ds}\\
			&\leq \abs{\psi}_{0,2} \lim_{N\longrightarrow\infty}\left( \int _0^T\abs{\hat{	\bar \bu}_N(s)-\bu_m(s)}^2_{0,\infty}ds\right)^\frac{1}{2}\left(\int _0^T\abs{\nabla\hat{	\bar c}_N(s)}^2_{0,2}ds\right)^\frac{1}{2}\\
			&\qquad+\abs{\psi}_{0,2} \lim_{N\longrightarrow\infty}\left( \int _0^T\abs{\bu_m(s)}^2_{0,\infty}ds\right)^\frac{1}{2}\left( \int _0^T\abs{\nabla(\hat{	\bar c}_N(s)-c_m(s))}^2_{0,2}ds\right)^\frac{1}{2}
			=0.
		\end{split}
	\end{equation*}
	Here, we have used the equivalence of norms on $\bh_m$,  the convergence \re{s2.33} and the fact that all weakly convergence sequences in Banach space is bounded. 
	
	Since  $\valabs{\theta^{\eps_m}_m(\hat{	\bar n}_N)-\theta^{\eps_m}_m( n_m)}= \valabs{\theta_m(\hat{	\bar n}_N)-\theta_m( n_m)}\leq \br_\theta\valabs{\hat{	\bar n}_N-n_m}$ (Lipschitz property of $\theta_m$) and  $\valabs{ \theta^{\eps_m}_m(n_m)}= \valabs{ \theta_m(n_m)+\eps_m}\leq m+1+\eps_m$ (see \re{3.3}) ,  the H\"older inequality yields
	\begin{equation*}
		\begin{split}
			Q_4&\leq  \lim_{N\longrightarrow\infty}\valabs{\int_0^t([\theta^{\eps_m}_m(\hat{	\bar n}_N(s))-\theta^{\eps_m}_m( n_m(s))]\hat{	\bar c}_N(s),\psi)ds}\\
			&\qquad+\lim_{N\longrightarrow\infty} \valabs{\int_0^t(\theta^{\eps_m}_m(n_m(s))[\hat{	\bar c}_N(s)-c_m(s)],\psi)ds}\\
			&\leq \br_\theta\abs{\psi}_{0,2} \lim_{N\longrightarrow\infty}\left(\int _0^T\abs{\hat{	\bar n}_N(s)-n_m(s)}^2_{0,2}ds\right)^\frac{1}{2}\left( \int _0^T\abs{\hat{	\bar c}_N(s)}^2_{0,\infty}ds\right)^\frac{1}{2}\\
			&\qquad+(m+1+\eps_m)\abs{\psi}_{0,2} \lim_{N\longrightarrow\infty}\left( \int _0^T\abs{\hat{	\bar c}_N(s)-c_m(s)}^2_{0,2}ds\right)^\frac{1}{2}=0.
		\end{split}
	\end{equation*}
	In the last line, we used the fact that  $D(A_1)\hookrightarrow \elm^\infty$, the convergence \re{s2.33} and the boundedness due to the weak convergence.  
	
	To prove that $Q_5=0$, we note that  up to a subsequence
	$
	\abs{\hat{	\bar c}_N(s)-c_m(s)}_{1,2}\to 0$, $\text{ a.e. } s\in[0,T]. 
	$
	By the H\"older inequality and \re{s1.5}, we note that  by using  $\valabs{F^1_m(c_m)}\leq 1$ we see that 
	\begin{equation*}
		\begin{split}
			&\sum_{k=1}^3\valabs{F^1_m(\check{	\bar c}_N)(\mathbf{g}_k\cdot\nabla \check{	\bar c}_N ,\psi)-F^1_m(c_m)(\mathbf{g}_k\cdot\nabla c_m ,\psi)}\\
			&\leq \abs{\psi}_{0,2}\sum_{k=1}^3\valabs{F^1_m(\check{	\bar c}_N)\mathbf{g}_k\cdot\nabla \check{	\bar c}_N-F^1_m(c_m)\mathbf{g}_k\cdot\nabla c_m }_{0,2}\\
			&\leq \abs{\psi}_{0,2}\sum_{k=1}^3\valabs{F^1_m(\check{	\bar c}_N)-F^1_m(c_m)}\abs{\mathbf{g}_k\cdot\nabla \check{	\bar c}_N }_{0,2}+\abs{\psi}_{0,2}\sum_{k=1}^3\valabs{F^1_m(c_m)}\abs{\mathbf{g}_k\cdot\nabla \check{	\bar c}_N-\mathbf{g}_k\cdot\nabla c_m }_{0,2}\\
			&\leq \abs{\psi}_{0,2}\sum_{k=1}^3\frac{\valabs{\abs{\check{	\bar c}_N}_{1,2}-\abs{c_m}_{1,2}}}{\abs{c_m}_{1,2}+1}\sum_{k=1}^3\abs{\mathbf{g}_k}_{0,\infty}\abs{\nabla \check{	\bar c}_N }_{0,2}+\abs{\psi}_{0,2}\sum_{k=1}^3\valabs{F^1_m(c_m)}\abs{\mathbf{g}_k}_{0,\infty}\abs{\nabla \check{	\bar c}_N-\nabla c_m }_{0,2}\\
			&\leq \abs{\psi}_{0,2}\sum_{k=1}^3\abs{\check{	\bar c}_N-c_m}_{1,2}\abs{\mathbf{g}_k}_{0,\infty}\abs{\check{	\bar c}_N }_{1,2}+\abs{\psi}_{0,2}\sum_{k=1}^3\abs{\mathbf{g}_k}_{0,\infty}\abs{\check{	\bar c}_N- c_m }_{1,2}.
		\end{split}
	\end{equation*}
%	In the last line, we use the fact that 
	Since $\abs{\check{	\bar c}_N}_{1,2}\leq \abs{\check{	\bar c}_N-c_m}_{1,2}+\abs{c_m}_{1,2}$, we can pass to the limit in the above inequality  and derive that 
	\begin{equation*}
		\lim_{N\longrightarrow\infty}\valabs{F^1_m(\check{	\bar c}_N(s))(g( \check{	\bar c}_N(s)) ,\psi)-F^1_m(c_m(s))(g( c_m(s)) ,\psi)}_{\mathcal{L}_{HS}(\mathbb{R}^3;\mathbb{R})}=0, 
	\end{equation*}
	for almost every  $s\in [0,T]$.  In addition, by using the H\"older inequality and the second inequality of  \re{s1.4} we get
	\begin{equation*}
		\begin{split}
			&\int_0^T\abs{F^1_m(\check{	\bar c}_N(s))(g( \check{	\bar c}_N(s)) ,\psi)-F^1_m(c_m(s))(g( c_m(s)) ,\psi)}_{\mathcal{L}_{HS}(\mathbb{R}^3;\mathbb{R})}^4ds\\
			&\leq 8\abs{\psi}^4_{0,2}\sum_{k=1}^3\abs{\mathbf{g}_k}^4_{0,\infty}\int_0^T(F^1_m(\check{	\bar c}_N(s)))^4\abs{\nabla \check{	\bar c}_N(s)}^4_{0,2}ds\\
			&\qquad+8\abs{\psi}^4_{0,2}\sum_{k=1}^3\abs{\mathbf{g}_k}^4_{0,\infty}\int_0^T(F^1_m( c_m(s)))^4\abs{\nabla  	c_m(s)}^4_{0,2}ds\\
			&\leq 16\abs{\psi}^4_{0,2}Tm^4\sum_{k=1}^3\abs{\mathbf{g}_k}^4_{0,\infty}.
		\end{split}
	\end{equation*}
	With this uniform integrability and almost everywhere convergence in hand,  we apply the Vitali convergence theorem and  derive that 
	\begin{equation*}
		\lim_{N\longrightarrow\infty}\int_0^T\abs{F^1_m(\check{	\bar c}_N(s))(g( \check{	\bar c}_N(s)) ,\psi)-F^1_m(c_m(s))(g( c_m(s)) ,\psi)}_{\mathcal{L}_{HS}(\mathbb{R}^3;\mathbb{R})}^2ds=0,  \ \bar{ \mathbb{P}}\text{-a.s.}.
	\end{equation*}
	In particular,
	\begin{equation*}
		\lim_{N\longrightarrow\infty}F^1_m(\check{	\bar c}_N(\cdot))(g( \check{	\bar c}_N(\cdot)) ,\psi)-F^1_m(c_m(\cdot))(g( c_m(\cdot)) ,\psi)=0, \text{ in probability in } \elm^2(0,T;\mathcal{L}_{HS}(\mathbb{R}^3;\mathbb{R})).
	\end{equation*}
	Since the convergence \re{s2.29}$_4$ holds, we apply   \cite[Lemma 2.15]{Debuss} and end with the proof of Lemma \ref{Lemmas2.14}.
\end{proof}
 Now, we give the last convergence result of this step.
	\begin{lemma}\label{Lemmas2.15}
		For every $\varphi\in \h^1$ and $t\in [0,T]$, the following convergences hold $\bar{ \mathbb{P}}$-a.s.,
		\begin{equation*}
			\begin{split}
				&Q_6:=\lim_{N\longrightarrow\infty}\left[(\bar n_N(t),\varphi)-(n_m(t),\varphi)\right]=0,\\
				&Q_7:=\lim_{N\longrightarrow\infty}\left[\int_0^t\langle A_1\hat{	\bar n}_N(s),\varphi\rangle ds-\int_0^t\langle A_1n_m(s),\varphi \rangle ds\right]=0,\\
				&Q_8:=\lim_{N\longrightarrow\infty}\left[\int_0^t\langle B_1(\hat{	\bar \bu}_N(s),  \hat{	\bar n}_N(s)),\varphi\rangle ds-\int_0^t\langle B_1( \bu_m(s),   n_m(s)),\varphi\rangle ds\right]=0,\\
				&Q_9:=\lim_{N\longrightarrow\infty}\left[\int_0^t\langle B_3(\theta_m(\hat{	\bar n}_N(s)),\hat{\bar c}_N(s)),\varphi\rangle ds-\int_0^t\langle B_3(\theta_m(n_m(s)),  c_m(s)),\varphi \rangle ds\right]=0.
			\end{split}
		\end{equation*}
	\end{lemma}
	\begin{proof}[\textbf{Proof of Lemma \ref{Lemmas2.15}}]
		 Let us fix arbitrary $\varphi\in \h^1$ and $t\in [0,T]$. The first and the second  convergences  come from \ref{s2.29}.  Since $\hat{	\bar \bu}_N$ and $\bu_m$ are live in $\bh_m$, they satisfy the Dirichlet  boundary condition. So, by making use an integration-by-part, the H\"older inequality,  the equivalence of norms on $\bh_m$ and the convergence \re{s2.33}, we get  $Q_8=0$ in a quite similar way as for the proof of $Q_3=0$.
	%	\begin{equation*}
		%	\begin{split}
			%	Q_8&\leq  \lim_{N\longrightarrow\infty}\abs{\int_0^t((\hat{	\bar \bu}_N(s)-\bu_m(s))\hat{	\bar n}_N(s),\nabla\varphi)ds}+\lim_{N\longrightarrow\infty} \abs{\int_0^t(\bu_m(s)(\hat{	\bar n}_N(s)-n_m(s)),\nabla\varphi)ds}\\
			%			&\leq \abs{\nabla\varphi}_{0,2} \lim_{N\longrightarrow\infty}\left( \int _0^T\abs{\hat{	\bar \bu}_N(s)-\bu_m(s)}^2_{0,\infty}ds\right)^\frac{1}{2}\left( \int _0^T\abs{\hat{	\bar n}_N(s)}^2_{0,2}ds\right)^\frac{1}{2}\\
			%			&\qquad+ \abs{\nabla\varphi}_{0,2} \lim_{N\longrightarrow\infty}\left( \int _0^T\abs{\bu_m(s)}^2_{0,\infty}ds\right)^\frac{1}{2}\left( \int _0^T\abs{\hat{	\bar n}_N(s)-n_m(s)}^2_{0,2}ds\right)^\frac{1}{2}\\
			%			&=0.
			%		\end{split}
		%	\end{equation*}
	By using the Lipschitz property of $\theta_m$,  the fact that $\abs{ \theta_m(n_m)}\leq 17(m+1)$,   the H\"older inequality the embedding $\h^1\hookrightarrow \elm^6$, the interpolation inequality, and the convergence \re{s2.33}, we obtain
	\begin{equation*}
		\begin{split}
			Q_9&\leq  \lim_{N\longrightarrow\infty} \valabs{\int_0^t([\theta_m(\hat{	\bar n}_N(s))-\theta_m( n_m(s))]\nabla\hat{\bar c}_N(s),\nabla\varphi)ds}\\
			&\qquad+\lim_{N\longrightarrow\infty}\valabs{\int_0^t(\theta_m(n_m(s)\nabla(\hat{	\bar c}_N(s)-c_m(s)),\nabla\varphi)ds}\\
			&\leq \bk\abs{\nabla\varphi}_{0,2} \lim_{N\longrightarrow\infty}\left(  \int _0^T\abs{\hat{	\bar n}_N(s)-n_m(s)}_{0,3}\abs{\nabla\hat{\bar c}_N(s)}_{0,6}ds\right)^\frac{1}{2}\\
			&\qquad+17T^\frac{1}{2}(m+1)\bk\abs{\nabla\varphi}_{0,2} \lim_{N\longrightarrow\infty}\left(\int _0^T\abs{\nabla(\hat{	\bar c}_N(s)-c_m(s))}^2_{0,2}ds\right)^\frac{1}{2}\\
			&\leq  \bk\lim_{N\longrightarrow\infty}\left( \int _0^T\abs{\hat{	\bar n}_N(s)-n_m(s)}^2_{0,2}ds\right)^\frac{1}{4}\left( \int _0^T\abs{\hat{	\bar n}_N(s)-n_m(s)}^2_{1,2}ds\right)^\frac{1}{4}\left(\int _0^T\abs{\hat{\bar c}_N(s)}^2_{2,2}ds\right)^\frac{1}{2}\\
			&=0,
		\end{split}
	\end{equation*}
	which completes  the proof of Lemma \ref{Lemmas2.15}. 
	\end{proof}
We can now complete the proof of the Proposition  \ref{lem3.2}. 	In fact,		
	by using Lemma \ref{Lemmas2.12}, Lemma \ref{Lemmas2.13}, Lemma \ref{Lemmas2.14} and Lemma \ref{Lemmas2.15}, we derive that 
	$$(\bar {\Omega},\bar{\mathcal{F}},\{\bar {\mathcal{F}}^m_t\}_{t\in [0,T]},\bar{ \mathbb{P}}, (\bu_m,c_m,n_m), (\bar W^m, \bar \beta^m))$$
	is a martingale solution of the globally modified  semi-discretization system \re{3.2} in the sense of Definition \ref{defi3.2}. The regularity and integrability stated in Proposition \ref{lem3.2} follows from  the embeddings $C([0,T],\h^1_{weak})\subset \elm^\infty(0,T;\h^1)$ and $C([0,T],\elm^2_{weak})\subset \elm^\infty(0,T;\elm^2)$, and  \re{s2.31} and \re{s2.29}.

\vspace{0.5cm}

%\subsection*{Data availability} No data was used for the research described in the article.

%\section*{Declarations}

\subsection*{Funding} Boris Jidjou Moghomye and Erika Hausenblas have been supported by  the Austrian Science Fund (FWF) under project number PAT 3076923, Grant DOI: 10.55776/PAT3076923. 

%\subsection*{Conflict of interest} Beyond the above-mentioned financial support, the authors certify that there are no conflict of interest for this work.

\subsection*{Acknowledgments}
We acknowledge financial support provided by the Austrian Science Fund (FWF). In particular, Boris Jidjou Moghomye  and partially Erika Hausenblas   were supported by the Austrian Science Fund, project no.: PAT 3076923, Grant DOI: 10.55776/PAT3076923. This work was initiated during the visit research of Boris Jidjou Moghomye to the  Dublin City University during  July 2024. He would like to thanks the School of Mathematical Science for hospitality.

\begin{center}
	
\end{center}

\begin{thebibliography}{99}\addcontentsline{toc}{chapter}{References}
		
%		\bibitem{Adam}
%	R. A.	 Adams,  Sobolev Spaces. \textit{Academic Press}, New York-London, 1975.
		
		
		\bibitem{Banas1}
		L. Ba$\check{n}$as, Z. Brzez\'niak,   M. Neklyudov,    A. Prohl,  Stochastic ferromagnetism: Analysis and
		Computation. De Gruyter Studies in Mathematics, Berlin, 2014.
		
		\bibitem{Banas}
		L. Ba$\check{n}$as, Z. Brzez\'niak,   M. Neklyudov,    A. Prohl,   A convergent finite-element-based discretization of the stochastic Landau–Lifshitz–Gilbert equation. \textit{IMA J. Numer. Anal.} \textbf{34} (2014)   502–549.
		
		
		%\bibitem{Barbu}
		%V. Barbu, G. Da Prato and M. R\"ockner, Stochastic porous media equations, Springer, Switzerland 2016.
		
		\bibitem{Bensoussan}
	 A.	 Bensoussan,  R. Temam,   Equations stochastiques du type Navier–Stokes.  \textit{J. Funct. Anal.} \textbf{13} (1973) 195–222.
		
		\bibitem{Billing}
	P.	 Billingsley,  Convergence of probability measures. \textit{John Wiley and Sons}, Chicago, 2013.
		
		
\bibitem{Blanchet}		
 A. Blanchet, J. A. Carrillo, N. Masmoudi.  Infinite time aggregation for the critical Patlak-Keller-Segel Model in $\mathbb{R}^2$. \textit{Comm. on Pure and Appl. Math.} \textbf{LXI} (2008) 1449-1481. 
		
	%	\bibitem{Bre1}
	%	 H. Brezis, Functional analysis, Sobolev spaces and partial differential equations. \textit{Springer,} New York,  2010.
		
	%	\bibitem{Bre5}
	%	Z. Brze\'zniak, Stochastic partial differential equations in M-type 2 Banach spaces, Potential Anal., \textbf{4} 145, 1995.
		
		
		
	%	\bibitem{Bre6}
	%Z.	 Brze\'zniak,  and B. Ferrario,   A note on stochastic Navier–Stokes equations with not regular multiplicative noise.  \textit{Stoch. PDE Anal Comput.} \textbf{5} (2017)  53–80.
		
		
		
		
	%	\bibitem{Bre2}
	%	Z. Brze\'zniak,   B. Goldys,   T. Jegaraj,   Large Deviations and Transitions Between Equilibria for Stochastic Landau–Lifshitz–Gilbert Equation.  \textit{Arch. Rational Mech. Anal.} \textbf{226} (2017)  497–558.
		
	%	\bibitem{Brez2}
	%	Z. Brze\'zniak,   F. Hornung,     U. Manna,  Weak martingale solutions for the stochastic nonlinear Schrödinger equation driven by pure jump noise.  \textit{Stoch PDE Anal. Comp.}  \textbf{8} (2020) 1–53.
		
		
		\bibitem{Bre7}
		Z. Brze\'zniak,    E.   Motyl,   Existence of a martingale solution of the stochastic Navier–Stokes equations in unbounded 2D and 3D domains. \textit{J. Differ. Equ.} \textbf{254} (2013) 1627–1685.
		
	%	\bibitem{Bre3}
	%	 Z. Brze\'zniak,  E.  Motyl,  M.   Ondrejat,  Invariant measure for the stochastic Navier-Stokes equations in unbounded 2D domains.  \textit{The Annals of Probability } \textbf{45} (2017) 3145–3201.
		
		
		
		
		%\bibitem{Bre4}
		%Z. Brze\'zniak,  J. Zhu, E. Hausenblas, Maximal inequalities for stochastic convolutions driven by compensated Poisson random measures in Banach spaces, Ann. Inst. Henri Poincaré Probab. Stat., \textbf{53} 937–956, 2017.
		
	%	\bibitem{Buc}
	%	T. Buckmaster and  V. Vicol, Nonuniqueness of weak solutions to the Navier-Stokes equation, Ann. Math., \textbf{189}  101–144, 2019.
		
		\bibitem{Caraballo}
		 T. Caraballo,  J. Real,  P. E. Kloeden, Unique strong solutions and V-attractors of a three dimensional system of globally modified Navier–Stokes equations. \textit{Adv. Nonlinear Stud.} \textbf{6} (2006)  411–436.
		
	%	\bibitem{Carlen}
	%	E. A. Carlen, A. Figalli, Stability for a GNS inequality and the LOG-HLS inequality, with application to the critical mass Keller-Segel equation. \textit{Duke Mathematical Journal} \textbf{162} (2013) 579-625.
		
%	\bibitem{Carriollo}	
 %J. A. Carriollo, H. Murakawa,  M. Sato, H. Togashi,  O. Trush,  A population dynamics model of cell adhesion incorporating population pressure and density saturation. \textit{Journal of Theo. Bio.} \textbf{474} (2019) 14-124. 	
	
		
	%	\bibitem{Chae}
	%	M. Chae,  K. Kang, J. Lee,  Existence of smooth solutions to coupled Chemotaxis-fluid equations. \textit{Discrete Contin. Dyn. Syst.} \textbf{33} (2013)  2271–2297.
		
	%	\bibitem{Da}
	%	G. Da Prato,   J. Zabczyk, Stochastic Equations in Infinite Dimensions, second edition. \textit{Encyclopedia of mathematics and its applications, Vol  152, Cambridge University Press,} Cambridge, 1992.
		
	%	\bibitem{Dareiotis}
	%	K. Dareiotis, B. Gess, M. V. Gnann, and  G. Gr\"un, Non-negative Martingale Solutions to the Stochastic Thin-Film Equation with Nonlinear Gradient Noise,  Arch. Rational Mech. Anal., \textbf{242}  179–234, 2021.
		
		\bibitem{Debuss}
	A.	 Debussche,  N. Glatt-Holtz,  R. Temam,  Local martingale and pathwise solutions for an abstract fluids model. \textit{Physica D} \textbf{240} (2011)  1123–1144. 
\bibitem{De}
A. De Bouard, A. Debussche, A semi-discrete scheme for the stochastic nonlinear Schr\"odinger equation. \textit{Numer. Math.} \textbf{96} (2004) 733-770.
		
	%	\bibitem{Deugoue}
	%	G. Deugoue,   B. Jidjou Moghomye,  T. Tachim Medjo, Fully discrete finite element approximation of the stochastic Cahn-Hilliard-Navier-Stokes system. \textit{IMA J. Numer. Anal.} \textbf{4} (2020) 3046-3112.
		
%		\bibitem{Dha}
%	G.	Dhariwal,   A. J\"ungel,  N. Zamponi,  Global martingale solutions for a stochastic population cross-diffusion system. \textit{Stochastic Process. Appl.} \textbf{129} (2019) 3792-3820.
		
		
		
%		\bibitem{Duan1}
%	J.	Duan, A. Millet,  Large deviations for the Boussinesq equations under random influences. \textit{Stochastic Process. Appl.} \textbf{119} (2009)  2052–2082.
		
%		\bibitem{Duan}
%	J.	 Duan,  W.  Wang, Effective dynamics of stochastic partial differential equations. \textit{Elsevier insights}, Amsterdam, 2014.
		
%		\bibitem{Duan2}
%	 R. J.	Duan,  A. Lorz,    P. A. Markowich,  Global solutions to the coupled chemotaxis-fluid equations. \textit{Comm. Part. Diff.		Equ.} \textbf{35} (2010)  1635–1673.
		
		\bibitem{Flan}
	 F.	Flandoli,   D. Gatarek,   Martingale and stationary solutions for stochastic Navier-Stokes equations. \textit{Probab. Theory Related Fields} \textbf{102} (1995)  367-391.
		
		%\bibitem{Gal}
		%C. G. Gal,  M. Grasselli and A. Miranville, Cahn-Hilliard-Navier-Stokes systems with moving contact lines, Calc. Var. \textbf{5550}, 2016.
		
		\bibitem{Glatt}
	N.	Glatt-Holtz,  R. Temam,    C. Wang, Time discrete approximation of weak solutions to
		stochastic equations of geophysical fluid dynamics and applications. \textit{Chin. Ann. Math.} \textbf{38B} (2017)  425–472.
		
		\bibitem{Granas}
		A. Granas,   J.  Dugundji, Fixed Point Theory. \textit{Springer}, New York, 2003.
		
		
		\bibitem{Grisvard}
	 P.	Grisvard,  Elliptic problems in nonsmooth domains. \textit{Pitman}, Boston, 1995.
		
		
%		\bibitem{Giorgini}
%	A.	Giorgini,  R. Temam, V. Xuan-Truong,  The Navier-Stokes Cahn-Hilliard equations for mildlycompressible binary fluid mixtures. \textit{Disc. cont. Dynam. Sys. Series B} \textbf{26} (2021) 337-366.
		
%		\bibitem{Haus}
%	E.	Hausenblas,  B. Jidjou Moghomye,   P. A.  Razafimandimby,  Global weak-martingale solutions of the 3D chemotaxis-Navier-Stokes system driven by transport noise, \textit{In preparation},

%\bibitem{Herrero}
%M. A. Herrero, J. J. L. Vel\'azquez, Singularity patterns in a chemotaxis model. \textit{Mathematische Annalen} \textbf{306} (1996) 583-623.

%\bibitem{Herrero1}
%M. A. Herrero,  J. J. L. Vel\'azquez, A blow-up mechanism for a chemotaxis model. \textit{Annali della Scuola Normale Superiore di Pisa} \textbf{24} (1997) 633-683.
		
	%	\bibitem{Hof}
	%	M. Hofmanová, R. Zhu and  X. Zhu, Global-in-time probabilistically strong and Markov solutions to stochastic 3D Navier–Stokes equations existence and non-uniqueness, Ann. Probab.,  \textbf{51} 524–579, 2023.
		
		
		%\bibitem{He}
		%H. He and Q. Zhang, Global existence of weak solutions for the 3D chemotaxis-Navier-Stokes equations,  Nonlinear Analysis Real World Applications, \textbf{35}  336-349, 2017.
		
		\bibitem{Ichikawa}
	A.	Ichikawa,  Stability of semilinear stochastic evolution equations. \textit{J. Math. Anal. Appl.} \textbf{1} (1982) 12-44.
		
		
		\bibitem{Jakubowski}
	 A.	Jakubowski, The almost sure Skorokhod representation for subsequences in nonmetric spaces. \textit{Theory Probab. Appl.} \textbf{42} (1998)   167–174.
		
%		\bibitem{Jager}
%		W. J\"ager, S. Luckhaus, On explosions of solutions to a system of partial differential equations modelling chemotaxis. \textit{Trans. of the Amer. Math. Soci.} \textbf{329} (1992) 819-824.
		
	%	\bibitem{Kallenberg}
	%	O. Kallenberg, Foundations of Modern Probability Probability and its Applications, Springer-Verlag, NewYork, 1997.
		
%		\bibitem{Krylov}
%	N. V.	Krylov, B. L.  Rozovskii,  Stochastic evolution equations. \textit{J. Sov. Math.} \textbf{16} (1981) 1233-1277.
		
%		\bibitem{Lorz}
%	A. 	Lorz,  Coupled chemotaxis fluid model. \textit{Math. Models Methods Appl. Sci.} \textbf{20} (2010)  987–1004.
		
%		\bibitem{Lorz1}
%	 A. 	Lorz, A coupled Keller–Segel-Stokes model global existence for small initial data and blow-up delay. \textit{Commun. Math. Sci.} \textbf{10} (2012)  555–574.
		
%	\bibitem{Maini}
%P. K.	Maini,  M. R. Myerscough, K. H. Winters,  J. D. Murrat,  Bifurcating spatially heterogeneous solutions in a chemotaxis model for biological pattern generation. \textit{Bulletin of Math. Bio.} \textbf{43} (1991) 701-719.
		
%		\bibitem{Motyl}
%	E. 	 Motyl,  Stochastic magneto-hydrodynamic equations (MHD) Invariant measures in 2D Poincar\'e domains. \textit{J. Math. Anal. Appl.} \textbf{514} (2022) 126317.


%\bibitem{Nagai}
%T. Nagai, Global existence of solutions to a parabolic system for chemotaxis in two space dimensions. \textit{Nonlinear Analysis} \textbf{30} (1997) 5381-5388.

%\bibitem{Painted}
 %K. J. Painter, Continuous models for cell migration in tissues and applications to cell sorting via differential chemotaxis. \textit{Bulletin of Math. Bio.} \textbf{71} (2009) 1117-1147.
		
%		\bibitem{Pazy}
%	A.	Pazy,   Semigroups of Linear Operators and Applications to Partial Differential Equations. \textit{Applied Mathematical		Sciences, vol. 44,} Springer-Verlag, New York, 1983.
		
		
	%	\bibitem{Pre}
	%	C. Pr\'ev\^ot, M. R\"ockner, A Concise Course on Stochastic Partial Differential Equations, in Lecture Notes Math., vol. 1905, Springer, Berlin, 2007.
		
		\bibitem{Razafimandimby}
	P.  A.	 Razafimandimby,  Grade-two fluids on non-smooth domain driven by multiplicative noise Existence, uniqueness and regularity.  \textit{J. Differential Equations}, \textbf{263} (2017) 3027–3089.
		
%		\bibitem{Raza}
%	P. A.	Razafimandimby, M. Sango, Weak Solutions of a Stochastic Model for		Two-Dimensional Second Grade Fluids. \textit{Boundary Value Problems} \textbf{2010} (2010)  1-40.
		
		\bibitem{Raza1}
	P. A.	 Razafimandimby,  M. Sango,  Existence and large time behavior for a stochastic model ofmodified magnetohydrodynamic equations. \textit{Z. Angew. Math. Phys.} \textbf{66} (2015) 2197-2235.
		
		
		\bibitem{Sim}
	J.	Simon,  Compact sets in the space $\elm^p(0,T;B)$. \textit{Ann. Mat. Pure Appl.} \textbf{146} (1987)  65-96.
		
		
		\bibitem{Sim1}
	J. 	 Simon,  Sobolcv, Besov and Nikolskii fractional Spaces Imbeddings and comparisons for vector valued spaces on an interval. \textit{Ann. Mat. Pura Appl.}  \textbf{157} (1990)  117-148.
		
%		\bibitem{Sritharan}
%	S. S.	 Sritharan,  P. Sundar,  Large deviations for the two-dimensional Navier–Stokes equations with multiplicative noise. \textit{Stochastic Process. Appl.} \textbf{116} (2006) 1636–1659.
		
		\bibitem{Taylor}
	 M. E.	Taylor, Partial Differential Equation I Basic Theory, Second ed. , \textit{ Applied Mathematical sciences, Vol 115, Springer}, Carolina, 2011. 
		
%		\bibitem{Temam}
%		R. Temam,  Navier-Stokes Equations. \textit{North-Holland}, Amsterdam, 1979.
		
%		\bibitem{Tuv}
%	I.	Tuval, L. Cisneros,  C. Dombrowski,   C. W. Wolgemuth,  J. O.  Kessler, R. E. Goldstein,  Bacterial swimming and oxygen transport near contact lines. \textit{Proc. Natl. Acad. Sci.} \textbf{102} (2005)  2277-2282.
		
		
		
		\bibitem{Vakhania}
		N. N. Vakhania,  V. I. Tarieladze,  S. A.  Chobanyan,  Probability Distributions on Banach spaces, \textit{D. Reidel Publishing Company}, Holland,  1987.
		
	%	\bibitem{Van}
	%	A. W. van der Vaart and J. A. Wellner, Weak Convergence and Empirical Processes, New York, Springer, 1996.
		
%		\bibitem{Filho}
%		A. L. C., Vianna Filho,    F.  Guil\'en-Gonz\'alez, Uniform in time solutions for a chemotaxis with potential consumption model.  \textit{Nonlinear Analysis Real World Applications}  \textbf{70} (2023)  103795.  
		
%		\bibitem{Winkler}
%	M. 	 Winkler, Global large-data solutions in a chemotaxis-(Navier–)Stokes system modeling cellular swimming in fluid drops. \textit{Comm. Partial Differential Equations} \textbf{37} (2012)  319–35.%\usepackage{xargs}  
		%\newcommandx{\unsure}[2][1=]{\todo[linecolor=red,backgroundcolor=red!25,bordercolor=red,#1]{#2}}
		%\newcommandx{\change}[2][1=]{\todo[linecolor=blue,backgroundcolor=blue!25,bordercolor=blue,#1]{#2}}
		%\newcommandx{\info}[2][1=]{\todo[linecolor=OliveGreen,backgroundcolor=OliveGreen!25,bordercolor=OliveGreen,#1]{#2}}
		%\newcommandx{\improvement}[2][1=]{\todo[linecolor=Plum,backgroundcolor=Plum!25,bordercolor=Plum,#1]{#2}}
		%\newcommandx{\thiswillnotshow}[2][1=]{\todo[disable,#1]{#2}}2, 2012.
		
%		\bibitem{Winkler1}
%		 M. Winkler, Global weak solutions in a three-dimensional chemotaxis-Navier–Stokes system. \textit{Ann. Inst. H. Poincar\'e Anal. Non Lin\'eaire} \textbf{33} (2016) 1329-1352.
		
%		\bibitem{Winkler2}
%	 M.	Winkler, Stabilization in a two-dimensional chemotaxis-Navier–Stokes system. \textit{Arch. Ration. Mech. Anal.} \textbf{211} (2014)  455–487.
		
		\bibitem{Zhai}
	J.	 Zhai,   T. Zhang,  2D stochastic Chemotaxis-Navier-Stokes system, \textit{J. Math. Pures Appl.} \textbf{138} (2020)  307-355.
		
		\bibitem{Zhang1}
		L. Zhang,   B. Liu,  Global martingale weak solution for the three-dimensional stochastic chemotatix-Navier-Stokes system with L\'evy processes. \textit{J. Funct. Anal.} \textbf{286} (2024) 110337.
		%\usepackage{xargs}  
		%\newcommandx{\unsure}[2][1=]{\todo[linecolor=red,backgroundcolor=red!25,bordercolor=red,#1]{#2}}
		%\newcommandx{\change}[2][1=]{\todo[linecolor=blue,backgroundcolor=blue!25,bordercolor=blue,#1]{#2}}
		%\newcommandx{\info}[2][1=]{\todo[linecolor=OliveGreen,backgroundcolor=OliveGreen!25,bordercolor=OliveGreen,#1]{#2}}
		%\newcommandx{\improvement}[2][1=]{\todo[linecolor=Plum,backgroundcolor=Plum!25,bordercolor=Plum,#1]{#2}}
		%\newcommandx{\thiswillnotshow}[2][1=]{\todo[disable,#1]{#2}}
%		\bibitem{Zhang}
%		Q. Zhang,   X. Zheng, Global well-posedness for the two-dimensional incompressible chemotaxis-Navier–Stokes equations. \textit{SIAM J. Math. Anal.} \textbf{46} (2014)  3078–3105.
		
%		\bibitem{Zhu}
%		J. Zhu,   Z. Brzeźniak, Maximal inequalities and exponential estimates for stochastic convolutions driven by lévy-type processes in Banach spaces with application to stochastic quasi-geostrophic equation. \textit{SIAM J. Math. Anal.} \textbf{51}  (2019) 2121–2167.
		
		
		
	\end{thebibliography}
\end{document}